\documentclass[reqno]{amsart}
\usepackage[top=1.1in, bottom=1in, left=1in, right=1in]{geometry}
\usepackage{amsfonts}
\usepackage{amssymb} 
\usepackage{amsthm}
\usepackage{amsmath}
\usepackage{mathrsfs}
\usepackage{color}
\usepackage{enumerate}
\usepackage[numbers,sort&compress]{natbib}
\usepackage{pdfsync}
\usepackage{esint}
\usepackage{graphicx}
\usepackage{float}
\usepackage{caption}
\usepackage{subfigure}
\usepackage{stmaryrd}
\usepackage{yhmath}
\usepackage[titletoc]{appendix}

\allowdisplaybreaks

\usepackage{tikz} 
\usetikzlibrary{calc}
\usetikzlibrary{intersections}
\usetikzlibrary{patterns}
\usetikzlibrary{positioning, shapes.geometric}

\usepackage[colorlinks,
linkcolor=black,      
anchorcolor=blue,      
citecolor=blue,      
]{hyperref}

\makeatother

\numberwithin{equation}{section}

\newtheorem{theorem}{Theorem}[section]
\newtheorem{definition}[theorem]{Definition}

\newtheorem{lemma}[theorem]{Lemma}
\newtheorem{remark}[theorem]{Remark}

\newtheorem{proposition}[theorem]{Proposition}
\newtheorem{corollary}[theorem]{Corollary}
\numberwithin{equation}{section}

\newcommand*{\Id}{\ensuremath{\mathrm{Id}}}
\newcommand*{\supp}{\ensuremath{\mathrm{supp\,}}}

\newcommand{\T}{{\mathbb{T}}}

\newcommand{\rr}{\mathring{R}}
\newcommand{\ru}{\mathring{R}_q}

\newcommand{\p}{\partial}
\renewcommand{\P}{\mathbb{P}}
\renewcommand{\div}{{\mathrm{div}}}
\newcommand{\curl}{{\mathrm{curl}}}
\renewcommand{\u}{{u_q}}

\renewcommand{\d}{{\rm d}}

\newcommand{\norm}[1]{\lVert#1\rVert}
\newcommand{\abs}[1]{|#1|}

\newcommand{\la}{\lambda_q}
\newcommand{\laq}{\lambda_{q+1}}

\newcommand{\rs}{r_{\perp}}

\newcommand{\R}{{\mathbb R}}

\def\a{{\alpha}}

\def\lbb{\lambda}
\def\wt{\widetilde}
\def\9{{\infty}}
\def\ve{{\varepsilon}}
\def\na{{\nabla}}

\def\({\left(}
\def\){\right)}
\def\tq{{\theta_q}}
\def\tq1{{\theta_{q+1}}}

\def\cq{\chi_{q+1}}
\def\ut{\underline{T}}
\def\ot{\overline{T}}

\begin{document}
	
\title[] {Admissible solutions of the Euler equations}

\author{Alexey Cheskidov}
\address{School of Sciences, Westlake University, China.}
\email[Alexey Cheskidov]{cheskidov@westlake.edu.cn}
\thanks{}

\author{Zirong Zeng}
\address{School of Mathematics and Key Laboratory of MIIT, Nanjing University of Aeronautics and Astronautics, China.}
\email[Zirong Zeng]{beckzzr@nuaa.edu.cn}
\thanks{}

\author{Deng Zhang}
\address{School of Mathematical Sciences, CMA-Shanghai, Shanghai Jiao Tong University, China.}
\email[Deng Zhang]{dzhang@sjtu.edu.cn}
\thanks{}

\keywords{Euler equations; Convex integration; Admissible solutions; Weak-strong uniqueness.}

\subjclass[2020]{35A02,\ 35D30,\ 35Q31.}

\begin{abstract}

In this paper, we prove the existence of an extensive class of initial data--characterized by intermittency dimension and roughness--for which the $d$-dimensional Euler equations ($d \ge 2$) admit a decreasing energy channel with infinitely many admissible solutions arising from the same initial data. This class includes initial data in $C^{1-\kappa}(\mathbb{T}^d)$ and $W^{1,p}(\mathbb{T}^d)$ for any $0<\kappa < 1$ and  $2< p < \infty$, thereby demonstrating the sharpness of the Lions and Yudovich uniqueness principles. Our proof is based on a new backward convex integration scheme employing approximate $\Lambda$-Euler equations to reduce the initial error and on uncovering a quantitative relationship between regularity, energy, and intermittency.
\end{abstract}

\maketitle

{\small
	\tableofcontents
}
\section{Introduction}

\subsection{Sharpness of weak-strong uniqueness}

We consider the $d$-dimensional Euler equations 
\begin{equation}\label{equa-euler}
\left\{\aligned
& \partial_t u +(u\cdot \nabla )u  + \nabla p = 0, \\
& \div\, u = 0, \\
& u(0)=v_0,
\endaligned
\right.
\end{equation}
where $d \geq 2$, 
$u:[0,\9)\times\T^d\rightarrow \R^d$ stands for the velocity field, 
and
$p:[0,\9)\times\T^d\rightarrow \R$ is the pressure of the fluid.  

The Euler equations are a fundamental model in continuum mechanics, governing the motion of inviscid fluids. Despite their physical importance, their rigorous mathematical analysis remains highly challenging—especially in three dimensions. Classical local well-posedness is established for sufficiently smooth initial data, e.g., in $C^{1,\alpha}$ spaces. However, during the past few decades, numerous constructions of fluid solutions exhibiting pathological or ``wild'' behavior have emerged.
These include the loss of energy balance, non-uniqueness, singularity formation, and anomalous dissipation. From the mathematical point of view such examples provide counterexamples to well-posedness, while from the physical point of view they are closely related to turbulence and the energy cascade conjectured by Kolmogorov, observed experimentally and numerically yet long elusive analytically.

The lack of well-posedness at low regularity is intrinsically linked to the transition from laminar (smooth, predictable) to turbulent (chaotic, irregular) flow regimes. At this level, weak solutions, defined in the distributional sense, typically with $L^2$ regularity, provide a natural framework for studying turbulent flows. However, weak solutions suffer from a fundamental drawback: \emph{non-uniqueness}. This motivates the search for additional admissibility criteria to select physically meaningful solutions.

The non-uniqueness of weak solutions for the Euler equations, especially in low regularity spaces, has attracted significant interest in literature. Pioneering constructions by Scheffer \cite{Scheffer93} and Shnirelman \cite{Shnirelman97} show the existence of nontrivial weak solutions with compact temporal support for 2D incompressible Euler. 
More recently, in a major breakthrough, De Lellis and Szk\'elyhidi \cite{dls09} introduced the convex integration framework, improving these results. This approach led to striking developments such as 
the non-uniqueness of weak solutions satisfying energy inequalities for 
Euler systems \cite{dls10, DS2017, SW2012}, the failure of uniqueness for entropy solutions in compressible isentropic flow \cite{dls10, cdk15}, the construction of $C^{1/3-}_{t,x}$ wild solutions of the Euler equations by Isett \cite{I18}, and dissipative solutions by Buckmaster-De Lellis-Sz\'{e}kelyhidi, Jr.-Vicol \cite{bdsv19}. We also refer to the work by Giri and Radu \cite{gr24} on the Onsager conjecture for the 2D Euler equations, and non-uniqueness of $C^{1/2-}$ solutions of the 3D forced Euler equations \cite{bhp23} by Bulut, Huynh, and Palasek. Recently, Giri-Kwon-Novack \cite{gkn26} proved the $L^3$-based  
strong Onsager theorem, in which weak solutions to the 3D Euler equations are constructed that dissipate total kinetic energy, satisfy the local energy inequality, and belong to the class $C_t(W^{1/3-,3}_x\cap L^{\infty-}_x)$.

On the positive side of solvability theory for fluid equations, 
the study of weak-strong uniqueness has a rich history. 
Foundational contributions include criteria by Leray, Prodi, and Serrin for Navier-Stokes \cite{leray1934, Prodi59, Serrin63} and Dafermos for conservation laws \cite{Dafermos10}. 
For Euler equations in dimensions  $d\geq 2$, 
Lions \cite{lions69} proved that if a classical (``strong") solution exists (typically in $L^1_tC^1_x$), 
then any {\it dissipative weak solution} sharing the same initial data must coincide with the classical solution.   
It  was further pointed out by De~Lellis and Sz\'{e}kelyhidi \cite{dls10} that any {\it admissible solution}, 
that is, weak solution in $C_tL^2_x$ satisfying the weak energy inequality 
\begin{align}\label{w-en-inq}
    \|u(t)\|_{L^2}^2\leq \|u(0)\|_{L^2}^2\quad \text{for a.e.}\quad t\in [0,T],
\end{align}
is a dissipative solution. 
Thus, according to the Lions weak-strong uniqueness condition, the following holds.

\begin{theorem}[Lions' weak-strong uniqueness, \cite{dls10,lions69}]\label{thm-weak-strong} 
Let $d\geq 2$. 
   If there exists a solution $$v\in C([0,T];L^2(\T^d))$$ of the Euler equations \eqref{equa-euler} satisfying $\nabla v+ \nabla v^T\in L^1([0,T];L^\infty(\T^d))$, then any admissible solution $u$ of \eqref{equa-euler} is equal to $v $ on $[0,T]\times \T^d$.
\end{theorem}

For the 2D Euler equations, another well-known uniqueness principle was 
proved by Yudovich  \cite{Yudo63} in terms of the boundedness of vorticity. 

\begin{theorem}[Yudovich's uniqueness principle, \cite{Yudo63}]\label{wsu-Yudo}
Let $v_0 \in L^2(\mathbb{T}^2)$ with $\omega_0:=\curl v_0 \in L^{\infty}(\mathbb{T}^2)$ and $\div v_0=0$. Then, there exists a unique global solution $v \in C([0,T];L^2(\T^2))$ with $\omega:=\curl v\in L^{\infty} ([0,T];L^\infty(\T^2))$ to the Euler equations \eqref{equa-euler}.
\end{theorem}

Recently, Mengual \cite{M24} considered a truncated power-law vorticity initial data for which he showed the non-uniqueness of bounded admissible solutions to the 2D Euler equations in 
$C_tL^2(\R^2)\cap L_t^\infty L^\infty(\R^2)$. This result shows the sharpness of Lions weak–strong
uniqueness principle, as well as sharpness of Yudovich’s proof of uniqueness in the class
of bounded admissible solutions in the 2D case. The proof in \cite{M24} relies primarily on constructing a suitable self-similar subsolution with the initial data given by \eqref{mi} below.

\medskip 
The aim of this paper is to introduce a very rich class of initial data that give rise to non-unique admissible solutions of the Euler equations in any dimension $d\geq 2$. In particular, our class contains rather general initial data with almost constant intermittency on the borderline of the Lions weak-strong uniqueness principle and the Yudovich's uniqueness principle. Also, our solutions have decreasing energy.

\medskip 
Our first main result addresses the sharpness of the Lions weak-strong uniqueness in every dimension. 
It holds for a rather general class of initial data  in any dimension $d\geq 2$. This class  can be characterized by the 
{\it intermittency dimension} and {\it roughness condition}. 

We define stability under radial multipliers as follows:
\begin{definition}
\label{def-radial-stability}
A divergence-free field $v_0\in L^2(\T^d)$ is said to be
\emph{stable under radial multipliers} if, for every bounded
real-valued radial symbol
\[
\rho:\mathbb Z^d\to\mathbb R, \qquad
\rho(\xi)=\rho(|\xi|),
\]
the field
\[
\mathcal T_\rho v_0(x) := \sum_{\xi\in\mathbb Z^d}
\rho(\xi)\widehat v_0(\xi)e^{i2\pi x\cdot\xi}
\]
is a stationary weak solution of the Euler equation.
\end{definition}

\begin{theorem} [Sharpness of the Lions weak-strong uniqueness] \label{thm-Nonuniq} 
Let $d\geq 2$, $  0< \kappa< 1 $, and  
\[
v_0 \in C^{1-\kappa}(\T^d)
\]
be stable under radial multipliers in
the sense of Definition~\ref{def-radial-stability}. If $v_0$ has intermittency and sufficient roughness, i.e., satisfying {\rm (H1)} and {\rm (H2)} below,  then there exist infinitely many distinct dissipative solutions of the Euler equations \eqref{equa-euler} with the initial data $v_0$. 
\end{theorem}

Theorem \ref{thm-Nonuniq} shows that the Lions weak-strong uniqueness principle fails if the $C^1_x$ regularity is weakened to $C^{1-\kappa}_x$ for any $\kappa\in (0,1)$, and hence establishes the sharpness of the principle.

\medskip 
Our second result is related to 
the sharpness of the Yudovich uniqueness criterion.

\begin{theorem} [Sharpness of the Yudovich uniqueness principle] \label{thm-Nonuniq-2}
Let $d\geq 2$, $2< p<\infty$, and
\[
v_0\in W^{1,p}(\T^d)
\]
be stable under radial multipliers in
the sense of Definition~\ref{def-radial-stability}. If $v_0$ has intermittency and roughness, i.e., satisfying  {\rm (H1)'} and {\rm (H2)'} below, then there exist infinitely many distinct dissipative solutions of the Euler equations \eqref{equa-euler} with the initial data $v_0$. 
\end{theorem}

\begin{remark} 
We note that 
condition \eqref{asp-reg-2} in {\rm (H1)'} excludes the homogeneous case where the intermittency dimension $D_{n,p,\infty} \equiv d$ (no intermittency), 
which is consistent with the Yudovich's uniqueness principle.  
Indeed, if $D_{n,p,\infty} \equiv d$, then Bernstein identity \eqref{ber-ineq} yields that $\|\Delta_n v_0\|_{W^{1,p}_x}\simeq\|\Delta_n v_0\|_{W^{1,\infty}_x}$ for all $1\leq p<\infty$ and $n$, which is consistent with uniqueness expectations in the Yudovich class. 
\end{remark}

The proof of Theorems~\ref{thm-Nonuniq} and~\ref{thm-Nonuniq-2} proceeds in two parts.
In the first part, we prove Theorem~\ref{Thm-Non-euler-Local}, which shows that there exist infinitely many distinct weak solutions to the Euler equations \eqref{equa-euler} with the same initial data, whose energy decreases over a local time interval. The proof of Theorem~\ref{Thm-Non-euler-Local} is given in Sections~\ref{Sec-Lam-Euler}--\ref{Sec-prf-thm}.
Based on Theorem~\ref{Thm-Non-euler-Local}, in Sections~\ref{Sec-global-diss}--\ref{Sec-sf-closure} we extend the local result to a global one, see Theorems~\ref{Thm-Non-euler}, and thereby complete the proof of Theorems~\ref{thm-Nonuniq} and~\ref{thm-Nonuniq-2}.

\subsection{General result: weak-strong uniqueness and intermittency}

Intermittency describes the spatial non-uniformity of energy transfer in fully developed turbulence. Physically, intermittent turbulence deviates from Kolmogorov's 1941 theory—which assumes a self-similar, space-filling energy cascade—by concentrating the energy transfer into sparse, intense fluid structures (``active eddies"). Mathematically, intermittency is quantified via the intermittency dimension $D\in [0,d]$ at scale $\ell$, such that
\begin{align*}
    \text{Number of active eddies}\sim \left(\frac{L}{\ell}\right)^D,
\end{align*}
where $L$ stands for the size of the domain. Here, active eddies are fluid structures of size $\ell$ (within the filtered velocity field) that participate in the energy transfer to the next scale $\ell/2$. For further physical interpretations, we refer to \cite{CS2023}.

Intermittency can also be characterized in terms of the level of saturation of the Bernstein inequality. 
More precisely, 
let $\eta_0 \in C_c^{\infty}(\mathbb{R})$ be an even function  such that $0 \leq \eta_0 \leq 1, \eta_0(r)=1$ for $|r| \leq 1$, and $\eta_0(r)=0$ for $|r| \geq 2$. 
Let  
\begin{align}\label{def-chin}
\chi(r):=\eta_0(r)-\eta_0(2r),\quad \chi_n(r):=\eta_0(r / 2^n)-\eta_0(r / 2^{n-1}), \quad \chi_{\leq n}(r)=\eta_0(r / 2^n), 
\quad n\in \mathbb{N}. 
\end{align}
The standard Littlewood-Paley projector $\Delta_n$, $n \in \mathbb{N}$, is defined by
\begin{align}\label{def-fn}
\Delta_n f(x) & :=\sum_{\xi \in \mathbb{Z}^d} \chi_n (|\xi| ) \widehat{f}(\xi) e^{i2\pi x \cdot \xi},\quad \text{if}\ n\geq 1,
\end{align} 
and if $n=0$, 
\begin{align}
\Delta_0 f(x) & :=\sum_{\xi \in \mathbb{Z}^d} \chi_{\leq 0} (|\xi| ) \widehat{f}(\xi) e^{i 2\pi x \cdot \xi}.
\end{align} 
Then intermittency can be defined as the saturation level of the  Bernstein inequality
\[
	\|\Delta_n v_0\|_{L^p(\T^d)} \leq C_B 2^{d(\frac1q-\frac1p)n} \|\Delta_n v_0\|_{L^q(\T^d)}, \quad   n\in \mathbb{N}, \ 1\leq q \leq p\leq \infty,
\]
for some constant $C_B>0$.

We also have
\[
	\|\Delta_n v_0\|_{L^q(\T^d)} \leq   \|\Delta_n v_0\|_{L^p(\T^d)}, \quad   n\in \mathbb{N}, \ 1\leq q \leq p\leq \infty.
\]
For $1\leq q\leq p\leq \infty$, define
\begin{align}\label{ber-ineq}
D_{n,q,p} := \sup\Bigl\{ D\in[0,d]:
\|\Delta_n v_0\|_{L^p(\T^d)} 
\leq C_B\,2^{(d-D)(\frac1q-\frac1p)n}
\|\Delta_n v_0\|_{L^q(\T^d)} \Bigr\}.
\end{align}
We refer to $D_{n,q,p}$ as the intermittency dimension of
$\Delta_n v_0$. If $D_{n,q,p}\in [0,d)$, then clearly
\[
\|\Delta_n v_0\|_{L^p(\T^d)}
=
C_B\,2^{(d-D_{n,q,p})(\frac1q-\frac1p)n}
\|\Delta_n v_0\|_{L^q(\T^d)}.
\]
If $D_{n,q,p}=d$, then
\[
\|\Delta_n v_0\|_{L^p(\T^d)} \simeq
\|\Delta_n v_0\|_{L^q(\T^d)}.
\]
\medskip 

We consider 
the following class of initial data 
characterized by intermittency and regularity (or roughness).  
Fix any 
\begin{align}  \label{ve-ve0}
\kappa \in (0,1).
\end{align}
In the main results, 
the parameter $\kappa$ measures the regularity distance 
between the regularity $C^{1-\kappa}_x$ of initial data 
and the uniqueness class $C^1_x$ 
in the weak-strong uniqueness criteria. 

\begin{enumerate}  
   \item[{\rm (H1)}]  {\bf Intermittency dimension}:   
   There exist $\kappa_1 \in[0,\kappa)$ and $N \in\mathbb{N}$  
   such that 
   \begin{align}\label{asp-int}
       |D_{n,2,\infty}-D_{m,1,\infty}|\leq \kappa_
       1,
   \end{align} 
for all $n,m\geq N$.
    \item[{\rm (H2)}] 
    {\bf Roughness}: 
    There exists $N \in\mathbb{N}$ 
    such that for all $n\geq N$, 
    \begin{align}\label{asp-rou}
\|\Delta_n v_0\|_{L^\infty(\T^d)}\gtrsim 2^{-(1-\kappa_2)n},
\end{align} 
for some $\kappa_2>(\kappa+\kappa_1)/2$.
\end{enumerate} 

For initial data in the space $W^{1,p}_x$, 
 $2< p<\infty$, we assume the initial data satisfy the following conditions: 
  \begin{enumerate}  
   \item[{\rm (H1)'}]  {\bf Intermittency dimension}:   
   There exist $C_*\in (0,1/4)$, $\kappa_*\in (0,d]$, and $N \in\mathbb{N}$ 
  such that 
    \begin{align}\label{asp-reg-2}
   d-D_{n,2,p} \geq \kappa_*,
\end{align} 
for $n\geq N$, and for $n,m\geq N$,
\begin{align}\label{asp-int-2}
     |D_{n,1,p}-D_{m,2,p}|\leq \frac{\kappa_*C_*}{p}. 
 \end{align} 
 \item[{\rm (H2)'}] 
    {\bf Roughness}: 
    There exist $N \in\mathbb{N}$  
    such that 
for all $n\geq N$, 
 \begin{align}\label{asp-rou-2}
\|\Delta_n v_0\|_{L^{p}(\T^d)} \gtrsim 2^{-(1+\frac{\kappa_*C_*}{p})n}. 
\end{align} 
\end{enumerate}

\begin{remark}   
The underlying fact here is 
that the intermittency condition  {\rm (H1)} characterizes the decay rate of the initial Reynolds stress \eqref{R0-0},  
as well as the growth order of the upper and lower energy profiles 
(see, e.g., \eqref{oe-ineq-1} 
and \eqref{ue-ineq} below). 
Moreover, 
the roughness condition {\rm (H2)} ensures a sufficient decay rate of the initial relaxed solution, 
i.e., the solution to the approximate $\Lambda$-Euler equations, 
see \eqref{eq-e-u0-2} below.  
As a consequence, these
two conditions together enable us to compare the 
perturbation energy 
and the dissipation energy, 
which are critical for constructing 
dissipative solutions to the Euler equations.
\end{remark} 
  
\medskip 
We have the following general result that 
provides infinitely many admissible solutions 
with global decreasing energy profiles 
and with prescribed initial data 
satisfying the above intermittency dimension 
and regularity roughness conditions. 

\begin{theorem}[Global dissipative solutions under radial stability]\label{Thm-Non-euler}
Let $d\ge2$. Assume either that
\begin{align*}
0<\kappa<1,\qquad v_0\in C^{1-\kappa}(\T^d),
\end{align*}
satisfies {\rm (H1)} and {\rm (H2)}, or that
\begin{align*}
v_0\in W^{1,p}(\T^d),\qquad  2< p<\infty,
\end{align*}
satisfies {\rm (H1)'} and {\rm (H2)'}.
Assume moreover that $v_0$ is stable under radial multipliers in
the sense of Definition~\ref{def-radial-stability}. Then there exist infinitely many distinct global dissipative weak
solutions $u\in C([0,\infty);L^2_x)$ of \eqref{equa-euler} with initial data $v_0$.
\end{theorem}

\subsection{Comments on main results}  \label{Subsec-Com-Main}
We conclude this section with several comments on the main results.

\medskip 

\paragraph{\bf $\bullet$ 
Sharpness of weak-strong uniqueness principles} 
Theorem \ref{Thm-Non-euler} 
provides, for the first time, a relationship between the invalidity of the weak-strong uniqueness principle 
and intermittency dimension 
and roughness 
of the initial data.  

\medskip 

\paragraph{\bf $\bullet$ Regularity and energy related via intermittency} 
Our proof reveals that the two key ingredients -- the {\it regularity} 
and {\it energy} -- 
in the weak-strong uniqueness principle have a deep relationship via {\it intermittency}. 

To be precise, 
on one hand,  
the high $C^{1-\kappa}_x$-regularity of initial data 
accelerates, via intermittency, 
the fast $L^1_x$ decay of the Reynolds stress near the initial time, 
see estimate \eqref{est-r0-linfty} below. 
While in the convex integration construction,  
as reflected in \eqref{asp-0-energy} and \eqref{asp-e-q} below,  
the $L^1_x$ decay order $\delta_{q}$ of Reynolds stress 
is indeed consistent with the order 
of the perturbation of energy. 
See also \eqref{rl1s} below. 
Thus, 
the high $C^{1-\kappa}_x$-regularity gives rise to 
small energy perturbations. 
This is favorable for the construction of dissipative solutions. 

On the other hand, since the initial data do not belong to $C^{1-\kappa+}$, 
this roughness amplifies, via  intermittency again, 
the dissipation of the energy of  
initial relaxed solution, 
that is, 
the solution to approximate $\Lambda$-Euler equations 
\eqref{equa-nse-2}. 
This fact can also be seen from the derivation of estimate \eqref{eq-e-u0-2} below.

As a result, 
the regularity of initial data leads to 
a competition between two energies: 
the unfavorable energy perturbation caused by the Reynolds stress 
and the favorable energy dissipation of the approximate $\Lambda$-Euler equations. 
We derive the quantitative estimates of both energies. 
It turns out that 
the dissipation of the total energy 
can be characterized by intermittency dimension 
and the regularity of initial data. 
See the proof of \eqref{en-dis-inq-5} below.
\medskip 

\paragraph{\bf $\bullet$ Lacunary parallel shear} One example of the initial data in Theorem~\ref{thm-Nonuniq} is the lacunary parallel shear flow. We choose
\[
\frac{\kappa}{2}<\kappa_2\leq\kappa
\]
and define
\begin{align}\label{Diri-kernel}
v_0(x) := \sum_{n=1}^{\infty} 2^{-(1-\kappa_2)n}
\cos(2\pi2^n x_1)\,e_d.
\end{align}
By definition \eqref{def-chin},
\[
\chi_n(2^m)=\delta_{nm}, \qquad m,n\geq1,
\]
and hence
\[
\Delta_n v_0 = 2^{-(1-\kappa_2)n} \cos(2\pi2^n x_1)\,e_d.
\]
Then
\[
\sup_{n\geq1}
2^{(1-\kappa)n}\|\Delta_n v_0\|_{L^\infty}<\infty,
\]
so $v_0\in C^{1-\kappa}(\T^d)$. Moreover,
\[
\|\Delta_n v_0\|_{L^\infty} = 2^{-(1-\kappa_2)n},
\]
and thus {\rm (H2)} holds.

The ratios
\[
\frac{\|\Delta_n v_0\|_{L^\infty}} {\|\Delta_n v_0\|_{L^1}}
\qquad\text{and}\qquad \frac{\|\Delta_n v_0\|_{L^\infty}}
{\|\Delta_n v_0\|_{L^2}}
\]
are independent of $n$. Increasing the Bernstein constant
$C_B$ in \eqref{ber-ineq} if needed, we obtain
\[
D_{n,1,\infty}=D_{n,2,\infty}=d, \qquad n\geq1.
\]
Thus {\rm (H1)} holds with $\kappa_1=0$, while {\rm (H2)} holds
because $\kappa_2>\kappa/2$.

Furthermore, $v_0$ is a parallel shear flow. More generally, let
\begin{align}\label{parasf}
v(x)=g(P_{e^\perp}x)e,
\end{align}
where $g\in L^2_x$, $e$ is a coordinate vector, and
$P_{e^\perp}$ is the orthogonal projection onto $e^\perp$.
For every bounded real-valued radial symbol $\rho$, there exists
$g_\rho\in L^2$ such that
\[
\mathcal T_\rho v(x)=g_\rho(P_{e^\perp}x)e.
\]
Therefore
\begin{align}\label{para-flow}
\div\mathcal T_\rho v=0, \qquad
(\mathcal T_\rho v\cdot\nabla)\mathcal T_\rho v=0
\end{align}
in the sense of distributions. Hence every parallel shear flow,
including \eqref{Diri-kernel}, is stable under radial multipliers.

\medskip

\paragraph{\bf $\bullet$ Extension to the Euclidean space $\mathbb{R}^2$: 
Taylor-Couette flow}  
The idea of constructing admissible solutions to Euler equations in the present work is very robust, 
we would expect that it can be extended to the case of Euclidean space $\R^2$, 
which will be studied in another paper. 
 In this case, a primary model is the following Taylor-Couette flow, which, under the linear propagation
$\mathcal{S}(t)$ given by \eqref{multi-m} is a steady state of Euler equations.
\begin{itemize} 
    \item[$\bullet$] 2D Taylor–Couette flow:
\begin{align}\label{tcflow}
      v_0(x)=g(r)e_{\theta},
  \end{align}
  where 
  $x=(x_1, x_2) = 
  (r\cos\theta, r\sin\theta)\in \R^2$, 
   $g : [0,+\infty) \to \mathbb{R}$ is a scalar function 
  such that 
  $\int_{0}^{\infty}g^2(s)s\,\d s<\infty$,
  and 
  $e_{\theta}=(-\sin\theta, \cos\theta )$, $e_r=(\cos\theta, \sin\theta )$ 
  are the orthonormal basis  of $\mathbb{R}^2$.  
\end{itemize}
 
Fix $r_0>0$ and define, for $r>0$,
\[
p(r):=\int_{r_0}^{r}\frac{g(s)^2}{s}\,\mathrm ds.
\]
For $0<r<r_0$, by Fubini's theorem, 
\[
2\pi\int_0^{r_0}|p(r)|r\,\mathrm dr = \pi\int_0^{r_0}g(s)^2s\,\mathrm ds<\infty,
\]
so $p\in L^1_{\mathrm{loc}}(\mathbb R^2)$. On
$\mathbb R^2\setminus\{0\}$,
\[
\operatorname{div}(v_0\otimes v_0)
=-\frac{g(r)^2}{r}e_r, \qquad
\nabla p=\frac{g(r)^2}{r}e_r.
\]
Moreover, $(v_0\otimes v_0)e_r=0$, and, since $|p|$ is non-increasing
on $(0,r_0)$ and $p\in L^1_{\mathrm{loc}}(\mathbb R^2)$, this implies that
$r^2|p(r)|\to 0$ as $r \to 0$. Therefore
yields
\[
\operatorname{div}(v_0\otimes v_0)+\nabla p=0
\qquad\text{in }\mathcal D'(\mathbb R^2),
\]
and hence $v_0$ is a stationary weak solution of the Euler equations.

Moreover, the following lemma shows that the Taylor-Couette flow 
is preserved under the linear propagation 
of the approximate $\Lambda$-Euler equations.

\begin{lemma}\label{lem-tcflow}
   Let $\mathcal{S}(t)$ be the propagation operator defined by a radial multiplier $m(t,\cdot)$ and $v_0$ is a Taylor-Couette flow on $\R^2$. Then, for every $t\geq 0$, $\mathcal{S}(t)v_0$ is a Taylor-Couette flow, 
   and thus, 
   is still a steady solution 
   of Euler equations. 
\end{lemma}

A typical example of Taylor-Couette flow that satisfies Hypothesis {\rm (H1)} and {\rm (H2)} in Theorem~\ref{thm-Nonuniq} is the truncated power-law vortices constructed by Mengual \cite{M24}.
In the recent work \cite{M24},  
by using a suitable self-similar subsolution and the $L^\infty$ convex integration scheme, 
Mengual proved the failure of the weak-strong uniqueness for the 2D Euler equations on $\R^2$ 
with the initial data given by the truncated power-law vortex 
\begin{align}\label{mi-1}
v_0(x)=\psi(|x|) |x|^{-\alpha}x^{\perp}, 
 \ \ x=(x_1, x_2)\in \mathbb{R}^2. 
\end{align} 
Here, $\alpha\in (0,1)$ is a parameter, $x^\perp:=(-x_2,x_1)$ and $\psi$ is a smooth cut-off function used to guarantee the boundedness of $L^2(\R^2)$ norm of $v_0$.

The following lemma shows that the truncated power-law vortex satisfies Hypothesis {\rm (H1)} and {\rm (H2)}  
(with straightforward modifications to the Euclidean space $\mathbb{R}^2$).

\begin{lemma}\label{lem-m-ini}
 Let $\alpha\in (0,1)$ and
\begin{align}\label{mi}
f(x)= \psi(x)|x|^{-\alpha}x^{\perp},\quad x\in \R^2,
\end{align} 
where $\psi: \R^2\rightarrow [0,1]$ is a smooth radial cut-off function such that
\begin{align*}
	\psi(x)=\begin{cases}
		1,\quad |x|\in [0,1],\\
		0,\quad |x|\in [2,+\infty).
	\end{cases}
\end{align*}  

Then, for any $p\in [1,\infty]$, we have
\begin{align}\label{est-n-f}
    \|\Delta_n f\|_{L^p(\R^2)}\simeq 2^{-(1-\alpha+\frac2p)n},\quad n\geq 1.
\end{align}
In particular, for every $\kappa$ satisfying
\[
\alpha\le\kappa<\min\{2\alpha,1\},
\]
and every
\[
0<\kappa_1<\min\{\kappa,2\alpha-\kappa\},
\]
\eqref{asp-int} 
holds for all sufficiently large dyadic indices, and \eqref{asp-rou} holds with
$\kappa_2=\alpha$.
\end{lemma}

The proofs of Lemmas \ref{lem-tcflow} 
and \ref{lem-m-ini} are postponed to the Appendix 
\ref{App-TC-Flow} below. 
\medskip

\subsection{Novelties of the proof} 
\label{Subsec-Novel}
Our proof introduces a  {\it two-step backward}  construction of intermittent convex integration. 
The main novelties of the proof 
are summarized as follows. 

\medskip 
{\bf $\bullet$ Approximate $\Lambda$-Euler equation}: 
We introduce a new  
approximate Euler equation, 
called the {\it approximate $\Lambda$-Euler equation}. See \eqref{equa-nse-2} below. 

In contrast to the classical convex integration schemes, 
where the dissipation is usually treated perturbatively, 
the approximate $\Lambda$-Euler equation
contains a dissipation term acting only on wavenumbers above a suitable time-dependent cutoff.
This truncation plays a crucial role in dissipating the energy of solutions to the Euler equations.

The approximate $\Lambda$-Euler equation 
is different from the $\Lambda$-Navier-Stokes equation (NSE) recently 
introduced in \cite{czz24,czz25}, 
where the truncation was applied to the quadratic nonlinearity, 
instead of the dissipation part. 
The advantage of the approximate $\Lambda$-Euler equation 
is that it behaves like the target Euler equations near the initial time while simultaneously providing energy dissipation for positive times (see Theorem \ref{thm-solu-le}).

The first feature allows us to control the initial Reynolds stress in the convex-integration scheme near the initial time, which is usually difficult when the initial data are prescribed. In the NSE context, the free Stokes flow could be regarded as a possible approximation to the NSE; however, the corresponding Reynolds stress is merely bounded and may not vanish at the initial time, which can lead to weak solutions losing continuity at that moment (see \cite{bcv21,bms21}).

The second, dissipative feature is essential for constructing admissible solutions to the Euler equations with decreasing energy profiles. A delicate point here is that, due to the high-wavenumber truncation, the energy dissipates only at high frequencies, which may not suffice to dominate the energy perturbations generated by Reynolds stresses. This leads to a competition between the two types of energy—the dissipative and the Reynolds-stress energies. Further details will be provided in Subsection \ref{Subsec-Energy-Chan}.

\medskip 

{\bf $\bullet$ Backward construction of solutions 
with prescribed initial data}:  
Constructing continuous-energy solutions with prescribed initial data presents a fundamental challenge in convex integration. In the Navier–Stokes context, Burczak, Modena, and Székelyhidi \cite{bms21}  employed a truncation scheme preserving the initial data and established non-uniqueness of solutions in $C((0,1];L^2)$ for arbitrary $L^2_x$ initial data.  
Buckmaster, Colombo, and Vicol \cite{bcv21} later achieved non-uniqueness in $C([0,\infty); L^2_x)$ for any $H^3_x$ initial data, as well as in $C((0,\infty); L^2_x)$ for general $L^2_x$ data, by means of a gluing method ensuring compatibility of the approximate solutions near $t=0$.

A fundamental obstruction in such constructions is the non-vanishing Reynolds stress at the initial time, which typically causes a discontinuity of the energy profile. This issue was resolved in our earlier work \cite{czz24} by introducing a new initial approximation and a backward convex integration scheme based on a suitable choice of backward times, which enforces continuity of the energy at $t=0$. 

In the present work, we refine this scheme by selecting backward times and Reynolds-stress amplitudes according to the decay rate of the initial Reynolds stress. This decay rate depends quantitatively on the intermittency dimension, the blow-up rate of the wavenumber near $t=0$, and the regularity gap from the critical $C^1_x$ class (see~\eqref{R0-0}). As a result, our backward construction yields continuous-energy solutions $u \in C([0,\infty);L^2_x)$ corresponding to prescribed initial data satisfying the intermittency and roughness constraints.

\medskip 
{\bf $\bullet$ Construction of decreasing energy channel}:
To construct admissible solutions to the Euler equations, we adopt the approach developed in our recent work \cite{czz25} to build a decreasing energy channel. In contrast to the NSE and MHD settings of \cite{czz25}, the Euler equations possess no intrinsic dissipation, necessitating reliance on a very weak dissipation mechanism present in the approximate $\Lambda$–Euler equations. 

As discussed in Subsection \ref{Subsec-Com-Main}, the principal difficulty lies in controlling the competition between the perturbation energy generated by the Reynolds stress and the dissipation energy associated with the approximate $\Lambda$–Euler dynamics, a balance that is particularly delicate near the critical Lions' $C^1_x$ threshold, where the approximate $\Lambda$–Euler equation approaches the classical Euler system.
To ensure admissibility, we balance these two energy components so that the total energy of the constructed solution decreases monotonically in time. Under intermittency and prescribed roughness of the initial data, we obtain quantitative dissipation estimates, thereby proving the existence of a nontrivial decreasing energy channel.

\medskip 
{\bf $\bullet$ Two-step backward convex integration}: 
  With the aforementioned objects, 
the convex integration scheme allows us to construct distinct dissipative solutions with prescribed initial data, but only up to a local time $T^*$. 
The obstruction to a global construction lies in the delicate fact that we lose a smallness control of the Reynolds stress loses  in the outer region. 

More precisely, when a solution of the $\Lambda$-Euler system is taken as the initial step of the iteration, the energy decay rate must be sufficient to control the growth rate of the perturbation energy caused by the Reynolds error. However, since the time when the energy of the initial-step solution decays to zero is unknown, a local time $T^*$ is chosen such that the approximate solution is non-zero on $[0,T^*]$. Hence, by the energy identity \eqref{eq-e-2}, its energy decay rate is also non-zero on this interval. By adjusting the wavenumber $\Lambda$, the growth rate of the perturbation energy before $T^*$ can be controlled, thereby preserving the decreasing property of total energy (of the approximate solution plus the perturbation).

After $T^*$, a lower bound on the energy decay rate of the initial-step solution is no longer available. Thus, the main difficulty in constructing a globally dissipative solution is to ensure that, after $T^*$, the perturbation energy arising from the Reynolds stress can be controlled or remains non-increasing.

In order to construct global-in-time dissipative solutions, 
we shall perform a second convex integration. 
The key idea is to construct a refined background approximate profile 
by carefully selecting a smooth stabilized cutoff (see \eqref{def-Lambda-fr-val}). 
It turns out that 
the Euler-Reynolds system structure of the smooth stabilized background approximate
solution, the $q=0$ 
regularity package, 
and a refined smallness control of Reynolds stress can be preserved 
on the problematic outer region. In particular, after the local time $T^*$, such improvements ensure that the perturbation energy arising from the Reynolds error no longer increases.
This allows us to construct distinct global dissipative solutions 
in the second convex integration scheme. 
In the proof, 
we also give a general criterion for 
global dissipative solutions, 
see Theorem \ref{thm-global-diss-abstract} below. 
 
\section{Two-step backward convex integration scheme}   \label{Sec-Strategy} 

This section presents the key ingredients 
in the two-step backward convex integration scheme. 
For simplicity, 
we prove the main results 
for the initial data $v_0\in C^{1-\kappa}_x$. 
The case where $v_0\in W^{1,p}_x$ can be proved in an analogous manner with slight modifications.  
In the sequel, 
we consider the following relaxed Euler-Reynolds system,
\begin{equation}\label{equa-nsr}
	\left\{\aligned
	&\p_t \u+ \div(\u\otimes\u)+\nabla P_q=\div \ru,  \\
	&\div \u = 0,\\
	& \u(0)=v_0, 
	\endaligned
	\right.
\end{equation} 
where $\ru$ stands for the symmetric traceless $d\times d$ Reynolds stress. 

\subsection{Approximate $\Lambda$-Euler equations}    

The first step in the backward convex integration scheme 
is to construct an appropriate initial relaxed solution.  
The crucial role here is played by 
the approximate $\Lambda$-Euler equation.

\medskip 
To be precise, choose $\kappa_0\in (0,1)$ such that
\begin{align}\label{def-k0}
\kappa_0>\kappa,\quad  \kappa_2>\frac12(\kappa_0+\kappa_1),
\end{align} 
where $\kappa_1$ and  $\kappa_2$ are given by assumptions {\rm (H1)} and {\rm (H2)}, respectively. We introduce the wavenumber  $\Lambda: (0,+\infty)\rightarrow \R^+$ 
satisfying the following properties:
\begin{itemize}
\item[{\rm (W1)}] $\Lambda: (0,+\infty)\rightarrow \R^+$ is a  strictly decreasing smooth function; 
\item[{\rm (W2)}] $\Lambda(t) \geq 2 $ for all $t\in (0,+\infty)$ and $\Lambda'(t)$ is uniformly bounded on $[1/2,+\infty)$;
\item[{\rm (W3)}] $\Lambda(t)=t^{-\eta}$ on $(0,1/2]$ 
with $\eta$ being an integer 
such that 
\begin{align}\label{def-eta}
\eta \in ( \frac{1}{4 \alpha}, \frac{1}{2 \alpha}),
\end{align}
and $\alpha$ is a small positive constant satisfying
\begin{align}\label{def-alpha}
        0<\alpha<\min\{\frac{2\kappa_2-\kappa_0-\kappa_1}{4 }, \frac{2-\kappa_0}{6}\}.
    \end{align}  
\end{itemize}

For convenience, we extend $\Lambda$ to $[0,+\infty)$ by defining
\begin{align*}
	\Lambda(0)=\infty.
\end{align*}

\medskip 
Let $\varphi:\R^d\to[0,1]$ be a smooth radial function such that
\begin{align*}
\varphi(x)=0\quad \text{for } |x|\le1,\qquad
0<\varphi(x)<1\quad \text{for } 1<|x|<2,\qquad
\varphi(x)=1\quad \text{for } |x|\ge2.
\end{align*}
 We choose $\varphi$ to be flat at $|x|=1$ and $|x|=2$, 
such that for each finite order $M$, with
$h:=\varphi^2$, there exist $\varepsilon_M\in(0,(M+1)^{-1})$ and $C_M>0$ such that
\begin{align}\label{flat-est}
|\partial^\beta h(x)|\le C_M h(x)^{1-\varepsilon_M},
\qquad 1\le |\beta|\le M,
\end{align}
with the convention that the right-hand side is zero where $h=0$. 
 For instance, 
we may choose $\varphi(x)=f(|x|)$ 
with 
\begin{align*}
    f(r)=\begin{cases}
        0,\quad & r\leq 1,\\[4pt]
        \dfrac{e^{-\frac{1}{r-1}}}{e^{-\frac{1}{r-1}}+e^{-\frac{1}{2-r}}},\quad & 1<r<2,\\[8pt]
        1,\quad & r\geq 2.
    \end{cases}
\end{align*} 

Let $\varphi_{\Lambda(t)}(\xi):= \varphi(\xi/\Lambda(t))$  
be the corresponding scaled function, 
$t\in [0,+\infty)$. 
The frequency truncated operator $\P_{\geq\Lambda(t)}$ is defined by
\begin{align}  
	\widehat{\P_{\geq\Lambda(t)}u}(\xi)
    := \widehat u(\xi) \varphi_{\Lambda(t)}(\xi), \ \ \xi\in \mathbb{Z}^d
\end{align} 
for any $u \in \mathscr{D}'(\mathbb{T}^d)$. 
In an analogous way, 
the operator $\P_{< \Lambda(t)}$ is defined by 
\begin{align}
	\widehat{\P_{<\Lambda(t)}u}(\xi):= \widehat u(\xi)(1-\varphi_{\Lambda(t)}(\xi)), \ \ \xi\in \mathbb{Z}^d. 
\end{align} 
Moreover, we define
\begin{align*}
	\P_{<\Lambda(t)}:=\Id,\qquad \P_{\geq \Lambda(t)}:=0
\end{align*}
whenever $\Lambda(t)=\infty$.

Now, 
we consider the approximate $\Lambda$-Euler equation
\begin{align}\label{equa-nse-2}
	\begin{cases}
		\p_t   u+\nu (-\Delta)^\alpha \P^2_{\geq \Lambda(t)} u + ( u\cdot \nabla )u+\nabla p=0,\\
		\div u=0,\\
		u(0)=v_0, 
	\end{cases}
\end{align}
where $\nu$ is a positive viscosity coefficient, $\alpha$ is given by {\rm (W3)}, 
the hypo-dissipative operator $(-\Delta)^{\alpha}$ is defined via the Fourier transform on the flat torus
\begin{align}
\mathcal{F}((-\Delta)^{\alpha}u)(\xi)= (2\pi|\xi|)^{2\alpha}\mathcal{F}( u)(\xi), 
\ \ \xi\in \mathbb{Z}^d, 
\end{align} 
and $\P^2_{\geq \Lambda(t)}$ 
stands for the composition of operators, 
i.e., 
$\P^2_{\geq \Lambda(t)} = \P_{\geq \Lambda(t)} \circ \P_{\geq \Lambda(t)}$.

In the approximate $\Lambda$-Euler equation \eqref{equa-nse-2},  
the dissipation acts only on frequencies above $\Lambda(t)$, while the nonlinearity remains unchanged. As a result, 
for any $t>0$, the viscosity 
is effective only at high wavenumbers exceeding $\Lambda(t)$, 
but vanishes at the initial time $t = 0$ 
due to the blow-up of the wavenumber.

This feature is different from that of the 
$\Lambda$-NSE in \cite{czz24,czz25} 
for the 3D NSE.    
In the $\Lambda$-NSE, 
the dissipation is not truncated, 
while the nonlinearity is 
truncated below the wavenumber. 
As a result, 
the $\Lambda$-NSE behaves like the NSE near the initial time, 
but not the Euler equations.

\medskip 
    Recall that $\Lambda$ is strictly decreasing and denote its inverse by $\Lambda^{-1}$. We define $\tau(0)=\infty$ if $\xi=0$. For $\xi\in\mathbb{Z}^d$ we define
    \begin{equation} \label{tau-def}
        \tau(\xi):=
        \begin{cases}
            \Lambda^{-1}(|\xi|), & \text{if} \quad |\xi|> \inf_t \Lambda(t),\\
            \infty, & \text{if} \quad |\xi|\le \inf_t \Lambda(t).
        \end{cases}
    \end{equation}
    
    Note that, in the large temporal regime 
    where $t> \tau(\xi)$, the solution $\wt u$ of 
    the approximate $\Lambda$-Euler equation \eqref{equa-nse-2}
    satisfies
    \begin{align}\label{equa-le-out}
        \p_t \widehat{\wt u}(\xi)+\nu (2\pi|\xi|)^{2\alpha} \varphi^2_{\Lambda(t)}(\xi)
        \widehat{\wt u}(\xi ) 
        +\sum_{ l+h=\xi} 2\pi i  (\xi\cdot\widehat{\wt u}(l) )\widehat{\wt u}(h) + i\widehat{p}(\xi)\xi=0,
        \end{align} 
    where $\widehat{\wt u}$ denotes the Fourier transform of $\wt u$, 
    while in the small temporal regime where 
    $t\leq \tau(\xi)$,  
    \begin{align}\label{equa-le-in}
        \p_t \widehat{\wt u}(\xi)+\sum_{ l+h=\xi}2\pi i  (\xi\cdot\widehat{\wt u}(l) )\widehat{\wt u}(h) + 2\pi i\widehat{p}(\xi)\xi=0.
        \end{align}
    In order to solve \eqref{equa-nse-2}, we introduce the propagation operator $\mathcal{S}(t):L^2_x \mapsto L^2_x$, 
    defined 
    by the multiplier 
    \begin{align}\label{multi-m} 
     m (t,\xi):=
    \begin{cases}
     1, &\text{if}\quad t\leq \tau(\xi),\\
        e^{-\nu\int_{\tau(\xi)}^t\varphi^2_{\Lambda(s)}(\xi)\d s(2\pi|\xi|)^{2\alpha}}, &\text{if}\quad t>\tau(\xi).
    \end{cases}
    \end{align}  
 Since $m$ is a radial multiplier, the assumption in Theorem~\ref{Thm-Non-euler} implies that $\mathcal{S}(t) v_0$ is a steady state of Euler equations for every $t\geq 0$. 
Hence, we may take
$\widetilde u(t) = \mathcal{S}(t) v_0$
as a solution to the $\Lambda$-Euler equation \eqref{equa-nse-2}.
More precisely, the explicit expression of $\wt u$
in the Fourier modes can be given by 
\begin{align}\label{def-solu}
\widehat {\wt u} (t,\xi):=
\begin{cases}
 \widehat v_0(\xi), &\text{if}\quad t\leq \tau(\xi),\\
    e^{-\nu\int_{\tau(\xi)}^t\varphi^2_{\Lambda(s)}(\xi)\d s(2\pi|\xi|)^{2\alpha}}\widehat v_0(\xi), &\text{if}\quad t>\tau(\xi).
\end{cases}
\end{align}

\medskip  
Theorem \ref{thm-solu-le} below contains the key properties of $\wt u$ 
used in this work. 

\begin{theorem}[Approximate $\Lambda$-Euler equations]  \label{thm-solu-le}  
Let $v_0$ be the initial data given by Theorem~\ref{Thm-Non-euler}. 
Then, 
the solution $\wt u$ given by \eqref{def-solu} of the approximate $\Lambda$-Euler equations \eqref{equa-nse-2}  is smooth for positive times 
and satisfies the energy balance:
\begin{align}\label{eq-e-2}
	\frac12 \|\wt u(t)\|_{L^2_x}^2 + \nu \int_{0}^{t}\|(-\Delta)^{\frac{\alpha}{2}} \P_{\geq \Lambda(s)}\wt u(s)\|_{L^2_x}^2\d s = \frac12 \|v_0\|_{L^2_x}^2, \qquad t\in [0,+\infty).
\end{align}
Moreover, $\wt u$ is continuous in $L^2_x$ at the initial time:
\begin{align}\label{s-con-v0}
	\wt u(t) \rightarrow v_0 \quad {\rm strongly\  in}\  L^2_x\ {\rm as}\  t\to 0^+.
\end{align}
\end{theorem}

The proof of Theorem \ref{thm-solu-le} 
is given in Section \ref{Sec-Lam-Euler} below. 

\medskip 
Now, we choose the solution $\wt u$ to 
the approximate $\Lambda$-Euler equations \eqref{equa-nse-2} 
as the initial relaxed 
solution to the Euler-Reynolds system \eqref{equa-nsr} 
at the initial step $q=0$.

Correspondingly, 
the initial Reynolds stress has the expression 
\begin{align}
	&\mathring{R}_0(t) := -\nu\mathcal{R} (-\Delta)^\alpha \P^2_{\geq\Lambda(t)} \wt u(t),  \label{r0}
\end{align}  
where $\mathcal{R}$ is the inverse divergence operator defined by (see \cite{dls13}) 
\begin{align}
	(\mathcal{R} v)^{kl}
	:= \partial_k \Delta^{-1} v^l + \partial_l \Delta^{-1} v^k
	-\frac{1}{d-1}\Big(\delta_{kl}+(d-2)\partial_k \partial_l \Delta^{-1}\Big)\div \Delta^{-1} v
	\label{operaru}
\end{align} 
for mean-free functions $v$. 

As mentioned above, the key fact here is that the initial Reynolds stress is uniformly bounded and vanishes at the initial time. For later use, we introduce the tail envelope of the intermittency dimensions by 
\begin{align}
    D^\sharp_{n,1,\9}:= \sup_{j\geq n}D_{j,1,\9},\qquad n\geq N.
\end{align}
By definition, $\{D^\sharp_{n,1,\9}\}_{n\geq N}$ is non-increasing and 
\[
0\leq D_{n,1,\9}\leq D^\sharp_{n,1,\9}\leq d.
\]
Moreover, \eqref{asp-int} implies
\begin{align}\label{est-1.8-2}
    |D_{n,2,\9}-D^\sharp_{m,1,\9}|\leq \kappa_1,\qquad n,m\geq N.
\end{align}

\begin{proposition}[Initial Reynolds stress]  \label{Prop-R0}  Let $d\geq 2$, 
$\kappa \in (0,1)$ and
$v_0\in C^{1-\kappa}(\T^d)$
be stable under radial multipliers in
the sense of Definition~\ref{def-radial-stability} 
and 
satisfy {\rm (H1)} and {\rm (H2)}. 
Then, the following holds:  

$(i)$ Uniform boundedness: 
there exists $C_0>0$ such that 
\begin{align}\label{est-r0-l1}
    \|\mathring{R}_{0}\|_{C_tL^{1}_x} \leq C_0 \|v_0\|_{L^2_x}. 
\end{align}

$(ii)$ Decay estimate: 
let $t_0\in (0,1/2)$ be such that $ t_0 \leq \Lambda^{-1} (2^N)$, 
where $N\in \mathbb{N}$ is as in {\rm (H1)} and {\rm (H2)}. 
For $t\in (0,t_0)$, 
let $n_*(t) \in \mathbb{N}$ be such that 
$2^{n_*(t)} \leq \Lambda(t) <2^{n_*(t)+1}$. 
Then, we have 
\begin{align} \label{R0-0}
\|\mathring{R}_{0}(t)\|_{L^{1}_x}
\leq C_R t^{\eta(2-2\alpha+d-D^\sharp_{n_*(t),1,\infty}-\kappa_0)},
\quad t\in (0,t_0).
\end{align}
where 
$\eta$ and $\alpha$ are as in the property {\rm (W3)} of the wavenumber,  $\kappa_0$ is given by \eqref{def-k0}
and $D^\sharp_{n_*(t),1,\infty}$ is the tail envelope of the
intermittency dimensions at the scale $n_*(t)$, and $C_R$ is a constant independent of $t$, chosen sufficiently large to absorb
the fixed universal constants appearing in the multiplier estimates below.

In particular, 
$\mathring{R}_{0}(t)$ vanishes strongly in $L^1_x$ 
as $t\to 0^+$. 
\end{proposition}

\subsection{Backward times} 
 
We construct relaxed solutions backward in time. 
That is, 
at each step $q\in \mathbb{N}$, we add the velocity perturbation 
on the time interval $[T_q, \infty)$,  
while no perturbations are put on $[0,T_q]$, 
where $\{T_q\}$ decreases to zero. 
The backward construction allows us to maintain the prescribed initial data in iterative steps, 
and thus, for the final limit solution. 
In order to keep the iterative estimates,  
the key subtlety lies in choosing the backward times and the amplitudes of the Reynolds stresses so that they are compatible with the decay rate 
\eqref{R0-0} of the initial Reynolds stress.

To be precise, 
for every $q\in \mathbb{N}$,  
we choose the parameters $\lbb_q$ and $\ell_q$ by 
\begin{equation}\label{la}
	\la=a ^{(b^q)}, \ \ \ell_q:=\la^{-30d}, 
\end{equation} 
where  $b\in 7\mathbb{N}$ is a large integer such that 
\begin{align} \label{b}
 b> 10^4 d^2\varepsilon^{-1} 
\end{align}
with $\ve<10^{-2}\in\mathbb Q$ sufficiently small.  Choose a real number $\theta>2$,
sufficiently large for all later lower-bound requirements, such that \eqref{sf-delta1-prep-11} and \eqref{theta-local-end} below hold and
\begin{align} \label{cho-theta}
4^{-\theta}\le t_0,\qquad
\|\wt u(4^{-\theta})\|_{L^2_x}>0,\qquad
\Lambda(4^{-\theta})\notin2^{\mathbb N}.
\end{align}
This is possible because $\wt u(t)\to v_0$ strongly in $L^2_x$, $v_0\ne0$, and the forbidden set of
$\theta$'s is discrete. Then choose the integer $a$ sufficiently large so that
\begin{align}
\ln a>(\theta+2)(d+2)\alpha^{-1}\ln4+\ln(1+\|v_0\|_{L^2_x}^2)+\ln C=:\ln a_0,
\label{def-a-0}
\end{align}
where $C$ is a large universal constant.

We divide the time interval $(0,+\infty)$ into infinitely many sub-intervals 
\begin{align*}  
 (0,+\infty) = 
  \bigcup\limits_{q\geq 1} [T_{q+1}, T_q] 
  \cup  [T_1,+\infty), 
\end{align*}
where $T_1:=4^{-\theta}-\ell_0\in (0,t_0]$ with $t_0$ as in Proposition \ref{Prop-R0},   
and 
\begin{align} \label{Tq-def}
T_{q+1}=4^{-(\theta+q)}-\ell_{q},\quad q\geq 1.
\end{align}
Note that by \eqref{la} and \eqref{def-a-0}, $T_{q}>4^{-(\theta+q)}$ for all $q\geq 1$.

The amplitude parameters $\{\delta_{q}\}_{q\geq 0}$ of Reynolds stress are chosen 
as follows. 
Let $\delta_1$ be large enough such that the 
upper bound of the energy profile  
in \eqref{asp-0-energy} below holds. 
Then, take $\delta_0$ and $\delta_2$ large enough such that 
\begin{align} \label{delta0}
     \delta_0\geq \max\{2\delta_1, \|\wt u\|^2_{C_tL^2_x}\} 
\end{align}
and 
\begin{align}\label{def-delta2} 
\delta_2 \geq \frac{8C_R}{c_*}C_0\|v_0\|_{L^2_x}, \end{align} 
where $C_0$ is given by \eqref{R0-L1} below and 
$c_*>0$ is a small parameter such that
\begin{align}\label{def-c8}
	0< c_*< \min\left\{1, \frac{\ve_u}{500d}\right\},
\end{align}
with $\ve_u$ being the geometric constant given by Geometric Lemma \ref{geometric lem 2}.  
Moreover, let 
\begin{align}\label{def-delta}
    &\delta_{q+ 2}:= \frac{8C_R}{c_*} (T_q+\ell_{q-1})^{k_q}= \frac{8C_R}{c_*} 4^{-(\theta+q-1)k_q }, \quad q\geq 1.
\end{align}  
Here, $C_R$ is given by \eqref{R0-0} and
\begin{align} \label{kq-def}
k_q := \eta\bigl(
2-2\alpha+d-D^\sharp_{n_q,1,\infty}-\kappa_0
\bigr).
\end{align}
is exactly the decay rate of the initial Reynolds stress in \eqref{R0-0},   
where $n_q\in \mathbb{N}$ satisfies  
\begin{align}\label{def-nq}
    \Lambda(T_q+\ell_{q-1})\in [2^{n_q}, 2^{n_q+1}). 
\end{align}
Moreover, we choose $\theta$ sufficiently large, such that
\begin{align}
\frac98\delta_3<\delta_1.
\label{sf-delta1-prep-11}
\end{align}

\medskip 
\paragraph{\bf Claim $1$:} 
The amplitude parameters 
$\{\delta_{q+2}\}_{q\geq 1}$ are strictly decreasing on $q$ and 
$\delta_{q+2} \to 0^+$ as $q\to +\infty$. 

Since $\eta\in (\frac{1}{4\alpha},\frac{1}{2\alpha})$,
\[
0<\alpha<
\min\left\{ \frac{2\kappa_2-\kappa_0-\kappa_1}{4},
\frac{2-\kappa_0}{6}
\right\},
\]
and $D^\sharp_{n_q,1,\infty}\leq d$, we have
\begin{align}\label{kq-1}
k_q &= \eta\bigl(
2-2\alpha+d-D^\sharp_{n_q,1,\infty}-\kappa_0
\bigr)>
\frac{1}{4\alpha}
\bigl(2-2\alpha-\kappa_0\bigr)
\geq 1, \qquad q\geq 1.
\end{align}
Since $0\leq D^\sharp_{n_q,1,\infty}\leq d$, we also have the
uniform upper bound
\begin{equation}\label{eq:k_q_upper_bound}
k_q \leq K_*:=
\eta(2-2\alpha+d-\kappa_0)
\leq \eta(d+2), \qquad q\geq 1.
\end{equation}
Taking into account \eqref{def-delta}  
and that 
the backward time $T_q+\ell_{q-1}$ decreases to $0$  
we thus infer that $\delta_{q+2} \to 0^+$ 
as $q\to +\infty$. Moreover, since $\Lambda$ is decreasing in time and
$T_q+\ell_{q-1}$ is decreasing in $q$, definition
\eqref{def-nq} implies
\[
n_q\geq n_{q-1}, \qquad q\geq 2.
\]
Since $\{D^\sharp_{n,1,\infty}\}_{n\geq N}$ is non-increasing,
we have
\begin{align}\label{kq-kq1}
k_q &= \eta\bigl(
2-2\alpha+d-D^\sharp_{n_q,1,\infty}-\kappa_0
\bigr)\geq \eta\bigl(
2-2\alpha+d-D^\sharp_{n_{q-1},1,\infty}-\kappa_0 \bigr)
=k_{q-1}.
\end{align}
Thus, $k_q$ is non-decreasing in $q$, and consequently
\[
(T_q+\ell_{q-1})^{k_q} \leq
(T_q+\ell_{q-1})^{k_{q-1}}.
\]
Taking into account that $k_q>1$ by \eqref{kq-1}, together with
the definition of the backward times in \eqref{Tq-def}, we obtain
\begin{align*}
    \delta_{q+ 2}= \frac{8C_R}{c_*} (T_q+\ell_{q-1})^{k_q}  
    < \frac{8C_R}{c_*}  
    \big(4(T_q+\ell_{q-1}) \big)^{k_{q-1}} 
    = \delta_{q+1}, 
    \qquad q\geq 2. 
\end{align*} 
Increasing $\delta_2$ if necessary, this yields the Claim $1$.

\medskip
\paragraph{\bf Claim $2$:}  
The initial Reynolds stress satisfies the following decay estimate
\begin{align}\label{est-r0-0}
\|\mathring{R}_0\|_{C_{\left[0, T_q+\ell_{q-1}\right]} L_x^1} \leq C_R\left(T_q+\ell_{q-1}\right)^{k_q}=\frac{1}{8} c_* \delta_{q+2}.
\end{align} 

For every $q \geq 1$ and $0<t \leq T_q+\ell_{q-1}$, the monotonicity of $\Lambda$ and (2.33) imply
$$
n_*(t) \geq n_q .
$$
Since $\left\{D_{n, 1, \infty}^{\sharp}\right\}_{n \geq N}$ is non-increasing, it follows that
$$
D_{n_*(t), 1, \infty}^{\sharp} \leq D_{n_q, 1, \infty}^{\sharp} .
$$
Consequently,
$$
\eta(2-2 \alpha+d-D_{n_*(t), 1, \infty}^{\sharp}-\kappa_0) \geq k_q .
$$
Since $0<t \leq T_q+\ell_{q-1}<1$, \eqref{R0-0} yields
$$
\|\mathring{R}_0(t)\|_{L_x^1} \leq C_R t^{k_q} \leq C_R\left(T_q+\ell_{q-1}\right)^{k_q} .
$$
Taking the supremum over $t \in\left[0, T_q+\ell_{q-1}\right]$ and using \eqref{def-delta}, we obtain 
\eqref{est-r0-0} and prove Claim $2$. 

\medskip 
In the backward construction,  
no perturbations are added over the time interval $[0,T_{q+1}]$ at step $q$.  
Thus, the Reynolds error $\mathring{R}_{q}$ 
at step $q$ is exactly the initial 
Reynolds stress, i.e., 
$\mathring{R}_{q} = \mathring{R}_{0}$ 
on $[0,T_{q+1}]$, 
and estimate \eqref{est-r0-0} indeed bounds 
the Reynolds error on  $[0,T_{q+1}]$  at each step $q$.  
More importantly, 
according to the claim,  
estimate \eqref{est-r0-0} 
yields that the Reynolds stress strictly decreases to $0$ along the iteration. 
This is important to guarantee that 
the constructed solutions 
in convex integration 
have a continuous energy profile at the initial time.

\subsection{Energy channels} \label{Subsec-Energy-Chan}
Concerning the construction of energy profiles, let us denote
\begin{align}\label{def-out}
    \overline{T}_{q+1} :=T_{q+1} + \ell_q,\quad\text{and}\quad \underline{T}_{q+1} :=T_{q+1} - \ell_q ,\quad q\geq 0.
\end{align}
Then, we choose a continuous function $e$, which is smooth on positive times, and a positive constant $\delta_1$ large enough 
such that
\begin{align}
	\frac34\delta_{2}\leq  e(t)- \|u_0(t)\|_{L^2_x}^2 \leq \delta_{1},\ \
	 & \text{for }\ t\in [\underline{T}_1,+\infty), \label{asp-0-energy} \\
	 \frac34\delta_{q+2}\leq  e(t)- \|u_0(t)\|_{L^2_x}^2 \leq \delta_{q+1},\ \
      &\text{for }\ t\in [ \underline{T}_{q+1},\overline{T}_{q}], \label{asp-e-q}
\end{align} 
and
\begin{align}\label{bdd-e-c1}
\| e\|_{C_{[T_{q+2},+\infty)}^1}\lesssim \lambda_q^{d+1},
\end{align}
where the implicit constant is independent of $q$.

\begin{remark}
    We note that, the set of 
    energy profiles satisfying \eqref{asp-0-energy}-\eqref{bdd-e-c1} is non-empty. In fact, 
    one can choose the energy profile 
    in the energy channel $(e^{\mathrm{h}},  e^{\mathrm{l}})$ 
    as in Section \ref{Sec-prf-thm} below, 
    where the slope of 
    $e^{\mathrm{h}}$ and $ e^{\mathrm{l}}$ 
    on $[T_{q+2}, +\infty)$ 
    can be bounded by 
    \begin{align*}
        \frac{\delta_{q+1}- \frac34\delta_{q+2} }{\overline{T}_{q}-\underline{T}_{q+1} }\lesssim \frac{\delta_{q+1} }{\overline{T}_{q}} 
        \lesssim 
        4^{-(\theta+q-2)k_{q-1}+(\theta+q-1)} 
        \lesssim \la^{d+1}, 
    \end{align*}   
via our choice of 
    $\delta_q$ in \eqref{def-delta} and  $\lambda_q$ in \eqref{la}-\eqref{def-a-0} below.
\end{remark}

The upper and lower energy estimates 
in \eqref{asp-0-energy} and \eqref{asp-e-q} 
differ in their orders. 
This discrepancy is important to construct 
a non-empty energy channel 
$(e^{\mathrm{h}}, e^{\mathrm{l}})$, 
such that the upper and lower energy profiles 
$e^{\mathrm{h}}$, $e^{\mathrm{l}}$   
decrease in time. 

To be precise, we construct 
the upper and lower energy profiles 
$e^{\mathrm{h}}$ and $ e^{\mathrm{l}}$ 
as follows 
\begin{align} \label{upe-undere-def}
	e^{\mathrm{h}}:= \|\wt u(t)\|_{ L^2_x}^2+\overline f(t),\quad  e^{\mathrm{l}}:= \|\wt u(t)\|_{L^2_x}^2+\underline f(t), 
	\quad t\in [0,+\infty), 
\end{align}  
where 
\begin{align} \label{upf-def}
	&\overline f|_{[\ot_1,+\infty)}=\delta_2,\quad \overline{f}|_{[\ot_2,\ot_1]}=\delta_2+\frac{\delta_3-\delta_2}{\ot_2-\ot_1}(t-\ot_1),\quad \overline{f}|_{[\ot_{q+1},\ot_{q}]}= \delta_{q+1}+\frac{\delta_{q+2}-\delta_{q+1}}{\ot_{q+1}-\ot_q}(t-\ot_q),
\end{align} 
and
\begin{align}  \label{underf-def}
	&\underline f|_{[\ut_1,+\infty)}=\frac34\delta_2,\quad \underline{f}|_{[\ut_2,\ut_1]}=\frac34\delta_2+\frac34\frac{\delta_3-\delta_2}{\ut_2-\ut_1}(t-\ut_1),\quad \underline{f}|_{[\ut_{q+1},\ut_{q}]}= \frac34\delta_{q+1}+\frac34\frac{\delta_{q+2}-\delta_{q+1}}{\ut_{q+1}-\ut_q}(t-\ut_q).
\end{align}
Here, $\ot_{q+1}$ and $\ut_{q+1}$ are given by \eqref{def-out}. 

It is not difficult to check that the energy channel is not empty. 
Moreover, 
the quantitative estimate of the dissipation energy and the perturbation energy are derived in Section \ref{Sec-prf-thm} below, 
which are important to ensure the decreasing property of the energy channel.  

We also note that the Reynolds stress 
on the outer region $[\ut_{1}, \infty)$ 
has a possibly large size $\delta_2$, 
which may lead to dissipative solutions 
up to a local time.  
This is the content of Theorem \ref{Thm-Non-euler-Local} in the next subsection.

\subsection{Local dissipative solutions} 

Now, we are ready to state the 
local version of Theorem \ref{Thm-Non-euler}. 

\begin{theorem} \label{Thm-Non-euler-Local} Let  $d\geq 2$, $0< \kappa< 1$, 
and $v_0$ be the initial data given by Theorem~\ref{Thm-Non-euler}. Then, 
there exist $T^*\in (0,+\infty)$ depending on $v_0$ and a pair of continuous functions $(e^{\mathrm{l}}, e^{\mathrm{h}}): [0,\infty)\rightarrow [0,\9)$, decreasing on $[0,T^*]$, with  
\begin{align}   
& e^{\mathrm{l}}(t) < e^{\mathrm{h}}(t), \qquad \forall\  0 < t\leq T^*, \label{el-eh}  \\ 
& e^{\mathrm{l}}(0)=e^{\mathrm{h}}(0)=\|v_0\|_{L^2_x}^2,  \label{el0-eh0}
\end{align} 
such that for every $e \in C([0,\infty))$, smooth on $(0,\infty)$, satisfying \eqref{bdd-e-c1} and
\[
e^{\mathrm{l}}(t) \leq e(t) \leq  e^{\mathrm{h}}(t), \qquad \forall t\in [0,\infty),
\]
there exists a global weak solution $u\in C({[0,\infty)}; L_x^2)$ to 
Euler equation \eqref{equa-euler} 
	satisfying 
  \begin{align}\label{energy-1}
 \|u(t)\|_{L^2_x}^2=e(t) \quad\text{for all} \quad t\in [0,\infty).
\end{align} 
\end{theorem}

The proof of Theorem \ref{Thm-Non-euler-Local} relies crucially on the following 
iterative estimates.

\begin{proposition} [Main iterative estimates]\label{Prop-Iterat} Take $a$ sufficiently large to satisfy \eqref{def-a-0}. Then, the following holds:
	
	Suppose that for some $q\in \mathbb{N}$, $(\u, \ru)$ solves  \eqref{equa-nsr}
	and satisfies  
   \begin{align}
	&  \| u_q \|_{ C_{t}L^2_x}\lesssim  \sum\limits_{n=0}^q \delta_{n}^{\frac12} ,\label{ul2}\\
	&  \| u_q  \|_{C_{[T_{q+2},+\infty),x}^{1}} \lesssim \lambda_{q}^{d+1},\label{uc} \\
	&  \| \mathring{R}_{q}\|_{C_{t} L^1_x}\leq  c_*\delta_{q+ 2} ,  \label{rl1s}\\
	&  \| \mathring{R}_{q} \|_{C^{1}_{[T_{q+2},+\infty)}L^1_x}+ \| \mathring{R}_{q}  \|_{C_{[T_{q+2},+\infty)}W^{1,1}_x} \lesssim \lambda_{q}^{2d+2},  \label{rl1b-s}\\
    &
u_q(t)\equiv u_0(t)\quad\text{for }0\le t\le T_{q+1},
\qquad
\mathring R_q(t)\equiv\mathring R_0(t)
\quad\text{for }0\le t\le\overline T_{q+1},
\label{rq=r0}\\
	&  \frac12\delta_{q+ 2}\leq  e(t)-\|u_q(t)\|_{L^2_x}^2 \leq 2\delta_{q+1},\ \ t\in [T_{q+1},+\infty),  \label{energy-est}
 \end{align}
where the implicit constants are independent of $q$.

	Then, there exists $(u_{q+1}, \mathring{R}_{q+1})$
	which solves \eqref{equa-nsr} and satisfies \eqref{ul2}-\eqref{energy-est} with $q+1$ replacing $q$.
	In addition,  we have
	\begin{align}\label{u-B-L2tx-conv}
		\|u_{q+1}-u_{q}\|_{C_tL^{2}_{x}}\lesssim	 \delta_{q+ 1}^{\frac12}, 
	\end{align} 
    where the implicit constants are independent of $q$.
\end{proposition}

The proof of Proposition \ref{Prop-Iterat} occupies Sections \ref{Sec-Interm-Flow}, 
\ref{Sec-Rey-stress} and \ref{Sec-Energy}. 
In the convex integration scheme, 
we use  the Mikado flows, 
initially introduced in \cite{DS2017},  
to construct velocity perturbations. 
In dimensions higher than three, 
the Mikado flows 
are concentrated near straight lines.   
Suitably adjusting the line directions, the Mikado flows can be chosen with disjoint supports. This technique was first introduced by Daneri and Sz\'ekelyhidi \cite{DS2017} and subsequently utilized by Isett \cite{I18}, Buckmaster, De Lellis, Sz\'ekelyhidi, and Vicol \cite{bdsv19}, and Novack and Vicol \cite{nv23} to address the Onsager conjecture.

In dimension two, however, standard Mikado flows are infeasible for achieving non-interaction, since any two non-parallel lines must intersect. To circumvent this difficulty, 
temporal oscillations are exploited 
by Cheskidov and Luo \cite{cl21}, 
and by Giri and Radu \cite{gr24}. 
In the present work, 
we construct dissipative solutions 
by using {\it intermittent Mikado flows}. 
The key observation is that 
the intersection of two differently directed Mikado flows 
concentrates on a much smaller spatial support than either flow individually.  
As a result, the Mikado flows 
have strong {\it intersection estimate} 
(see Lemma \ref{product estimate lemma} below) 
which suffices to control the Reynolds oscillation error.

\subsection{Extension to global dissipative solutions}  
\label{Subsec-Ext-Global}

As mentioned in Subsection \ref{Subsec-Novel}, 
in order to construct global dissipative solutions, 
we will choose a refined background approximate profile 
with a new scale $\delta_2^\sharp$, 
the same scale as $\delta_3$, 
in place of the previous possibly large one $\delta_2$.

For a divergence-free background approximate solution
$U_0\in C([0,+\infty);L^2_x)\cap C^1((0,+\infty);C_x^{\infty})$ and a number
$\delta_2^\sharp>0$, we write
\begin{align}
\mathcal E_{\mathrm{ch}}[U_0,\delta_2^\sharp](v_0)
:=\Bigl\{e\in C([0,+\infty))\cap C^1((0,+\infty)):
&e(0)=\|v_0\|_{L^2_x}^2,
\notag\\
&e\text{ satisfies \eqref{asp-0-energy}, \eqref{asp-e-q}, and \eqref{bdd-e-c1}}
\notag\\
&\text{with }u_0\text{ replaced by }U_0
\text{ and }\delta_2\text{ replaced by }\delta_2^\sharp
\Bigr\}.
\label{def-Ech-sharp}
\end{align}
Let \begin{align}\label{def-t_*}
    T_*:=T_1+\ell_0=4^{-\theta}.
\end{align}

The following theorem provides 
an abstract criterion 
to construct global dissipative solutions. 
Its proof will be given in Section \ref{Sec-global-diss} below. 

\begin{theorem}[Abstract global dissipative criterion]
\label{thm-global-diss-abstract}
Assume the hypotheses of Theorem~\ref{Thm-Non-euler}. Let
$(u_0^\sharp,p_0^\sharp,\mathring R_0^\sharp)$ be a smooth solution of the
Euler-Reynolds system \eqref{equa-nsr} at level $q=0$, with
$u_0^\sharp(0)=v_0$, and let $E_0^\sharp(t):=\|u_0^\sharp(t)\|_{L^2_x}^2$.
Assume that
\begin{align}
u_0^\sharp(t)=u_0(t),
\qquad
\mathring R_0^\sharp(t)=\mathring R_0(t),
\qquad 0\le t\le T_* 
\label{gd-agree-full}
\end{align}
and that $E_0^\sharp$ is non-increasing on $[0,+\infty)$. Assume moreover that
there exists $\delta_2^\sharp$ such that
\begin{align}
&\delta_3\le \delta_2^\sharp\le \frac32\delta_3,
\label{gd-delta-scale}
\\
&\delta_1\ge \frac34\delta_2^\sharp,
\label{gd-delta1-abstract}
\\ 
&\frac{8}{c_*}\|\mathring R_0^\sharp\|_{C_{[T_1,+\infty)}L^1_x}
\le \delta_2^\sharp,
\label{gd-floor}
\\
&E_0^\sharp(\ot_2)-E_0^\sharp(\ut_1)
\ge 4(\delta_2^\sharp-\delta_3),
\label{gd-bridge-abstract}
\\
&\|u_0^\sharp\|_{C_{[T_2,+\infty),x}^{1}}\lesssim \lambda_0^{d+1},
\label{gd-u0-c1}
\\
&\|\mathring R_0^\sharp\|_{C^{1}_{[T_2,+\infty)}L^1_x}
+\|\mathring R_0^\sharp\|_{C_{[T_2,+\infty)}W^{1,1}_x}
\lesssim \lambda_0^{2d+2}.
\label{gd-r0-c1w11}
\end{align}
Then, for every $r\in [0,\tfrac12]$, there exists a global weak Euler solution
$u_r\in C([0,+\infty);L^2_x)$ with initial data $v_0$ such that
\begin{align}
\|u_r(t)\|_{L^2_x}^2=e_{\mathrm{gd},r}^\sharp(t),
\qquad t\ge 0,
\label{gd-energy-exact}
\end{align}
where $e_{\mathrm{gd},r}^\sharp\in
\mathcal E_{\mathrm{ch}}[u_0^\sharp,\delta_2^\sharp](v_0)$. In particular, there exists a global
dissipative weak Euler solution with initial data $v_0$. If
$\delta_2^\sharp>\delta_3$, then there are infinitely many distinct global
dissipative weak Euler solutions with initial data $v_0$.
\end{theorem}

In order to apply the abstract criterion 
to construct global dissipative solutions, 
we shall choose a smooth stabilized cutoff $\Lambda^{\mathrm{sf}}$ 
in the approximate $\Lambda^{\mathrm{sf}}$-Euler equations, 
and reduce the construction to verify the Euler-Reynolds system, 
the level-$0$ estimates 
and the floor as the same scale as $\delta_3$. See Theorem \ref{thm-sf-conditional} for the conditional reduction.

Finally, 
all these estimates are verified in Section \ref{Sec-sf-closure} 
with the new smooth stabilized cutoff 
in Section \ref{Sec-sf-closure} below.

\subsection{Final remark on the Sobolev case $W^{1,p}$} 
The case where $v_0\in W^{1,p}_x$ 
can be treated in an analogous way as above. 
The slight modifications are summarized as follows. 

For the Sobolev case, we similarly introduce
\begin{align}\label{def-dsp}
D_{n, 1, p}^{\sharp}:=\sup _{j \geq n} D_{j, 1, p}, \quad n \geq N .
\end{align}
Then $\{D_{n, 1, p}^{\sharp}\}_{n \geq N}$ is non-increasing and
$$
0 \leq D_{n, 1, p} \leq D_{n, 1, p}^{\sharp} \leq d .
$$
Moreover, \eqref{asp-int-2} implies
\begin{align}\label{1.11-2}
    |D_{n, 1, p}^{\sharp}-D_{m, 2, p}| \leq \frac{\kappa_* C_*}{p}, \quad n, m \geq N .
\end{align}
Indeed, this follows by applying \eqref{asp-int-2} to every
$j\geq n$ and then taking the supremum over $j$.

We choose the wavenumber $\Lambda:(0,+\infty) \to \mathbb{R}^+$ to still  satisfy properties 
{\rm (W1)}, {\rm (W2)}, 
while {\rm (W3)}  is replaced by 
\begin{itemize}
\item[{\rm (W3)'}] $\Lambda(t)=t^{-\eta}$ on $(0,\tfrac12]$, where
$\eta\ge2$ is an integer satisfying
\[
\frac1{4\alpha}<\eta<\frac1{2\alpha}, \qquad
0<\alpha< \min\left\{
\frac{\kappa_*(1-4C_*)}{4p},\frac13 \right\},
\]
and $\kappa_*$ is as in {\rm(H1)'}.
\end{itemize}

Now, for the substitution of \eqref{R0-0},  
we have the decay estimate of the initial Reynolds stress 
  \begin{align}\label{R0-0-2}
\|\mathring{R}_{0}(t)\|_{L^{1}_x}
\leq
C_R t^{\eta\left( 2-2\alpha + (d-D^\sharp_{n_*(t),1,p})
\left(1-\frac1p\right) - \frac{\kappa_*C_*}{p}
\right)}, \quad t\in(0,t_0).
\end{align}
Correspondingly, 
the amplitude parameter $\delta_{q+2}$ 
is chosen as in \eqref{def-delta}, 
but with $k_q$ replaced by 
\begin{align}
k_q := \eta\left( 2-2\alpha +(d-D^\sharp_{n_q,1,p})
\left(1-\frac1p\right) -\frac{\kappa_*C_*}{p}
\right).
\end{align}
where $n_q$ is given by \eqref{def-nq}.

It still holds that $\{\delta_{q+2}\}_{q \geq 1}$ is strictly decreasing in $q$ and $\delta_{q+2} \to 0^{+}$ as $q \to+\infty$. Indeed, since $T_q+\ell_{q-1} \leq t_0$, we have $n_q \geq N$. By \eqref{asp-reg-2} and \eqref{1.11-2},
$$
d-D_{n_q, 1, p}^{\sharp} \geq d-D_{n_q, 2, p}-\frac{\kappa_* C_*}{p} \geq \kappa_*(1-\frac{C_*}{p}).
$$
Therefore,
$$
(d-D_{n_q, 1, p}^{\sharp})(1-\frac{1}{p})-\frac{\kappa_* C_*}{p} \geq \kappa_*(\frac{1}{2}-C_*)>0 .
$$
Since $\eta>1 /(4 \alpha)$ and $\alpha<1 / 3$, it follows that
$$
k_q=\eta(2-2 \alpha+(d-D_{n_q, 1, p}^{\sharp})(1-\frac{1}{p})-\frac{\kappa_* C_*}{p})>1 .
$$
Hence, $\delta_{q+2} \to 0^{+}$as $q \to+\infty$.

Moreover, as above, $n_q \geq n_{q-1}$ for $q \geq 2$. Since $\{D_{n, 1, p}^{\sharp}\}_{n \geq N}$ is non-increasing,
\begin{align}\label{kq-kq1-2}
k_q=\eta (2-2 \alpha+ (d-D_{n_q, 1, p}^{\sharp} ) (1-\frac{1}{p} )-\frac{\kappa_* C_*}{p} ) \geq \eta (2-2 \alpha+ (d-D_{n_{q-1}, 1, p}^{\sharp} ) (1-\frac{1}{p} )-\frac{\kappa_* C_*}{p} )=k_{q-1} . 
\end{align}
Thus,
$$
 (T_q+\ell_{q-1} )^{k_q} \leq (T_q+\ell_{q-1} )^{k_{q-1}}
$$
and, using \eqref{Tq-def},
$$
\delta_{q+2}=\frac{8 C_R}{c_*} (T_q+\ell_{q-1})^{k_q}<\frac{8 C_R}{c_*}(4(T_q+\ell_{q-1}))^{k_{q-1}}=\delta_{q+1}\ 
\qquad q\geq 2. 
$$

\section{Approximate \texorpdfstring{$\Lambda$}{Lambda}-Euler equations} \label{Sec-Lam-Euler}

In this section, 
we prove Theorem \ref{thm-solu-le} 
and Proposition \ref{Prop-R0} concerning the 
approximate $\Lambda$-Euler equations 
and the corresponding initial Reynolds stress, 
respectively. 

Let us first prove that 
the propagation operator $\mathcal{S}(t)$ 
is uniformly bounded on $L^p_x$ 
for $p\in [1,+\infty]$.

\begin{lemma}[$L^p$-bound of propagation operator]\label{bd-m} 
For any $1\le p\le\infty$ and $n\ge1$, there exists $C_{p,d}>0$, independent of
$n$, $t$, and the later parameter $\theta$, such that for any $t\in(0,+\infty)$,
\begin{align}  \label{St-Lp-bound}
  \|\mathcal {S}(t) \P_{\geq \Lambda(t)}^2\Delta_n f\|_{L^p_x} 
  \leq C_{p,d} \|\Delta_n f\|_{L^p_x} 
\end{align}  
for all mean-free functions $f\in L^p_x$. 
\end{lemma}

\begin{proof}
Let $\Psi:\R^d \rightarrow [0,1]$ be a smooth radial function such that
\begin{align}\label{def-Psi-lem31}
	\Psi(x)=\begin{cases}
		0,\quad |x|\in [0,1/4]\cup [4,+\infty),\\
		1,\quad |x|\in [1/2,2], \\
	\end{cases}
\end{align}  
and $\Psi_{n}(x):= \Psi(x/2^{n})$.
In view of Young's inequality
\begin{align}\label{est-snv}
    \| \mathcal{S}(t)\P_{\geq \Lambda(t)}^2\Delta_nf \|_{ L^p_x}
    = \| \mathcal{F}^{-1}(m(t)\varphi_{\Lambda(t)}^2 \Psi_n \chi_n \widehat f)\|_{L^p_x}\leq \| \mathcal{F}^{-1}(m(t)\varphi_{\Lambda(t)}^2\Psi_n)\|_{L^1_x}\|\Delta_nf\|_{L^p_x},
\end{align}
it suffices to prove that $\mathcal{F}^{-1}(m(t)\varphi_{\Lambda(t)}^2\Psi_{n})$ is bounded in $L^1_x$, uniformly with respect to $t$. 

By the Poisson summation formula, 
we have the following identity 
relating the Fourier series and the Fourier transform, 
\begin{align}\label{tr-tran}
    \mathcal{F}^{-1}(m(t)\varphi^2_{\Lambda(t)}\Psi_n)(x)= \sum_{k\in \mathbb{Z}^d} \mathcal{F}^{-1}_{\R^d}(m(t)\varphi^2_{\Lambda(t)}\Psi_n)(x+k)\quad \forall x\in \T^d,
\end{align}
where $\mathcal{F}_{\R^d}^{-1} f(x):= \int_{\R^d} e^{2\pi ix\cdot \xi} f(\xi) \d \xi$.
This yields that 
\begin{align}\label{mvp-l1}
    \| \mathcal{F}^{-1}(m(t)\varphi^2_{\Lambda(t)}\Psi_n)\|_{L^1_x} 
    \leq \sum_{k\in \mathbb{Z}^d}\| \mathcal{F}^{-1}_{\R^d}(m(t)\varphi^2_{\Lambda(t)}\Psi_n)(x+k)\|_{L^1_x} \leq \| \mathcal{F}^{-1}_{\R^d}(m(t)\varphi^2_{\Lambda(t)}\Psi_n)\|_{L^1(\R^d)}. 
\end{align} 
Note that
\begin{align}\label{mvp-l1-2}
    \| \mathcal{F}^{-1}_{\R^d}(m(t)\varphi^2_{\Lambda(t)}\Psi_n)\|_{L^1(\R^d)}
    &=2^{dn} \| \mathcal{F}^{-1}_{\R^d}(m(t, 2^n\cdot)\varphi^2_{\Lambda(t)}(2^n\cdot)\Psi(\cdot))(2^n \cdot )\|_{L^1(\R^d)}\notag\\
    &= \| \mathcal{F}^{-1}_{\R^d}(m(t, 2^n\cdot)\varphi^2_{\Lambda(t)}(2^n\cdot)\Psi(\cdot))\|_{L^1(\R^d)}.
\end{align}

Let
\begin{align*}
    b_{n,t}(\zeta):=m(t,2^n\zeta)\varphi_{\Lambda(t)}^2(2^n\zeta)\Psi(\zeta).
\end{align*}
Since $\Psi$ is compactly supported, 
it is enough to prove that
\begin{align}\label{symbol-w11}
\|b_{n,t}\|_{W^{d+1,1}(\R^d)}\lesssim 1
\end{align}
uniformly in $n\geq 1$ and $t>0$, because then
\begin{align}\label{sob-kernel}
\|\mathcal{F}^{-1}_{\R^d}(b_{n,t})\|_{L^1(\R^d)}\lesssim \|b_{n,t}\|_{W^{d+1,1}(\R^d)}\lesssim 1.
\end{align}

Choose $n_0\in\mathbb N$ such that
\begin{align}\label{def-n0-bd}
2^{n_0-2}\ge 2\Lambda(1/2).
\end{align}
We distinguish two regimes.

\medskip
\noindent{\it Low shells: $1\le n<n_0$.}
For every fixed $n<n_0$, let
\[
\Gamma_n:=\{\xi\in\mathbb Z^d:\chi_n(\xi)\neq0\}.
\]
Since $\Gamma_n$ is finite, for $g:=\Delta_n f$ we have
\begin{align*}
\mathcal S(t)\P_{\geq\Lambda(t)}^2g(x)
=\sum_{\xi\in\Gamma_n}
m(t,\xi)\varphi_{\Lambda(t)}^2(\xi)\widehat g(\xi)e^{2\pi i x\cdot\xi}.
\end{align*}
Using $|m(t,\xi)\varphi_{\Lambda(t)}^2(\xi)|\le1$, we obtain
\begin{align*}
\|\mathcal S(t)\P_{\geq\Lambda(t)}^2g\|_{L^p_x}
&\le \|\mathcal S(t)\P_{\geq\Lambda(t)}^2g\|_{L^\infty_x}
\le \sum_{\xi\in\Gamma_n}|\widehat g(\xi)|  \\
&\le |\Gamma_n|\|g\|_{L^1_x}
\le |\Gamma_n|\|g\|_{L^p_x}.
\end{align*}
Since there are only finitely many integers $1\le n<n_0$, and $n_0$ is fixed independently of
$\theta$, we obtain \eqref{St-Lp-bound} for all low shells after enlarging $C_{p,d}$ if necessary.

\medskip
\noindent{\it High shells: $n\ge n_0$.}
In this case, for $\xi\in\supp\Psi_n$,
\begin{align}\label{high-shell-support}
|\xi|\ge 2^{n-2}\ge 2\Lambda(1/2).
\end{align}

Note that $g(\xi):= |\xi|^{2\alpha}$ is  homogeneous of degree $2\alpha$. 
Then, 
by Euler's Homogeneous Function Theorem (see, e.g., \cite[p287]{Apostol}), 
$\nabla^N g$ is homogeneous of degree $2\alpha-N$ 
for any $N\in \mathbb{N}^+$. 
This yields that,  
if $\xi:=|\xi|\omega$ with $\omega\in \mathbb{S}^{d-1}$, 
\begin{align*}
    (\nabla^N g)(\xi)= (\nabla^N g)(|\xi|\omega)=|\xi|^{2\alpha-N } (\nabla^N g)(\omega).
\end{align*}
Because $(\nabla^N g)(\omega)$ is uniformly bounded for all $\omega\in \mathbb{S}^{d-1}$, 
	it follows that
\begin{align}\label{est-f-a}
    |\nabla^N ( |\xi|^{2\alpha})|& = | |\xi|^{ 2\alpha-N} (\nabla^N g)(\omega) |\lesssim |\xi|^{ 2\alpha-N}.
\end{align}

For $t>\tau(\xi)$, 
Let
\begin{align*}
    A(t,\xi):=\int_{\tau(\xi)}^t\varphi^2_{\Lambda(s)}(\xi)\d s.
\end{align*}
Since $\varphi$ is smooth and identically zero on $\{|x|\leq 1\}$, 
all derivatives of $\varphi$ vanish on $\{|x|=1\}$.
Hence every boundary term generated by differentiating the lower limit $\tau(\xi)$ vanishes at $s=\tau(\xi)$.
Therefore, for every $N\in \mathbb{N}$,
\begin{align*}
     \nabla^N A(t,\xi)&=\int_{\tau(\xi)}^t \nabla^N \big(\varphi^2_{\Lambda(s)}(\xi)\big) \d s.
\end{align*}
It follows that for $N\geq 1$,
 \begin{align}\label{est-N}
\big|\nabla^N A(t,\xi)\big|
&\lesssim \int_{\tau(\xi)}^t \Lambda^{-N}(s)
\sum_{N_{1}+N_{2}=N}\big|(\nabla^{N_{1}}\varphi)(\xi/\Lambda(s))\big|
\big|(\nabla^{N_{2}}\varphi)(\xi/\Lambda(s))\big| \d s\notag\\
&\lesssim |\xi|^{-N}\int_{\tau(\xi)}^{\Lambda^{-1}(|\xi|/2)} \sum_{N_1+N_2=N}
\big|(\xi/\Lambda(s))^N(\nabla^{N_{1}}\varphi)(\xi/\Lambda(s))(\nabla^{N_{2}}\varphi)(\xi/\Lambda(s))\big| \d s\notag\\
&\lesssim |\xi|^{-N}(\Lambda^{-1}(|\xi|/2)-\tau(\xi) ) \notag\\
&= |\xi|^{-N}(2^{\frac1\eta}\tau(\xi)-\tau(\xi))
\lesssim \tau(\xi)|\xi|^{-N}\lesssim  |\xi|^{-N-\frac{1}{\eta}}, 
\end{align} 
where we used \eqref{high-shell-support} to ensure that
\begin{align*}
\Lambda^{-1}(|\xi|/2)\le \frac12,
\end{align*}
so property {\rm (W3)} applies on
\begin{align*}
[\tau(\xi),\Lambda^{-1}(|\xi|/2)]
=[\tau(\xi),2^{\frac1\eta}\tau(\xi)].
\end{align*}
The implicit constant is independent of $t$.

Let
\[
G(t,\xi):=A(t,\xi)(2\pi|\xi|)^{2\alpha}, \qquad
B(\xi):=|\xi|^{2\alpha-\frac1\eta}.
\]
For every multi-index $\beta$ with $|\beta|=N\geq1$, estimates
\eqref{est-f-a} and \eqref{est-N} give
\begin{align}\label{est-G-der}
|\partial_\xi^\beta G(t,\xi)| \lesssim
\bigl(G(t,\xi)+B(\xi)\bigr)|\xi|^{-N}.
\end{align}
Indeed, the $G$ term is the contribution in which no derivative falls on $A$;
every other contribution contains at least one derivative of $A$.

By the Faà di Bruno formula, every term in
$\partial_\xi^\beta e^{-\nu G}$ has the form
\[
e^{-\nu G}
\prod_{j=1}^m\partial_\xi^{\beta_j}G, \qquad
1\leq m\leq N,\quad \beta_j\neq0,\quad
\sum_{j=1}^m\beta_j=\beta.
\]
After applying \eqref{est-G-der}, a term containing $r$ factors of $B$ is bounded by
\[
e^{-\nu G}G^{m-r}B^r|\xi|^{-N}
\lesssim |\xi|^{-N}.
\]
Here $e^{-\nu G}G^{m-r}\lesssim1$, while
$B(\xi)\leq1$ on the high shells because
$2\alpha-\frac1\eta<0$. Consequently,
\begin{align}\label{est-Ne}
|\nabla^Nm(t,\xi)|
\lesssim|\xi|^{-N},
\qquad N\geq1,\quad t>\tau(\xi),
\end{align}
uniformly in $t$.
	
We also note that 
since $\nabla^N \varphi$ is supported in the annulus
$\{x\in \R^d: 1<|x|<2\}$ for $N\geq1$, 
\begin{align}\label{est-phi-2}
    |\nabla^N \varphi^2_{\Lambda(t)}(\xi)|\leq 
    \Lambda^{-N}(t)  \sum_{N_{1}+N_{2}=N}| (\nabla^{N_{1}}\varphi)(\frac{
    \xi}{\Lambda(t)})(\nabla^{N_{2}}\varphi)(\frac{
    \xi}{\Lambda(t)})|\lesssim |\xi|^{-N},\quad N\geq 1,
\end{align}
where the implicit constant is independent of $t$. 

Thus, 
combining \eqref{est-Ne} and \eqref{est-phi-2} we arrive at
\begin{align}\label{est-n-mv}
    |\nabla^N (m(t,\xi)\varphi^2_{\Lambda(t)}(\xi))|&\leq |\nabla^N (m(t,\xi))\varphi^2_{\Lambda(t)}(\xi)|+|m(t,\xi)\nabla^N (\varphi^2_{\Lambda(t)}(\xi))|\notag\\
&\quad+\sum_{N_{1}+N_{2}=N,\, N_{1}N_{2}\neq 0} |\nabla^{N_1} (m(t,\xi))||\nabla^{N_2} (\varphi^2_{\Lambda(t)}(\xi))|\notag\\
    &\lesssim |\xi|^{-N},
\end{align}
for all $t>\tau(\xi)$, 
uniformly in $t$.

Now, applying \eqref{est-n-mv} and the chain rule on the support of $\Psi$, 
we obtain for any multi-index $\beta$ with $|\beta|\leq d+1$,
\begin{align}\label{high-deri-b}
\|\partial^\beta_\zeta b_{n,t}\|_{L^\infty_\zeta}
\lesssim 1, \qquad n\geq n_0,\ \ t>0.
\end{align}
Indeed, if $t\leq \tau(2^n\zeta)$, then $b_{n,t}(\zeta)=0$, while if $t>\tau(2^n\zeta)$, then each term in the chain rule has the form
\[
2^{n|\beta_1|}\big|\nabla^{|\beta_1|}(m(t,\xi)\varphi_{\Lambda(t)}^2(\xi))\big|_{\xi=2^n\zeta}
\big|\partial^{\beta_2}\Psi(\zeta)\big|\lesssim 2^{n|\beta_1|}(2^n)^{-|\beta_1|}\lesssim 1,
\qquad \beta_1+\beta_2=\beta.
\]
Since $\Psi$ is compactly supported, \eqref{high-deri-b} yields
\begin{align}\label{high-shell-bound}
\|b_{n,t}\|_{W^{d+1,1}(\R^d)}\lesssim 1,
\qquad n\geq n_0,\ \ t>0.
\end{align}

Therefore,  
combining \eqref{high-shell-bound}, \eqref{sob-kernel}, \eqref{mvp-l1} and
\eqref{mvp-l1-2}, we derive that
$\mathcal{F}^{-1}(m(t)\varphi^2_{\Lambda(t)}\Psi_n)$
is bounded in $L^1_x$ uniformly with respect to $t$ for all high shells.
Together with the low-shell estimate above and \eqref{est-snv},
we consequently obtain \eqref{St-Lp-bound} and finish the proof.
\end{proof}

Next we prove Theorem \ref{thm-solu-le} concerning the approximate $\Lambda$-Euler equations. 

\medskip 
\paragraph{\bf Proof of Theorem \ref{thm-solu-le}}  
We recall that $\wt u$ has the explicit expression \eqref{def-solu} 
in the Fourier modes. 

Let us first treat the spatial regularity of $\wt u$.
Let $t'>0$ and $t=t'/2$. If $|\xi|>2\Lambda(t)$, then
$\tau(\xi)<t$ and, since $\Lambda$ is decreasing,
\[
\varphi_{\Lambda(r)}(\xi)=1,
\qquad t\leq r\leq t'.
\]
Consequently,
\[
|m(t',\xi)|^2 \leq
e^{-2\nu(t'-t)(2\pi|\xi|)^{2\alpha}} =
e^{-\nu t'(2\pi|\xi|)^{2\alpha}}.
\]
Using also $|m(t',\xi)|\leq1$, we obtain, for every $s\geq0$,
\begin{align}\label{est-u-hs}
\|\widetilde u(t')\|_{\dot H^s_x}^2
&\lesssim \sum_{|\xi|\leq2\Lambda(t)}
|\xi|^{2s}|\widehat{\widetilde u}(t',\xi)|^2
+\sum_{|\xi|>2\Lambda(t)}
|\xi|^{2s}e^{-\nu t'(2\pi|\xi|)^{2\alpha}} |\widehat v_0(\xi)|^2
\notag\\
&\lesssim \left( \Lambda(t'/2)^{2s}
+(t')^{-\frac{s}{\alpha}} \right)\|v_0\|_{L^2_x}^2
\notag\\
&\lesssim \left( 1+(t')^{-\frac{s}{\alpha}} \right)\|v_0\|_{L^2_x}^2.
\end{align}
Indeed, for $0<t'\leq1$, property {\rm (W3)} and
$\eta<1/(2\alpha)$ give
\[
\Lambda(t'/2)^{2s}\lesssim(t')^{-\frac{s}{\alpha}},
\]
whereas $\Lambda(t'/2)$ is uniformly bounded for $t'\geq1$.
Hence, for every $t_0>0$,
\begin{equation}\label{spa-regu}
\|\widetilde u\|_{C([t_0,+\infty);\dot H^s_x)}
\lesssim \left( 1+t_0^{-\frac{s}{2\alpha}}
\right)\|v_0\|_{L^2_x}.
\end{equation}

Regarding the first temporal derivative of $\wt u$, 
by equation \eqref{equa-nse-2} and \eqref{spa-regu}, 
we get
\begin{align}\label{tem-regu-1}
\|\p_t \wt u\|_{C([t,+\infty);\dot{H}^s_x)}
& \leq \nu\|(-\Delta)^\alpha \P^2_{\geq \Lambda(\cdot)} \wt u\|_{C([t,+\infty);\dot{H}^s_x)}
+ \|\div( \wt u\otimes \wt u)\|_{C([t,+\infty);\dot{H}^s_x)}\notag\\
&\lesssim \| \wt u \|_{C([t,+\infty);\dot{H}^{s+2\alpha}_x)}+ \| \wt u\otimes \wt u\|_{C([t,+\infty);{H}^{s+1}_x)}\notag\\
&\lesssim \| \wt u \|_{C([t,+\infty);\dot{H}^{s+2\alpha}_x)}+ \| \wt u \|_{C([t,+\infty);{H}^{s+d}_x)}^2\notag\\
&\lesssim (1+t^{-\frac{s+2\alpha}{2\alpha}})\|v_0\|_{L^2_x}+(1+t^{-\frac{s+d}{\alpha}})\|v_0\|_{L^2_x}^2  \notag \\
&\lesssim (1+ t^{-\frac{s+d}{\alpha}}) (\|v_0\|_{L^2_x}+\|v_0\|_{L^2_x}^2) 
  <\infty
\end{align} 
for any $t>0$, 
where the implicit constants are universal.

We next prove the higher temporal regularity of $\wt u$ directly from \eqref{def-solu}. 
Fix
\begin{align*}
	0<t_0<+\infty, \qquad t_1:=\frac{t_0}{2},
\end{align*}
and split $\wt u=\wt u_{\rm low}+\wt u_{\rm high}$ by
\begin{align*}
	\widehat{\wt u}_{\rm low}(t,\xi)
	&:=1_{\{|\xi|\leq 2\Lambda(t_1)\}}\widehat{\wt u}(t,\xi),\\
	\widehat{\wt u}_{\rm high}(t,\xi)
	&:=1_{\{|\xi|> 2\Lambda(t_1)\}}\widehat{\wt u}(t,\xi).
\end{align*}
Since the set $\{\xi\in \mathbb{Z}^d: |\xi|\leq 2\Lambda(t_1)\}$ is finite, 
it suffices to study each Fourier mode separately for $\wt u_{\rm low}$. 
For every fixed $\xi$, the coefficient in \eqref{def-solu} is smooth on $[t_0,+\infty)$: 
away from $t=\tau(\xi)$ this is immediate, while at $t=\tau(\xi)$ the two pieces match to all orders because $\Lambda$ and $\varphi$ are smooth and $\varphi$ vanishes on $\{|x|\leq 1\}$. 
Hence,
\begin{align}\label{tem-regu-low}
	\wt u_{\rm low}\in C^\infty([t_0,+\infty)\times \mathbb{T}^d).
\end{align}

For the high-frequency part, if $|\xi|>2\Lambda(t_1)$ and $t\in [t_0,+\infty)$, then $\tau(\xi)<t_1$, and, since $\Lambda$ is decreasing,
\begin{align*}
	\varphi_{\Lambda(t)}(\xi)=1,\qquad t\in [t_1,+\infty).
\end{align*}
Therefore, using \eqref{def-solu}, we obtain
\begin{align}\label{def-solu-high}
	\widehat{\wt u}_{\rm high}(t,\xi)
	=e^{-\nu(t-t_1)(2\pi|\xi|)^{2\alpha}}\widehat{\wt u}(t_1,\xi),
	\qquad t\in [t_0,+\infty).
\end{align}
It follows that for any $m\in \mathbb{N}$ and $s\geq 0$,
\begin{align}\label{tem-regu-high}
	\|\p_t^m \wt u_{\rm high}(t)\|_{\dot{H}^s_x}^2
	&\lesssim \sum_{|\xi|>2\Lambda(t_1)} |\xi|^{2s+4\alpha m} e^{-2\nu(t-t_1)(2\pi|\xi|)^{2\alpha}} |\widehat{\wt u}(t_1,\xi)|^2\notag\\
	&\leq C_{m,s,t_0}\sum_{|\xi|>2\Lambda(t_1)} |\widehat{\wt u}(t_1,\xi)|^2\notag\\
	&\leq C_{m,s,t_0}\|v_0\|_{L^2_x}^2,
\end{align}
since $t-t_1\geq t_0/2>0$ on $[t_0,+\infty)$ and $\mathcal{S}(t_1)$ is bounded on $L^2_x$. 
Moreover, \eqref{def-solu-high} shows that $\p_t^m \wt u_{\rm high}\in C([t_0,+\infty);\dot{H}^s_x)$ for all $m\in \mathbb{N}$ and $s\geq 0$.
Combining this with \eqref{tem-regu-low}, we infer that
\begin{align*}
\p_t^m \wt u\in C([t_0,+\infty);H^s_x),
	\qquad m\in \mathbb{N},\ s\geq 0.
\end{align*}
Since $0<t_0<+\infty$ are arbitrary, 
$\wt u$ is smooth for positive times.

The energy balance \eqref{eq-e-2} 
follows easily from the skew-symmetry of the nonlinearity.

Concerning the $L^2_x$ continuity \eqref{s-con-v0},  
for any $\xi\in \mathbb{Z}^d$,  
since $\tau(\xi)\in (0,+\infty]$ and
$\widehat{\wt u}(t,\xi)=\widehat{v}_0(\xi)$ for every $0<t\leq \tau(\xi)$,
we infer that
\begin{align*}
\widehat{\wt u}(t,\xi)-\widehat{v}_0(\xi)\to 0,
\quad {\rm as}\ t\to 0^+.
\end{align*}
Moreover,
\begin{align*}
\big|\widehat{\wt u}(t,\xi)-\widehat v_0(\xi)\big|^2
\leq 4|\widehat v_0(\xi)|^2,
\end{align*}
and
\begin{align*}
\sum_{\xi\in\mathbb{Z}^d} |\widehat v_0(\xi)|^2=\|v_0\|_{L^2_x}^2<\infty.
\end{align*}
Thus, by the Lebesgue dominated convergence theorem we derive 
\begin{align*}
  \| {\wt u}(t)-{v}_0 \|^2_{L^2_x}
  =\sum\limits_{\xi\in \mathbb{Z}^d} 
   |\widehat{\wt u}(t,\xi)-\widehat v_0(\xi)|^2 
   \to 0,\quad \text{as}\quad t\to 0^+. 
\end{align*}
Therefore, 
the proof of Theorem \ref{thm-solu-le} is complete. 
\hfill $\square$

\medskip 
Then, 
we prove Proposition \ref{Prop-R0} 
for the initial Reynolds stress. 

\paragraph{\bf Proof of Proposition \ref{Prop-R0}}  
By \eqref{def-alpha}, we have $2\alpha<1$. Since $\Lambda(t)\geq 2$ by {\rm (W2)},
the Fourier multiplier of $\mathcal{R} (-\Delta)^\alpha \P_{\geq\Lambda(t)}^2$
is bounded by
\[
|\xi|^{2\alpha-1}1_{\{|\xi|\geq \Lambda(t)\}}
\leq \Lambda^{2\alpha-1}(t)\leq 2^{2\alpha-1}\lesssim 1.
\]
Therefore,
\begin{align}  \label{R0-L1}
    \|\mathring{R}_{0}(t)\|_{L^{1}_x} &\lesssim \|\mathcal{R} (-\Delta)^\alpha \P_{\geq\Lambda(t)}^2 \wt u (t)\|_{ L^{2}_x} \lesssim \|\wt u(t)\|_{L^2_x} \leq  C_0\|v_0\|_{L^2_x}, 
\end{align} 
where the last step follows from the $L^2$-boundedness of the propagation operator $\mathcal{S}$, 
and $C_0$ is a universal constant.

Regarding the decay estimate,  
we infer from Lemma~\ref{bd-m}, \eqref{def-solu} and \eqref{r0} that for any $t\in (0,t_0]$, 
\begin{align} \label{est-r0-linfty}
	\|\mathring{R}_{0}(t)\|_{L^{1}_x} &\lesssim \|\mathcal{R} (-\Delta)^\alpha \P_{\geq\Lambda(t)}^2\wt u (t)\|_{ L^1_x} \notag\\
       &\lesssim  \Lambda^{-1+2\alpha}(t)\sum_{n\geq n_*} \| \mathcal{S}(t)\P_{\geq\Lambda(t)}^2\Delta_n v_0\|_{ L^1_x} \notag\\
	       &\lesssim  \Lambda^{-1+2\alpha}(t)\sum_{n\geq n_*}\| \Delta_n v_0 \|_{ L^1_x}, 
	\end{align}
	where $n_* = n_*(t)$ is such that $2^{n_*}\leq \Lambda(t)<2^{n_*+1}$.

Then, using the characterization of intermittency in terms of Bernstein's inequality \eqref{ber-ineq}, we get
\begin{align*}
\Lambda^{-1+2 \alpha}(t) \sum_{n \geq n_*}\left\|\Delta_n v_0\right\|_{L_x^1} & \lesssim \Lambda^{-1+2 \alpha}(t) \sum_{n \geq n_*} 2^{-\left(d-D_{n, 1, \infty}\right) n}\left\|\Delta_n v_0\right\|_{L_x^{\infty}} \\
& \lesssim \Lambda^{-1+2 \alpha}(t) \sum_{n \geq n_*} 2^{-\left(d-D_{n_*, 1, \infty}^{\sharp}\right) n}\left\|\Delta_n v_0\right\|_{L_x^{\infty}} .
\end{align*}
Here we used
$$
D_{n, 1, \infty} \leq D_{n_*, 1, \infty}^{\sharp}, \quad n \geq n_* .
$$
Let
$$
A_*(t):=d-D_{n_*(t), 1, \infty}^{\sharp}+1-\kappa_0 .
$$
Since $D_{n_*(t), 1, \infty}^{\sharp} \leq d$ and $\kappa_0<1$, we have
$$
A_*(t) \geq 1-\kappa_0>0 .
$$
Consequently, by the Cauchy-Schwarz inequality,
\begin{align*}
\Lambda^{-1+2 \alpha}(t) \sum_{n \geq n_*} 2^{- (d-D^{\sharp}_{n_*, 1, \infty}) n}\|\Delta_n v_0\|_{L_x^{\infty}} & =\Lambda^{-1+2 \alpha}(t) \sum_{n \geq n_*} 2^{-A_*(t) n} 2^{(1-\kappa_0) n}\|\Delta_n v_0\|_{L_x^{\infty}} \\
& \lesssim \Lambda^{-1+2 \alpha}(t) 2^{-A_*(t) n_*}(\sum_{n \geq n_*} 2^{2(1-\kappa_0) n}\|\Delta_n v_0\|_{L_x^{\infty}}^2)^{1 / 2} \\
& \lesssim \Lambda^{-2+2 \alpha-(d-D_{n_*, 1, \infty}^{\sharp})+\kappa_0}(t)(\sum_{n \geq n_*} 2^{2(1-\kappa_0) n}\|\Delta_n v_0\|_{L_x^{\infty}}^2)^{1 / 2} .
\end{align*}

Note that, since $v_0\in C^{1-\kappa}_x=B^{1-\kappa}_{\infty,\infty}$ and $\kappa<\kappa_0< 1$, by
the standard Besov embedding,
$B^{1-\kappa}_{\infty,\infty}\hookrightarrow B^{1-\kappa_0}_{\infty,2}$.
Hence
\begin{align}\label{asp-reg}
	\sum_{n=0}^\infty (2^{(1-\kappa_0)n} \|\Delta_n v_0\|_{L^\infty_x})^2 =\|v_0\|_{B^{1-\kappa_0}_{\infty,2}}^2\lesssim \|v_0\|_{C^{1-\kappa}_x}^2< \infty.
\end{align} 

Taking into account the property 
{\rm (W3)} of the wavenumber  
we thus obtain  \eqref{R0-0} 
and finish the proof.   
\hfill $\square$

\begin{remark}
We note that in the case where $v_0\in W^{1,p}_x$, 
estimate \eqref{est-r0-linfty} is replaced by
\begin{align}\label{est-r0-lp}
\|\mathring{R}_0(t) \|_{L_x^1} & \lesssim \Lambda^{-1+2 \alpha}(t) \sum_{n \geq n_*} 2^{- (d-D_{n, 1, p} ) (1-\frac{1}{p} ) n} \|\Delta_n v_0 \|_{L_x^p} \notag\\
& \lesssim \Lambda^{-1+2 \alpha}(t) \sum_{n \geq n_*} 2^{- (d-D_{n_*, 1, p}^{\sharp} ) (1-\frac{1}{p} ) n} \|\Delta_n v_0 \|_{L_x^p}\notag\\
& \lesssim \Lambda^{-2+2 \alpha- (d-D_{n_*, 1, p}^{\sharp} ) (1-\frac{1}{p} )+\frac{\kappa_* C_*}{p}}(t) (\sum_{n \geq n_*} 2^{2 (1-\frac{\kappa_* C_*}{p} ) n} \|\Delta_n v_0 \|_{L_x^p}^2 )^{1 / 2}
\end{align}
in the proof of Proposition \ref{Prop-R0}, where $\kappa_* \in(0, d]$ and $C_* \in\left(0, \frac{1}{4}\right)$ are positive constants given by {\rm (H1)'}.
Indeed, for $n\geq n_*$,
$$
D_{n, 1, p} \leq D_{n_*, 1, p}^{\sharp}
$$ 
and 
$$
(d-D_{n_*, 1, p}^{\sharp})(1-\frac{1}{p})+1-\frac{\kappa_* C_*}{p}>0,
$$
so the last inequality of \eqref{est-r0-lp} follows from the Cauchy-Schwarz inequality, exactly as in the $C^{1-\kappa}$ case.

Moreover, \eqref{asp-reg} is replaced by
\begin{align}\label{asp-reg-3}
\sum_{n=0}^\infty (2^{(1-\frac{\kappa_*C_*}{p})n} \|\Delta_n v_0\|_{L^{p}(\T^d)})^2 \lesssim  \|v_0\|_{W^{1,p}_x}^2< \infty.
\end{align}
\end{remark}

\section{Velocity fields} \label{Sec-Interm-Flow}

In this section, 
we use the Mikado flows introduced in \cite{DS2017} 
to construct the velocity fields at level $q+1$ 
and verify the corresponding inductive estimates. 
The building blocks 
will be indexed by the following three parameters
\begin{equation}\label{larsrp}
\rs := \lambda_{q+1}^{-\frac67},  \quad 
\lambda := \lambda_{q+1},  \quad  \sigma=\lambda_{q+1}^{\frac17}.
\end{equation}

Let us recall from \cite{DS2017} the following Geometric Lemma. 
\begin{lemma} [Geometric Lemma, \cite{DS2017}, Lemma 4.2]
	\label{geometric lem 2}
	There exist a finite set $\mathcal{K} \subseteq \mathbb{S}^{d-1} \cap \mathbb{Q}^d$ and smooth positive functions
    $\gamma_{(k)}: B_{\varepsilon_u}(\Id) \to \mathbb{R}$,
	where $B_{\varepsilon_u}(\Id)$ is the ball of radius $\varepsilon_u(>0)$ centered at the identity
	in the space of $d \times d$ symmetric matrices,
	such that for  $S \in B_{\varepsilon_u}(\Id)$ we have the following identity:
	\begin{equation}
		\label{sym}
		S = \sum_{k \in \mathcal{K}} \gamma_{(k)}^2(S) k \otimes k.
	\end{equation}
\end{lemma}
Let $\{k,k_1,...,k_{d-1}\}$ be an orthonormal basis associated with $k$.
Note that, there exists $N_{\mathcal{K}} \in \mathbb{N}$ such that
\begin{equation*}
	\{ N_{\mathcal{K}} k,N_{\mathcal{K}}k_1,...,
    N_{\mathcal{K}}k_{d-1} \} \subseteq N_{\mathcal{K}} \mathbb{S}^{d-1} \cap \mathbb{Z}^d.
\end{equation*}
If necessary, identify antipodal pairs in $\mathcal K$, preserving \eqref{sym}, by replacing the
coefficients of $k$ and $-k$ with
\begin{align*}
\widetilde\gamma_{(k)}(S)
:=\bigl(\gamma_{(k)}^2(S)+\gamma_{(-k)}^2(S)\bigr)^{1/2}.
\end{align*}
After relabeling, and renaming $\widetilde\gamma$ as $\gamma$, no two distinct elements of
$\mathcal K$ are antipodal. For each $k\in\mathcal K$, let
\begin{align*}
Q_k=
\begin{cases}
\Id-2\dfrac{(e_1-k)\otimes(e_1-k)}{|e_1-k|^2},& k\ne e_1,\\[1ex]
\Id,& k=e_1 .
\end{cases}
\end{align*}
Then $Q_k\in O(d)\cap M_d(\mathbb Q)$ and $Q_ke_1=k$. Let $\{k,k_1,\ldots,k_{d-1}\}$ be the columns of $Q_k$. Since $\mathcal K$ is finite, there exists
$N_{\mathcal K}\in\mathbb N$ such that
\begin{equation}\label{NLambda}
\{N_{\mathcal K}k,N_{\mathcal K}k_1,\ldots,
N_{\mathcal K}k_{d-1}\}
\subset\mathbb Z^d
\qquad\text{for every }k\in\mathcal K.
\end{equation}

In dimension $d=2$, no two distinct directions are parallel, and hence
\begin{align*}
c_{\mathcal K}:=\min_{k\ne k'\in\mathcal K}|k^\perp\cdot k'|>0 .
\end{align*}

Let $M_*$ denote the geometric constant such that
\begin{align}	\label{M bound}
	\sum_{k\in\mathcal K} \norm{\gamma_{(k)}}_{C^4(B_{\varepsilon_u}(\Id))} \leq  M_*.
\end{align}

\subsection{Mikado flows}
Let $\Phi : \mathbb{R}^{d-1} \to \mathbb{R}$ be any smooth cut-off function supported in a ball with center at $1/2$ 
and radius sufficiently small. 
We normalize $\Phi$ such that $\phi := - \Delta\Phi$  
has the unit $L^2_x$-norm 
\begin{equation}\label{e4.91}
\int_{\T^{d-1}} \phi^2(x)\d x = 1.
\end{equation}
Next, we rescale $\phi$ and $\Phi$ by
\begin{equation*}
	\phi_{\rs}(x) := \rs^{-\frac{d-1}{2}}\phi\left(\frac{x}{\rs}\right), \quad
	\Phi_{\rs}(x):=    \rs^{-\frac{d-1}{2}}\Phi\left(\frac{x}{\rs}\right).
\end{equation*}
By an abuse of notation, let $\phi_{\rs}$ and $\Phi_{\rs}$ also denote their
periodic extensions to $\mathbb{T}^{d-1}$. Since $b\in 7\mathbb{N}$, we have
$\sigma=\lambda_{q+1}^{\frac17}\in \mathbb{N}$, and thus \eqref{NLambda}
implies
\begin{equation*}
    \sigma N_{\mathcal K}k_j\in \mathbb{Z}^d,\qquad j=1,\dots,d-1.
\end{equation*}
Hence the functions introduced below are well defined and periodic on
$\mathbb{T}^d$.

The {\it intermittent Mikado flows} are defined by
\begin{equation*}
	W_{(k)}(x) :=  \phi_{\rs} ( \sigma  N_{\mathcal K}k_1\cdot (x-\alpha_k),\dots,\sigma N_{\mathcal K}k_{d-1}\cdot (x-\alpha_k))k,\ \ x\in \T^d,\ \  k\in\mathcal K,
\end{equation*}
where
$\alpha_k\in \mathbb{R}^d$, $k\in\mathcal K$, 
are suitable shifts such that 
$\{W_{(k)}\}$ have mutually
disjoint supports in dimensions $d\geq 3$. The existence of such $\{\a_k\}$ can be guaranteed by taking $\rs$ sufficiently small. 
In the low dimensional case where d=2, 
$\{W_{(k)}\}$ always have intersections, 
and thus, 
we simply take $\a_k\equiv0$ for all $k\in\mathcal K$.

For brevity of notations, let
\begin{equation}\label{snp}
	\begin{array}{ll}
		&\phi_{(k)}(x) := \phi_{\rs} ( \sigma N_{\mathcal K}k_1\cdot (x-\alpha_k),\dots,\sigma N_{\mathcal K}k_{d-1}\cdot (x-\alpha_k)), \\
		&\Phi_{(k)}(x) :=  (\sigma\rs^{-1}N_{\mathcal K})^{-2} \Phi_{\rs} ( \sigma N_{\mathcal K}k_1\cdot (x-\alpha_k),\dots,\sigma  N_{\mathcal K}k_{d-1}\cdot(x-\alpha_k)), 
        \quad x\in \T^d, \ \ k\in\mathcal K.
		\end{array}
\end{equation}
In particular, 
\begin{equation}\label{snwd}
	W_{(k)} = \phi_{(k)} k,\quad  k\in\mathcal K.
\end{equation}

Furthermore, 
to construct the incompressible corrector of velocity  perturbations,
we also use the skew-symmetric potential $W_{(k)}^c$,
defined by
\begin{align}\label{potential vector}
		W_{(k)}^c := \nabla \Phi_{(k)}\otimes k-k\otimes \nabla \Phi_{(k)}, \quad k\in\mathcal K.
	\end{align}
It holds that
\begin{equation}\label{wcwc}
	 \div W_{(k)}^c=W_{(k)}.
\end{equation}

The estimates of building blocks are contained in the following Lemma. 

\begin{lemma} [Estimates of Mikado flows, \cite{S2013}] \label{buildingblockestlemma}
	For any $p \in [1,\infty]$, $k\in\mathcal K$,
 and $N,\,M \in \mathbb{N}$, 
it follows that 
	\begin{align}
		&\left\|\nabla^{N} \phi_{(k)}\right\|_{L^{p}_{x}}+\left\|\nabla^{N} \Phi_{(k)}\right\|_{L^{p}_{x}}
		\lesssim (r_{\perp}^{d-1})^{\frac 1p- \frac12} \lambda^{N},  \label{intermittent estimates2}   
	\end{align}  
and 
	\begin{align}
		&\displaystyle \|\nabla^{N} W_{(k)} \|_{C_t  L^{p}_{x}}+ \sigma\rs^{-1} \|\nabla^{N} W_{(k)}^c  \|_{C_t  L^{p}_{x}}
	 \lesssim (r_{\perp}^{d-1})^{\frac 1p- \frac12} \lambda^{N}.   \label{ew}
	\end{align}
	The implicit constants above are independent of $\rs,\,\lambda$ and $\sigma$.
\end{lemma}

In the 2D case,
as mentioned above, 
the Mikado flows have intersections with each other, 
which will affect the Reynolds stresses. 
The key point here is that the intersections 
have much smaller spatial supports. 
This is important to construct the dissipative energy solutions 
in the 2D case.

\begin{lemma}[Intersection estimate]  \label{product estimate lemma} 
In the 2D case where $d=2$, 
for every $k \neq k'\in \mathcal K$ and $p \in [1, \infty]$, we have
\begin{equation}  \label{intersect-phik1-phik1'}
\|\phi_{(k)}\phi_{(k')} \|_{C_tL^p_x}\lesssim \rs^{\frac{2}{p}-1},
\end{equation}
where the implicit constant is independent of $q$.
\end{lemma}

\begin{proof}
Let us start with the $L^\infty$-estimate.
By \eqref{snp}, 
it is clear that 
\begin{equation}   \label{inter-psik-psik'-bdd}
\|\phi_{(k)}\phi_{(k')}\|_{C_{x}}\lesssim \rs^{-1}.
\end{equation}

Next, we estimate the volume of the support of $\phi_{(k)}\phi_{(k')}$.
As $ \phi_{(k)}\phi_{(k')}$ is $(\mathbb{T}/\sigma)^2$-periodic,
it suffices to consider the support on a 
small square with side length $\sigma^{-1}$ and then multiply the result by $\sigma^2$.

In each such square, $\operatorname{supp}\phi_{(k)}$ is contained in at most
$N_{\mathcal K}$ strips of width $\sim\sigma^{-1}r_\perp$ and length
$\sim\sigma^{-1}$, and similarly for $k'$. By
$|k^\perp\cdot k'|\ge c_{\mathcal K}$, the intersection of any two such strips has area
$\lesssim(\sigma^{-1}r_\perp)^2$.

As there are at most $N_{\mathcal{K}}$ thin rectangles in the support of $\phi_{(k)}$,
the number of intersections for two distinct rectangles is at most $N_{\mathcal{K}}^2$.
Thus, the area of the support of $ \phi_{(k)}\phi_{(k')}$ is bounded by
\begin{equation} \label{psupp}
\abs{\supp(\phi_{(k)}\phi_{(k')})}
     \lesssim (\sigma^{-1}\rs)^2\sigma^2\lesssim \rs^2.
\end{equation}

Thus, using \eqref{inter-psik-psik'-bdd}  and \eqref{psupp} 
we get
\begin{align}
\|\phi_{(k)}\phi_{(k')}\|_{C_tL^p_x}
 \leq& |\supp\,(\phi_{(k)}\phi_{(k')})|^{\frac 1p}\,
       \| \phi_{(k)}\phi_{(k')}\|_{C_{x}} 
 \lesssim  \rs^{\frac 2p -1},
\end{align}
which yields \eqref{intersect-phik1-phik1'} and finishes the proof.
\end{proof}

\subsection{Velocity perturbations}   \label{Sec-Pert}
We are now in stage to define the velocity  perturbations, which consist of two components: the principal part and the incompressibility corrector. 

\medskip
\paragraph{\bf Principal part.} 
We begin with the amplitude function $a_{(k)}$ 
of the velocity flows. 
We mollify the Reynolds errors in space and time  
by 
\begin{equation}\label{mol}
	\begin{array}{ll}
		\mathring{R}_{\ell_q}:= (\ru *_{x} \phi_{\ell_q}) *_{t} \varphi_{\ell_q}, 
	\end{array}
\end{equation}
where $\phi_{\varepsilon}$ is  a family of standard mollifiers on $\T^d$, supported on a ball of radius $\ve$ in $\T^d$ , 
$\varphi_{\varepsilon}$ is a family of mollifiers on $\R$ supported on $[0,\ve]$, 
$0<\ve<1$, 
and $\ell_q$ is the mollification parameter given by \eqref{la}.

Let
\begin{align}
&	\rho (t,x) :=   2 \varepsilon_u^{-1} (\ell_q^2+|\mathring{R}_{\ell_q}(t,x) |^2)^{\frac12}
 ,\label{rhob}\\
 &	\gamma_{q}(t):=  \frac1d (e -f_q   - \|u_q \|_{L^2_x}^2 )*_{t}  \varphi_{\ell_q}(t), \label{def-gamma}
\end{align}
where $\varepsilon_{u}$ is given by 
Geometric Lemma \ref{geometric lem 2} and 
$f_q$ is the temporal step function
\begin{align}\label{def-fq}
f_q(t):= \frac34\delta_{q+3}\, 1_{[T_{q+2}+\ell_q,T_{q+1})}(t)
+\frac12\delta_{q+2}\, 1_{[T_{q+1},+\infty)}(t).
\end{align}

We define the amplitude of velocity perturbation by
\begin{equation}\label{akb}
	a_{(k)}(t,x):= (\rho (t,x)+\gamma_{q} (t))^{\frac{1}{2} }\gamma_{(k)}
       \left(\Id-\frac{\mathring{R}_{\ell_q}(t,x)}{\rho(t,x)+\gamma_{q} (t)}\right), \quad k\in\mathcal K,
\end{equation}
where $\gamma_{(k)}$ and $\mathcal K$ are as in Geometric Lemma~\ref{geometric lem 2}.

The analytic estimates of the amplitudes 
are summarized in Lemma~\ref{mae} below.  

\begin{lemma} [Estimates of amplitudes] \label{mae}
For every $N\geq 1$, $k\in\mathcal K$ and $t\in [0,+\infty)$, we have
	\begin{align}
		&\norm{a_{(k)}}_{C_tL^2_{x}} \lesssim \delta_{q+1}^{\frac12},\label{est-akbl2}\\
			&\norm{ a_{(k)} }_{C_{t,x} } \lesssim  \ell_q^{-\frac{d+1}{2}} , \label{est-akbc}\\
		& \norm{ a_{(k)} }_{C_{t,x}^N} \lesssim  \ell_q^{-(d+5)(N+3)}. \label{est-akbn}
	\end{align}
\end{lemma}

Now, we define the principal part of the velocity perturbations by
	\begin{align}\label{pv}
		w_{q+1}^{(p)} &:= \sum_{k\in\mathcal K} a_{(k)}W_{(k)}.
	\end{align}
In view of Geometric Lemma \ref{geometric lem 2}, 
we have the algebraic identity 
\begin{align}  \label{vel oscillation cancellation calculation}
	 w_{q+ 1}^{(p)} \otimes  w_{q+ 1}^{(p)}+ \mathring{R}_{\ell_q}
	=& (\rho+\gamma_{q}) \Id
	 +\sum_{k\in\mathcal K} a_{(k)}^2\P_{\neq0}( W_{(k)}\otimes W_{(k)}) \notag\\
	&\quad + \sum_{k\neq k'\in\mathcal K} a_{(k)}a_{(k')}W_{(k)}\otimes W_{(k')}.
\end{align}
The last sum vanishes in dimensions $d\geq 3$ because of the disjointness of
the supports of Mikado flows. In the 2D case, it is the intersection
contribution that will later be absorbed into $\mathring{R}_{osc.3}$.

\paragraph{\bf Incompressibility corrector} 
In order to maintain the divergence free property of 
relaxed solutions, 
we also introduce the incompressible corrector, defined by
	\begin{align}  \label{wqc}
		w_{q+1}^{(c)}
		&:=   \sum_{k\in\mathcal K } W^c_{(k)} \nabla a_{(k)}  ,
	\end{align}
where $W^c_{(k)}$ is the potential given by  \eqref{potential vector}, $k\in\mathcal K$.
Note that, by \eqref{wcwc},
	\begin{align} \label{div free velocity}
		&  w_{q+1}^{(p)} + w_{q+1}^{(c)}
		=\div\left(  \sum_{k\in\mathcal K} a_{(k)}W^c_{(k)} \right).
	\end{align}
By virtue of the fact that $a_{(k)}W^c_{(k)}$ is skew-symmetric for all $k\in\mathcal K$, we have
\begin{align} \label{div-wpc-dpc-0}
	\div (w_{q+1}^{(p)} + w_{q +1}^{(c)})= 0.
\end{align}

\medskip 
\paragraph{\bf Velocity fields} Now we define the velocity
perturbations at level $q+1$. Note that, in view of \eqref{est-r0-0}, the initial Reynolds error already satisfies the inductive estimate \eqref{rl1s} before time $T_{q+1}$. This allows us to add perturbations only on the outer temporal interval $[T_{q+1},+\infty)$ to balance the errors. In order to maintain the smoothness of approximate solutions, we use the smooth cut-off function $\chi_{q+1}$ such that
\begin{align}  \label{chiq-def}
	\chi_{q+1}(t)=\begin{cases}
		1,\quad t\in [T_{q+1},+\infty),\\
		0,\quad t\in [0,\frac{T_{q+2}+T_{q+1}}{2}],\\
	\end{cases}
\end{align}
and for  $0\leq N\leq 2$ 
\begin{align}\label{est-chiq}
\|	\chi_{q+1}\|_{C^N_t}\lesssim \ell_{q}^{-N}.
\end{align}

The velocity perturbation  $w_{q+1}$ at level $q+1$ is now defined by	
\begin{align}
w_{q+1} := \wt w_{q+1}^{(p)} +\wt  w_{q+1}^{(c)}, 
\end{align}
where
\begin{align}
	\wt w_{q+1}^{(p)}:= \chi_{q+1}w_{q+1}^{(p)}, \quad \wt w_{q+1}^{(c)} :=\chi_{q+1} w_{q+1}^{(c)} 
\end{align}  
denote the temporal truncated versions of the principal and incompressible components, respectively. 
Hence, it follows that 
\begin{align}
	w_{q+1} &:= \chi_{q+1}w_{q+1}^{(p)}+\chi_{q+1} w_{q+1}^{(c)}.
		\label{velocity perturbation}
	\end{align}

The velocity field at level $q+1$ is thus defined by
\begin{align}  \label{q+1 iterate}
		  u_{q+1}:= u_{q} + w_{q+1}.
\end{align}

Lemma \ref{totalest} below provides the  analytic estimates of the velocity perturbations.

\begin{lemma}  [Estimates of velocity perturbations] \label{totalest}
	For any $\rho \in[1,\9]$, $0\leq N\leq 1$, we have 
	\begin{align}
		&\norm{ \na^N w_{q+1}^{(p)}}_{C_tL^\rho_x } \lesssim \ell_q^{-\frac{d+1}{2}}  \lambda^N(\rs^{d-1} )^{\frac{1}{\rho}-\frac12} ,\label{uprinlp}\\
		&\norm{ \na^N w_{q+1}^{(c)} }_{C_tL^\rho_x  } \lesssim \ell_q^{-4(d+5)}  \lambda^{N-1} ( \rs^{d-1}   )^{\frac{1}{\rho}-\frac12},   \label{divcorlp}
	\end{align}
and 
\begin{align}
& \norm{ w_{q+1}^{(p)} }_{C_{t,x}^1 }  + \norm{ w_{q+1}^{(c)} }_{C_{t,x}^1 }
\lesssim \lambda^{d+1} .\label{principal c1 est}
\end{align}
\end{lemma}

\begin{proof}
Applying Lemmas \ref{buildingblockestlemma} and \ref{mae}, for any $\rho \in[1,\9]$, we get
\begin{align}\label{uplp}
\norm{\nabla^N w_{q+1}^{(p)} }_{C_tL^\rho_x }  \lesssim& \, \sum_{k\in\mathcal K }
\sum\limits_{N_1+N_2 = N}
\|a_{(k)}\|_{C_tC^{N_1}_{x}}
\norm{ \nabla^{N_2}  W_{(k)} }_{C_tL^\rho_x } \notag\\
\lesssim & \, \sum\limits_{N_1+N_2 = N>0}\ell_q^{-(d+5)(N_1+3)}\lambda^{N_2}(  \rs^{d-1} )^{\frac{1}{\rho}-\frac12}+ \ell_q^{-\frac{d+1}{2}}  \lambda^N (\rs^{d-1} )^{\frac{1}{\rho}-\frac12}    \notag\\
\lesssim & \, 	\ell_q^{-\frac{d+1}{2}}  \lambda^N (\rs^{d-1} )^{\frac{1}{\rho}-\frac12} 
\end{align}
and 
\begin{align} \label{lp vel corrector estimate}
 \norm{\na^N w_{q+1}^{(c)}}_{C_tL^\rho_x}
	 \lesssim&
	\sum\limits_{k\in\mathcal K } \sum_{N_1+N_2=N}
         \norm{ \nabla^{N_1+1} a_{(k)} }_{C_{t,x}} \norm{\na^{N_2} W^c_{(k)}}_{C_tL^{\rho}_x  }
	   \nonumber   \\
       \lesssim & \, \sum\limits_{N_1+N_2 = N}\ell_q^{-(d+5)(N_1+4)}\lambda^{N_2-1}(  \rs^{d-1} )^{\frac{1}{\rho}-\frac12}\notag\\
	 \lesssim&  \ell_q^{-4(d+5)}  \lambda^{N-1} (  \rs^{d-1})^{\frac{1}{\rho}-\frac12}.
\end{align}
	
Regarding the $C^1$-estimate \eqref{principal c1 est}, using Lemmas \ref{buildingblockestlemma} and \ref{mae} again we get
\begin{align} \label{wprincipal c1 est}
	\norm{ w_{q+1}^{(p)} }_{C_{t,x}^1}
	\lesssim&  \, \sum_{k\in\mathcal K }
	\|a_{(k)}\|_{C_{t,x}^{1} }\norm{   W_{(k)}  }_{C_{x}^{1}} \lesssim \ell_q^{-4(d+5)} \lambda^{\frac{3}{7}d+\frac47}
    \lesssim \lambda^{d+1} ,
  \end{align}
and, similarly,
\begin{align} \label{uc c1 est}
	\norm{ w_{q+1}^{(c)} }_{C_{t,x}^1 }
	& \lesssim   \sum_{k\in\mathcal K }
	\left\| W^c_{(k)} \nabla a_{(k)} \right\|_{C_{t,x}^1 }  \lesssim \lambda^{d+1}.
\end{align}

Thus, combining \eqref{wprincipal c1 est}-\eqref{uc c1 est} 
we verify \eqref{principal c1 est} and finish the proof.
\end{proof}

\subsection{Inductive estimates} \label{Subsec-induc-vel-mag} Below we verify the inductive estimates \eqref{ul2}, \eqref{uc} and \eqref{u-B-L2tx-conv} of the velocity fields 
at level $q+1$. 
Let us recall the following Lemma from \cite{cl21}.

\begin{lemma}[$L^p$ decorrelation \cite{cl21}]   \label{Decorrelation1}
	Let $\sigma\in \mathbb{N}$ and $f,g:\mathbb{T}^d\rightarrow \R$ be smooth functions. Then for every $p\in[1,+\9]$,
	\begin{equation}\label{lpdecor}
\big|\|fg(\sigma\cdot)\|_{L^p(\T^d)}-\|f\|_{L^p(\T^d)}\|g\|_{L^p(\T^d)} \big|\lesssim \sigma^{-\frac{1}{p}}\|f\|_{C^1(\T^d)}\|g\|_{L^p(\T^d)}.
	\end{equation}
\end{lemma}

Applying Lemmas~\ref{buildingblockestlemma} and \ref{Decorrelation1} with $p=2$ we get
\begin{align}
	\label{Lp decorr vel-1}
	\norm{ w^{(p)}_{q+1}}_{C_t L^2_{x}}
	&\lesssim \sum\limits_{k\in\mathcal K }
       \Big(\|a_{(k)} \|_{C_tL^2_{x}}\norm{\phi_{(k)}}_{C_tL^2_x} +\sigma^{-\frac12}\|a_{(k)}\|_{C^1_{t,x}} \norm{  \phi_{(k)} }_{C_tL^2_x}\Big)\notag\\
       &\lesssim  \delta_{q+ 1}^{\frac12}+ \lambda_{q+1}^{-\frac{1}{28}}.
\end{align}
Note that, by \eqref{def-a-0} and the fact that $\eta\in  (\frac{1}{4\alpha}, \frac{1}{2\alpha})$, we get
\begin{align}
    \lambda_{q+1}^{-\frac{1}{28}}< \la^{-1}< (\frac{8C_R}{c_*}4^{(-\theta-q+2)k_{q-1}})^{\frac12}= \delta_{q+1}^{\frac12}. 
\end{align}
Hence, we have
\begin{align}
	\label{Lp decorr vel}
	\norm{ w^{(p)}_{q+1}}_{C_t L^2_{x}}
       \lesssim  \delta_{q+ 1}^{\frac12}.
\end{align}
Then, using \eqref{est-chiq}, \eqref{velocity perturbation}, \eqref{Lp decorr vel} and Lemma \ref{totalest}  we get
\begin{align}  \label{e3.41}
	\norm{ w_{q+1} }_{C_tL^2_{x}}&\leq \|\chi_{q+1}\|_{C_t} \norm{ w_{q+1}^{(p)} }_{C_tL^2_{x}} +\|\chi_{q+1}\|_{C_t}\norm{ w_{q+1}^{(c)} }_{C_tL^2_{x}}\notag\\
     &\lesssim \delta_{q+ 1}^{\frac12}+   \ell_q^{-4(d+5)}\lambda_{q+1}^{-1}\lesssim  \delta_{q+ 1}^{\frac12},
\end{align}
which verifies \eqref{u-B-L2tx-conv}. 

Taking into account the inductive estimate \eqref{ul2} 
we also obtain 
\begin{align}\label{verubl2}
	&  \| u_{q+1} \|_{C_tL^2_x}\leq \| u_q \|_{C_tL^2_x}+\| w_{q+1} \|_{C_t L^2_{x}} \lesssim \sum_{n=0}^{q+1}\delta_{n}^{\frac{1}{2}},
\end{align}
which verifies \eqref{ul2} at level $q+1$.

Regarding the $C^1$-estimate \eqref{uc}, since $\widetilde u$ is given by the explicit multiplier,
\[
\partial_t\widetilde u
=-\nu(-\Delta)^\alpha\P_{\geq\Lambda(t)}^2\widetilde u.
\]
Because $2\alpha<1$, Sobolev embedding and \eqref{spa-regu} give
\begin{align}\label{est-wtu-c-1}
\|\widetilde u\|_{C^1_{[T_{q+3},+\infty),x}}
&\lesssim \|\widetilde u\|_{C_{[T_{q+3},+\infty)}H^{d+2}_x}
+\|\partial_t\widetilde u\|_{C_{[T_{q+3},+\infty)}H^{d+1}_x}\notag\\
&\lesssim \|\widetilde u\|_{C_{[T_{q+3},+\infty)}H^{d+2}_x}\notag\\
&\lesssim \left(1+T_{q+3}^{-\frac{d+2}{2\alpha}}\right)\|v_0\|_{L^2_x}.
\end{align}
By \eqref{def-a-0}, we get
\begin{align}\label{est-wtu-c}
    \|\wt u \|_{ C_{[T_{q+3},+\infty),x}^{1}} &\lesssim \lambda_{q+1}^{d+1}.
\end{align} 
Hence, using \eqref{tem-regu-1}, Lemma~\ref{totalest},  \eqref{est-wtu-c} and \eqref{def-a-0} we derive that
\begin{align}\label{verifyuc1}
	\norm{u_{q+1}}_{C^1_{[T_{q+3},+\infty),x}}&  \lesssim  \norm{u_{q}}_{C^1_{[T_{q+3},+\infty),x}}+\norm{ w_{q+1}}_{C^1_{t,x}}\notag\\
 &\lesssim \norm{\wt u}_{C^1_{[T_{q+3},+\infty),x}}+ \sum_{n=1}^{q}\norm{w_{n}}_{C^1_{t,x}} + \norm{w_{q+1}}_{C^1_{t,x}} \notag\\
 &\lesssim \lambda_{q+1}^{d+1}+ \sum_{n=1}^{q}\lambda_{n}^{d+1}+ \lambda_{q+1}^{d+1} \notag\\
 &\lesssim \lambda_{q+1}^{d+1},
\end{align}
which verifies \eqref{uc} at level $q+1$.

\section{Reynolds stress}   \label{Sec-Rey-stress}
This section is devoted to proving the inductive estimates 
\eqref{rl1s} and \eqref{rl1b-s} of Reynolds stress
at level $q+1$. 
To begin with, we first decompose the Reynolds stress by
\begin{align}\label{ru}
\div\mathring{R}_{q+1}  - \nabla P_{q+1} 
		&\displaystyle = \underbrace{\chi_{q+1} \partial_t (w_{q+1}^{(p)}+w_{q+1}^{(c)} )  +\div \big(u_{q} \otimes w_{q+1} + w_{q+ 1} \otimes u_{q}\big) }_{ \div\mathring{R}_{lin} +\nabla P_{lin} }   \notag\\
		&\displaystyle\quad+ \underbrace{\div (\wt  w_{q+1}^{(p)} \otimes \wt  w_{q+1}^{(p)}  +  \chi_{q+1}^2\mathring{R}_{q} ) }_{\div\mathring{R}_{osc}  +\nabla P_{osc}}  
        +\underbrace{\div\Big(\wt w_{q+1}^{(c) }\otimes w_{q+1}+ \wt  w_{q+1}^{(p) }\otimes   \wt w_{q+1}^{(c)} \Big)}_{\div\mathring{R}_{cor}  +\nabla P_{cor}}\notag\\
		&\quad +  \underbrace{  \partial_t\chi_{q+1}(w^{(p)}_{q+ 1}+w^{(c)}_{q+ 1}) +  (1-\chi_{q+1}^2)\div \mathring{R}_{q} }_{\div\mathring{R}_{cut} +\nabla P_{cut}} .
\end{align}

Applying the inverse divergence operator $\mathcal{R}$ 
given by \eqref{operaru} 
we can choose the Reynolds stress at level $q+1$ by
\begin{align}\label{rucom}
	\mathring{R}_{q+1}  := \mathring{R}_{lin}  +   \mathring{R}_{osc} + \mathring{R}_{cor} +\mathring{R}_{cut} ,
\end{align}
where $\mathring{R}_{lin}$ is the linear error
\begin{align}\label{rbp}
\mathring{R}_{lin}:= &  \mathcal{R} ( \chi_{q+1}\partial_t (w^{(p)}_{q+ 1}+w^{(c)}_{q+ 1}) ) +  u_{q}\mathring\otimes w_{q+1} + w_{q + 1} \mathring\otimes u_{q},
\end{align}
$\mathring{R}_{osc}$ is the oscillation error
\begin{align} \label{rob}
	\mathring{R}_{osc}  &:=  \chi_{q+1}^2\sum_{k\in\mathcal K}\mathcal{R}\P_{\neq 0}\left ( \P_{\neq 0}(W_{(k)}\otimes W_{(k)})\nabla (a_{(k)}^2)\right) +\cq^2(\mathring{R}_{q}-\mathring{R}_{\ell_q}) \notag\\
    &\quad + \chi_{q+1}^2\sum_{k \neq k' \in \mathcal K }\mathcal{R}\div(a_{(k)}a_{(k')}W_{(k)}\otimes W_{(k')}),
\end{align}
$\mathring{R}_{cor}$ is the corrector error
\begin{align}\label{rbp2}
	\mathring{R}_{cor}  := &\,  \wt w_{q+1}^{(p)} \mathring\otimes \wt w_{q+1}^{(c)} +\wt w_{q+1}^{(c)} \mathring\otimes w_{q+1},
\end{align}
and $\mathring{R}_{cut}$ is the cut-off error
\begin{align} \label{RBcut-def}
	\mathring{R}_{cut}
	& := \mathcal{R}   (  \partial_t\chi_{q+1}(w^{(p)}_{q+ 1}+w^{(c)}_{q+ 1})) +  (1-\chi_{q+1}^2)\mathring{R}_{q}.
\end{align}

The oscillation error  $\mathring{R}_{osc}$ 
can be further decomposed into three components 
\begin{align*}
	\mathring{R}_{osc}  = \mathring{R}_{osc.1}  +  \mathring{R}_{osc.2}+\mathring{R}_{osc.3} ,
\end{align*}
where $\mathring{R}_{osc.1} $ contains the low-high spatial  oscillations
\begin{align*}
	\mathring{R}_{osc.1}
	&:=  \cq^2 \sum_{k\in\mathcal K } \mathcal{R} \P_{\neq 0}\left(\P_{\neq 0}(W_{(k)}\otimes W_{(k)})\nabla (a_{(k)}^2)\right) ,
\end{align*}
$\mathring{R}_{osc.2} $  contains the mollification error
\begin{align*}
	\mathring{R}_{osc.2}
	&:= \cq^2 (\mathring{R}_{q}-\mathring{R}_{\ell_q}),
\end{align*}
and $\mathring{R}_{osc.3} $  contains the intersection error
\begin{align*}
	\mathring{R}_{osc.3}
	&:= \chi_{q+1}^2\sum_{k \neq k' \in \mathcal K }\mathcal{R}\div(a_{(k)}a_{(k')}W_{(k)}\otimes W_{(k')}).
\end{align*} 
We note that in dimensions $d\geq 3$, 
the intersection error $\mathring{R}_{osc.3}$ does not appear, 
because of the disjointness of the supports of Mikado flows. 
In contrast, 
we do have the intersection error 
in the low dimensional 2D case. 
In that case, 
the intersection estimate in Lemma \ref{product estimate lemma} 
is important to control the intersection error. 

Since
\[
\overline T_{q+2} <\frac{T_{q+2}+T_{q+1}}2 <T_{q+1},
\]
we have $\chi_{q+1}=0$ on $[0,\overline T_{q+2}]$. By
\eqref{q+1 iterate} and \eqref{rucom}--\eqref{RBcut-def},
\[
u_{q+1}=u_q, \qquad \mathring R_{q+1}=\mathring R_q \quad\text{on }[0,\overline T_{q+2}].
\]
Using \eqref{rq=r0}, we obtain
\[
u_{q+1}=u_0\quad\text{on }[0,T_{q+2}], \qquad \mathring R_{q+1}=\mathring R_0
\quad\text{on }[0,\overline T_{q+2}],
\]
which implies \eqref{rq=r0} at level $q+1$.

\subsection{Decay estimate} \label{sucsec-est-r-s}
Let us first prove the decay estimate \eqref{rl1s}.  
In order to bound the Calder\'on-Zygmund operators in the Reynolds stress, 
instead of $L^1_x$-norm, 
we estimate the $L^p_x$-norm   
with 
\begin{align}\label{def-p}
 p: =\frac{6(d-1)}{6(d-1)-7\ve}\in (1,2),
\end{align}
where $\ve<10^{-2}\in \mathbb{Q}$ is a small parameter. Note that 
\begin{align*}
   ( \rs^{d-1})^{\frac1p-1}= \lambda_{q+1}^\ve.
\end{align*}

We consider two temporal regimes 
$[0,+\infty)=\mathcal{I}_1\cup \mathcal I_2$, 
where $\mathcal I_1:=[0, \frac{T_{q+2}+T_{q+1}}{2}]$ 
is the inner temporal regime
and $\mathcal I_2:=[\frac{T_{q+2}+T_{q+1}}{2},+\infty)$ 
is the outer temporal regime.  
To ease notations, 
we let $T^{\rm med}_{q+1} = \frac{T_{q+2}+T_{q+1}}{2}$.

\medskip 
The inner regime is easy to deal with. 
Since the cut-off function $\chi_{q+1}$ 
vanishes on this interval due to \eqref{chiq-def}, 
we have that for any $t\in \mathcal I_1$,
$\mathring{R}_{lin} (t) =\mathring{R}_{osc}  (t) =\mathring{R}_{cor}(t) =0$ and $\mathring{R}_{cut}(t) =\mathring{R}_{0}(t) $. 
It follows that 
\begin{align}\label{rq1u-3}
	\|\mathring{R}_{q+1} \|_{ C_{\mathcal I_1} L^1_{x}}&= \|\mathring{R}_{cut} \|_{  C_{\mathcal I_1}L^1_x}= \|\mathring{R}_0 \|_{  C_{\mathcal I_1}L^1_x} \leq \frac12 c_*\delta_{q+ 3},
\end{align}
which verifies \eqref{rl1s} on $\mathcal I_1$.

In the following, 
we treat the outer temporal regime $\mathcal I_2$. 
Let us consider the four components of the Reynolds stress separately.

\medskip
\paragraph{\bf Linear error.}
Using Lemmas \ref{buildingblockestlemma} and \ref{mae}, \eqref{larsrp} and \eqref{div free velocity}, we get
\begin{align}\label{Rlin-t}
	\| \cq \mathcal{R}\partial_t( w_{q+1}^{(p) }+w_{q+1}^{(c)})\|_{C_{\mathcal I_2}L_x^p}
	&\lesssim \sum_{k\in\mathcal K }\| \mathcal{R} \div\partial_t( a_{(k)} W^c_{(k)})\|_{C_{\mathcal I_2}L_x^p}  \notag \\  
	&\lesssim \sum_{k\in\mathcal K }  \|  a_{(k)} \|_{C_{t,x}^1}\| W^c_{(k)} \|_{L^{p}} \notag\\
	&\lesssim \ell_q^{-4(d+5)} \lambda_{q+1}^{-\frac{3}{7}(d-1)-1+\ve}.
\end{align}
For the nonlinear terms in \eqref{rbp}, we have
\begin{align} \label{linear estimate1}
	&\norm{  ( u_{q} \mathring\otimes w_{q+1} + w_{q + 1} \mathring\otimes u_{q})}_{ C_{\mathcal I_2}L^p_x} 
	\lesssim\, \norm{u_{q}}_{C_{\mathcal I_2}L^\infty_x} \norm{ w_{q+1}}_{C_{\mathcal I_2} L^p_x} 
	\lesssim\, \lambda_q^{d+1} \ell_q^{-\frac{d+1}{2}}  \lambda_{q+1}^{-\frac{3}{7}(d-1)+\ve}.
\end{align}
Thus, combining \eqref{Rlin-t} and \eqref{linear estimate1} together we arrive at
\begin{align}   \label{linear estimate}
\norm{\mathring{R}_{lin}}_{C_{\mathcal I_2}L^1_x}
	& \lesssim \lambda_{q+1}^{-\frac{3}{7}(d-1)+2\ve}.
\end{align}

\paragraph{\bf Oscillation error.} 

We recall the following stationary phase lemma to handle the errors of Reynolds stress.
\begin{lemma}[\cite{lt20}, Lemma 6; see also \cite{bv19b}, Lemma B.1] \label{commutator estimate1}
	Let $a \in C^{2}\left(\mathbb{T}^{d}\right)$. For all $1<p<+\infty$, it follows that 
	$$
	\left\||\nabla|^{-1} \P_{\neq 0}\left(a \P_{\geq k} f\right)\right\|_{L^{p}\left(\mathbb{T}^{d}\right)} \lesssim k^{-1}\left\|\nabla^{2} a\right\|_{L^{\infty}\left(\mathbb{T}^{d}\right)}\|f\|_{L^{p}\left(\mathbb{T}^{d}\right)} 
	$$
 for any smooth function $f \in L^{p}\left(\mathbb{T}^{d}\right)$.
\end{lemma} 

We estimate the three components $\mathring{R}_{osc.i}$, 
$1\leq i\leq 3$, below. 

First, for the low-high oscillation component $\mathring{R}_{osc.1}$, 
using Lemmas~\ref{buildingblockestlemma}, \ref{mae} and \ref{commutator estimate1} we obtain
\begin{align}  \label{I1-esti}
	\norm{\mathring{R}_{osc.1} }_{ C_{\mathcal I_2}L^p_x}
	& \lesssim  \sum_{ k\in\mathcal K } \sigma^{-1} \norm{ \cq^2}_{C_{\mathcal I_2}} \norm{ \na^3(a^2_{(k)})}_{C_{t,x}}
	\norm{\phi_{(k)}^2  }_{L^p_x }  \nonumber  \\
	& \lesssim  \ell_q^{-9(d+5)} \sigma^{-1} \norm{\phi^2_{(k)} }_{L^p_x }\lesssim \ell_q^{-9(d+5)}\lambda_{q+1}^{-\frac17+\ve}.
\end{align} 

Moreover, the mollification error $\mathring{R}_{osc.2}$ 
can be bounded easily by \eqref{la} and \eqref{rl1b-s} 
as follows 
\begin{align}  \label{I3-esti}
	\norm{\mathring{R}_{osc.2}  }_{ C_{\mathcal I_2}L_x^1}
	&\lesssim \norm{ \cq^2}_{C_{\mathcal I_2}}  \norm{\mathring{R}_{q}-\mathring{R}_{\ell_q}}_{C_{\mathcal I_2}L_x^1}  \notag \\ 
    &\lesssim  \ell_q ( \|\mathring{R}_{q} \|_{  C^{1}_{[T_{q+2},+\infty)}L^1_x }+	\| \mathring{R}_{q}  \|_{C_{[T_{q+2},+\infty)}W^{1,1}_x} )\lesssim \ell_q^{\frac12}.
\end{align}

Regarding the intersection error $\mathring{R}_{osc.3}$, 
as mentioned above, 
$\mathring{R}_{osc.3}\equiv 0$ in dimensions $d\geq 3$.  
In the 2D case, 
we apply the intersection estimate in Lemma~\ref{product estimate lemma} and Lemma \ref{mae} to get
\begin{align}\label{I4-esti}
\norm{\mathring{R}_{osc.3}}_{C_{\mathcal I_2}L^p_x} & \lesssim  \sum_{ k \neq k' \in \mathcal K} \norm{a_{(k)}a_{(k')}W_{(k')}\otimes W_{(k)}}_{C_{\mathcal I_2}L^p_x} \notag\\	
	& \lesssim \sum_{ k \neq k' \in \mathcal K} \norm{a_{(k)} }_{C_{t,x}}\norm{a_{(k')} }_{C_{t,x}}  \|\phi_{(k)}\phi_{(k')}\|_{L^p_x}\notag\\
	& \lesssim  \ell^{-3}_q\lambda^{-\frac67+2\ve}_{q+1}.
\end{align}

Thus, combining estimates 
\eqref{I1-esti}, \eqref{I3-esti} and \eqref{I4-esti} 
we obtain 
\begin{align}
	\label{oscillation estimate}
	\norm{\mathring{R}_{osc} }_{C_{\mathcal I_2}L^1_x} 
    \leq \sum\limits_{i=1}^3 \norm{\mathring{R}_{osc, i} }_{C_{\mathcal I_2}L^1_x} 
	&\lesssim \ell_q^{\frac12}.
\end{align}

\paragraph{\bf Corrector error.} Regarding the  corrector error, by the
H\"older inequality, Lemma \ref{totalest} and \eqref{e3.41} we have
\begin{align}\label{corrector estimate}
	\norm{\mathring{R}_{cor} }_{ C_{\mathcal I_2}L^{1}_x}
	&\lesssim \norm{  w_{q+1}^{(c)}}_{C_{\mathcal I_2}L^{2}_x} (\norm{  w^{(p)}_{q+1}  }_{C_{\mathcal I_2}L^{2}_{x}} + \norm{w_{q+1} }_{C_{\mathcal I_2}L^{2}_{x}}) \notag \\ 
	&\lesssim \ell_q^{-4(d+5)} \lambda_{q+1}^{-1} \delta_{q+ 1}^{\frac12}\lesssim \lambda_{q+1}^{-1+\ve} .
\end{align}

\paragraph{\bf Cut-off error.}  
We further divide the outer regime into two subregimes 
$\mathcal{I}_2 = [T^{\rm med}_{q+1}, \infty) =  \mathcal{I}_{21} \cup  \mathcal{I}_{22}$, 
where 
$\mathcal{I}_{21}:=[T_{q+1},+\infty)$ 
and 
$\mathcal{I}_{22}:=[T^{\rm med}_{q+1}, T_{q+1})$.  
For $t\in \mathcal{I}_{21}$, 
we have 
$\mathring{R}_{cut}(t)=0$.
Moreover, 
for $t\in  \mathcal{I}_{22}$, 
using \eqref{est-r0-0} we derive 
\begin{align}  \label{R-cut-esti-1} 
\|\mathring{R}_{cut} \|_{C_{ \mathcal{I}_{22}}L^1_x}
	&\lesssim \|\mathcal{R} \partial_t\chi_{q+1}(w^{(p)}_{q+ 1}+w^{(c)}_{q+ 1})\|_{C_{\mathcal{I}_{22}}L^p_x} + \|(1-\chi_{q+1}^2 )\mathring{R}_{\ell_q} \|_{C_{\mathcal{I}_{22}}L^1_x} \notag \\
    &\lesssim \ell_q^{-\frac{d+1}{2}} (\rs^{d-1} )^{\frac1p-\frac12} + \ell_q^{-5(d+5)}\lambda_{q+1}^{-1} (\rs^{d-1})^{\frac1p-\frac12} +\| \mathring R_{q}\|_{C_{[T_{q+2},T_{q+1}]}L^{1}_{x}}\notag \\
	&\lesssim \lambda_{q+1}^{-\frac{3}{7}(d-1)+2\ve}+\frac18 c_*\delta_{q+ 3}.
\end{align}
Now, combining  \eqref{linear estimate}-\eqref{R-cut-esti-1}  we arrive at
\begin{align} \label{rq1b}
\|\mathring{R}_{q+1} \|_{ C_{\mathcal I_2}L^1_{x}}
	&\leq \| \mathring{R}_{lin}  \|_{ C_{\mathcal I_2}L^1_x} +  \| \mathring{R}_{osc}  \|_{ C_{\mathcal I_2}L^1_x}
	+  \|\mathring{R}_{cor}  \|_{ C_{\mathcal I_2}L^{ 1}_x} +  \|\mathring{R}_{cut} (t) \|_{ C_{\mathcal I_2}L^1_x}\leq C\ell_q^{\frac12}+\frac18c_*\delta_{q+3},
\end{align}
where $C$ is a universal constant. Then, by \eqref{def-a-0}, we get
\begin{align} \label{rq1b-end}
\|\mathring{R}_{q+1} \|_{ C_{\mathcal I_2}L^1_{x}}
	&\leq \frac12c_*\delta_{q+3}.
\end{align}

\subsection{Growth estimate}  \label{sucsec-est-r-b}
In this part, we verify the growth estimate \eqref{rl1b-s} for the Reynolds stress. We divide the temporal regime into two 
subregimes 
$[T_{q+3},\infty) = \mathcal I_1' \cup \mathcal I_2'$,  
where $\mathcal I_1':=[T_{q+3},T_{q+2}]$
and $\mathcal I_2':=[T_{q+2},+\infty)$.

First, in the inner regime where $t\in \mathcal I_1':=[T_{q+3},T_{q+2}]$, 
since $\chi_{q+1}(t)\equiv 0$ due to \eqref{chiq-def}, 
no perturbations are added to the velocity fields. 
Hence, 
we have  $u_{q+1}(t)= u_{q}(t)$. By  \eqref{est-r0-c-end} and \eqref{est-r0-w-1}  we get
\begin{align}
    \|\mathring{R}_{q+1} \|_{ C^{1}_{\mathcal  I_1'} L^1_x}+ \|\mathring{R}_{q+1} \|_{C_{\mathcal  I_1'} W^{1,1}_x} = \| \mathring{R}_{0}  \|_{ C^{1}_{\mathcal  I_1'} L^1_x}+ \|  \mathring{R}_{0} \|_{C_{\mathcal  I_1'}  W^{1,1}_x} \lesssim \lambda_{q+1}^{2d+2}.
\end{align}

Below we prove estimate \eqref{rl1b-s} in the outer regime 
where $\mathcal  I_2'=[T_{q+2},+\infty)$. 
We consider the four components of the Reynolds stress separately.

\medskip 
\paragraph{\bf Linear error} 
We have 
\begin{align}\label{Rlin-t-c}
	\| \cq \mathcal{R}\partial_t( w_{q+1}^{(p) }+w_{q+1}^{(c)})\|_{C^{1}_{\mathcal  I_2'}L_x^1}
	&\lesssim  \sum_{k\in\mathcal K }\| \cq\mathcal{R} \div\partial_t( a_{(k)} W^c_{(k)}) \|_{C^{1}_{\mathcal  I_2'}L_x^p}\notag\\
    &\lesssim \sum_{k\in\mathcal K } \|\cq \|_{C^{1}_{\mathcal  I_2'} } \|a_{(k)} \|_{ C^{2}_{I_2',x}}\| W^c_{(k)}\|_{L_x^p}\notag\\
    &\lesssim \ell_q^{-5(d+2)} \lambda_{q+1}^{-1}(\rs^{d-1})^{\frac1p-\frac12}\notag\\
	&\lesssim   \lambda_{q+1}^{2d+2}
\end{align}
and, similarly, 
\begin{align}\label{Rlin-t-h}
	\| \cq \mathcal{R}\partial_t( w_{q+1}^{(p) }+w_{q+1}^{(c)})\|_{C_{\mathcal  I_2'}W_x^{1,1}} &\lesssim \sum_{k\in\mathcal K }\| \cq\mathcal{R} \div\partial_t( a_{(k)} W^c_{(k)}) \|_{C_{\mathcal  I_2'}W_x^{1 ,p}}\notag\\
    &\lesssim \sum_{k\in\mathcal K }\|\cq \|_{C_{\mathcal  I_2'}}\| a_{(k)}  \|_{C^2_{t,x}} \|W^c_{(k)}\|_{W_x^{1 ,p}} \notag\\
    &\lesssim \ell_q^{-5(d+5)}(\rs^{d-1})^{\frac1p-\frac12}\notag\\
	&\lesssim\lambda_{q+1}^{2d+2}.
\end{align}
For the second and third terms in \eqref{rbp}, we get
\begin{align} \label{linear estimate1-c}
	\norm{u_{q} \mathring\otimes  w_{q+1} +  w_{q + 1} \mathring\otimes u_{q}}_{C^{1}_{\mathcal  I_2'}L_x^1}  
	&\lesssim \norm{u_{q}}_{C^{1}_{\mathcal  I_2'}L_x^2}   \norm{  w_{q+1}}_{C^{1}_{t,x}} \lesssim \lambda_{q+1}^{2d+2}
\end{align}
and, similarly, 
\begin{align} \label{linear estimate1-h}
\norm{  u_{q} \otimes  w_{q+1} + w_{q + 1} \otimes u_{q}}_{C_{\mathcal  I_2'}W_x^{1,1}}   	
	\lesssim  \lambda_{q+1}^{2d+2}.
\end{align}

Therefore, combining \eqref{Rlin-t-c}-\eqref{linear estimate1-h} we get
\begin{align}   \label{linear estimate-c}
	\norm{\mathring{R}_{lin} }_{C^{1}_{\mathcal  I_2'}L_x^1}+ 	\norm{\mathring{R}_{lin} }_{C_{\mathcal  I_2'}W_x^{1,1}}
     & \lesssim \lambda_{q+1}^{2d+2}.
\end{align}

\paragraph{\bf Oscillation error.}
Let $p$ be given by \eqref{def-p}. The multiplier bounds
\[
\mathcal R\P_{\neq0}:L^p\to W^{1,p}, \qquad
\mathcal R\div:L^p\to L^p, \qquad
\mathcal R\div:W^{1,p}\to W^{1,p}
\]
give
\begin{align*}
\|\mathring R_{\mathrm{osc}.1}\|_{C^1_{\mathcal I_2'}L^p_x}
+\|\mathring R_{\mathrm{osc}.1}\|_{C_{\mathcal I_2'}W^{1,p}_x}
&\lesssim \ell_q^{-10(d+5)-1}\lambda_{q+1}^{1+\varepsilon}
\lesssim\lambda_{q+1}^{2d+2}
\end{align*}
and
\begin{align*}
\|\mathring R_{\mathrm{osc}.2}\|_{C^1_{\mathcal I_2'}L^1_x}
+\|\mathring R_{\mathrm{osc}.2}\|_{C_{\mathcal I_2'}W^{1,1}_x}
&\lesssim \ell_q^{-1}\lambda_q^{2d+2}
\lesssim\lambda_{q+1}^{2d+2}.
\end{align*}
For $d\geq3$, $\mathring R_{\mathrm{osc}.3}=0$. For $d=2$,
thanks to Lemma~\ref{product estimate lemma} we have
\[
\|\mathring R_{\mathrm{osc}.3}\|_{C^1_{\mathcal I_2'}L^p_x}
+\|\mathring R_{\mathrm{osc}.3}\|_{C_{\mathcal I_2'}W^{1,p}_x}
\lesssim
\ell_q^{-10(d+5)-1}\lambda_{q+1}^{\frac17+2\varepsilon}
\lesssim\lambda_{q+1}^{2d+2}.
\]
Since $|\T^d|=1$ and $p>1$,
\begin{align}\label{ro-esti-c}
\|\mathring R_{\mathrm{osc}}\|_{C^1_{\mathcal I_2'}L^1_x}
+\|\mathring R_{\mathrm{osc}}\|_{C_{\mathcal I_2'}W^{1,1}_x} &\lesssim\lambda_{q+1}^{2d+2}.
\end{align}

\paragraph{\bf Corrector error.} Using \eqref{principal c1 est} in Lemma  \ref{totalest} we derive
\begin{align}\label{corrector estimate-c}
	\norm{\mathring{R}_{cor}   }_{C^{1}_{\mathcal  I_2'}L_x^1}
	\lesssim \norm{\wt  w_{q+1}^{(c)}}_{C^{1}_tL_x^2}  (\norm{\wt  w^{(p)}_{q+1}  }_{C^{1}_tL_x^2}  + \norm{  w_{q+1} }_{C^{1}_tL_x^2} )
	 \lesssim \lambda_{q+1}^{2d+2}
\end{align}
and
\begin{align}\label{corrector estimate-h}
	\norm{\mathring{R}_{cor}   }_{C_{\mathcal  I_2'}W_x^{1,1}}
	\lesssim \norm{ \wt  w_{q+1}^{(c)}}_{C_tH_x^{1}}  (\norm{\wt  w^{(p)}_{q+1}  }_{C_tH_x^{1}}  + \norm{  w_{q+1} }_{C_tH_x^{1}} )
	\lesssim& \lambda_{q+1}^{2d+2}.
\end{align}

\paragraph{\bf Cut-off error.} 
Finally, we focus on the cut-off error.  Using \eqref{rl1b-s} 
 at level $q$,  \eqref{est-chiq}  and \eqref{principal c1 est},  we get 
\begin{align}\label{cut-est-c}
    \|\mathring{R}_{cut}\|_{C^{1}_{\mathcal  I_2'}L_x^1} 
    &\lesssim \|\partial_t\chi_{q+1}\|_{C^{1}_{\mathcal  I_2'}} 
    \|w^{(p)}_{q+ 1} + w^{(c)}_{q+ 1}\|_{C^{1}_{\mathcal  I_2'}L_x^p} 
    +\|(1-\chi_{q+1}^2 )\|_{C^{1}_{\mathcal  I_2'}} \|\mathring{R}_{q}\|_{C^{1}_{\mathcal  I_2'}L_x^1}\notag\\
    &\lesssim  \ell_q^{-2} \lambda_{q+1}^{d+1}+\ell_q^{-1} \lambda_{q}^{2d+2} \notag\\
    &\lesssim \lambda_{q+1}^{2d+2}
\end{align}
and, similarly
\begin{align}\label{cut-est-h}
    \|\mathring{R} _{cut}\|_{C_{\mathcal  I_2'}W_x^{1,1}} &\lesssim   \|\partial_t\chi_{q+1}\|_{C_{\mathcal  I_2'}} 
    \|w^{(p)}_{q+ 1} + w^{(c)}_{q+ 1}\|_{C_{\mathcal  I_2'}W_x^{1,p}} +\|(1-\chi_{q+1}^2 )\|_{C_{\mathcal  I_2'}} \|\mathring{R}_{q}\|_{C_{\mathcal  I_2'}W_x^{1,1}}\notag\\
     &\lesssim \lambda_{q+1}^{2d+2}.
\end{align}

Consequently, 
combining  estimates \eqref{linear estimate-c}-\eqref{cut-est-h} we obtain 
\begin{align} \label{rq1b-c}
\norm{\mathring{R}_{q+1} }_{C^{1}_{\mathcal  I_2'}L_x^1}+ 	\norm{\mathring{R}_{q+1} }_{C_{\mathcal  I_2'}W_x^{1,1}}
	& \lesssim \lambda_{q+1}^{2d+2},
\end{align}
thereby verifying \eqref{rl1b-s} at level $q+1$.

\section{Energy inductive estimates} \label{Sec-Energy}

This section deals with the
iterative estimate \eqref{energy-est} of the energy. Here, the argument follows in an analogous manner as in our previous paper \cite{czz25}.

First, in the inner temporal regime 
where $t\in \mathcal I_1 := [T_{q+2},T^{\rm med}_{q+1}]$ 
with 
$T^{\rm med}_{q+1} := \frac{T_{q+2}+ T_{q+1}}{2}$, 
no perturbations are added to the velocity field, 
so we have $u_{q+1}(t)=u_0(t)$. Hence,  
the energy estimate  \eqref{energy-est} 
follows directly from 
\eqref{asp-0-energy} and \eqref{asp-e-q}.

In the following, 
we treat the more delicate outer regime where 
$t\in \mathcal I_2 := [T^{\rm med}_{q+1},+\infty)$.  
By definition, 
we write 
\begin{align}\label{ver-en-diff}
& e(t)-\|u_{q+1}(t)\|_{L_x^2}^2-f_q(t)\notag\\
=&\Big(e(t)-\|u_q(t)\|_{L_x^2}^2-f_q(t)
-\|\widetilde w_{q+1}^{(p)}(t)\|_{L_x^2}^2\Big)\notag\\
&+\Big(
-2\int_{\T^d}(\widetilde w_{q+1}^{(p)}\cdot\widetilde w_{q+1}^{(c)})(t)\,\d x
-\|\widetilde w_{q+1}^{(c)}(t)\|_{L_x^2}^2
-2\int_{\T^d}u_q(t)\cdot w_{q+1}(t)\,\d x
\Big)\notag\\
=:& I_1+I_2 .
\end{align}
where $f_q$ is the temporal step function 
defined by \eqref{def-fq}.

For the estimate of $I_2$, using \eqref{e3.41}, \eqref{def-a-0} and Lemma~\ref{totalest} we get
\begin{align}\label{ver-en-2}
\|\wt  w^{(p)}_{q+1}\cdot\wt  w_{q+1}^{(c)}\|_{C_tL^1_x}
	&\lesssim   \|  w^{(p)}_{q+1} \|_{C_tL^2_x}  \| w_{q+1}^{(c)} \|_{C_tL^2_x}  \lesssim \delta_{q+1}^{\frac{1}{2}}\ell_q^{-4(d+5)}\lambda_{q+1}^{-1} 
 \leq \frac{1}{96}\delta_{q+3},
\end{align}
and, similarly,
\begin{align}\label{ver-en-3}
 \|\wt w_{q+1}^{(c)}\|_{C_tL^2_x}^2 \leq  \frac{1}{96}\delta_{q+3}.
\end{align}
For the third term in $I_2$, by \eqref{uc} and Lemma~\ref{totalest},
\begin{align}\label{ver-en-4}
	  \| u_q \cdot w _{q+1} \|_{C_{[T^{\rm med}_{q+1},+\infty)}L^1_x}
	& \lesssim \|  u_q\|_{C_{[T^{\rm med}_{q+1},+\infty),x}  }\| w_{q+1} \|_{C_tL^1_x }  \lesssim \lambda_q^{d+1}\ell_q^{-\frac{d+1}{2}}\laq^{-\frac37(d-1)} \leq  \frac{1}{96}\delta_{q+3}. 
\end{align}
Thus, combining \eqref{ver-en-2}-\eqref{ver-en-4} we arrive at
\begin{align}\label{ver-en-i2}
	\sup_{t\in [T^{\rm med}_{q+1},+\infty)}|I_2(t)|\leq \frac{1}{24}\delta_{q+ 3}.
\end{align}

Now, we estimate the first term $I_1$. 
Note that, by taking the trace in \eqref{vel oscillation cancellation calculation},
\begin{align*}
|\widetilde w_{q+1}^{(p)}(t)|^2
&=d\chi_{q+1}^2(\rho+\gamma_q)
+\chi_{q+1}^2\sum_{k\in\mathcal K}a_{(k)}^2(t)
P_{\neq0}\bigl(|W_{(k)}(t)|^2\bigr)\\
&\quad+\chi_{q+1}^2
\sum_{\substack{k,k'\in\mathcal K\\ k\ne k'}}
a_{(k)}a_{(k')}\bigl(W_{(k)}\cdot W_{(k')}\bigr)(t).
\end{align*}
Thus
\begin{align}\label{ver-en-diff-3}
I_1=I_{11}+I_{12}+I_{13},
\end{align}
where $I_{11}$ 
\begin{align*}
  I_{11}:= e(t) - \|u_{q}(t)\|_{L_x^2}^2-f_q(t)-\chi_{q+1}^2(t)\big( e - \|u_{q}\|_{L_x^2}^2-f_q\big)*_{t}  \varphi_{\ell_q} (t) 
\end{align*}
is the mollification term, 
\begin{align*}
I_{12}:=-\chi_{q+1}^2(t)\int_{\T^d}
\left(2d\varepsilon_u^{-1}(\ell_q^2+|\mathring R_{\ell_q}(t)|^2)^{1/2}
+\sum_{k\in\mathcal K}a_{(k)}^2(t)P_{\neq0}(|W_{(k)}(t)|^2)\right)\,\d x,
\end{align*}
and
\begin{align*}
I_{13}:=-\chi_{q+1}^2(t)
\sum_{\substack{k,k'\in\mathcal K\\ k\ne k'}}
\int_{\T^d}a_{(k)}a_{(k')}
\bigl(W_{(k)}\cdot W_{(k')}\bigr)\,\d x .
\end{align*}

We partition the outer interval as
\[
\mathcal I_2 =\mathcal I_{2,1}\mathbin{\dot\cup} \mathcal I_{2,2}\mathbin{\dot\cup}
\mathcal I_{2,3},
\]
where
\[
\mathcal I_{2,1}:=[T_{q+1}+\ell_q,+\infty), \qquad
\mathcal I_{2,2}:=[T_{q+1},T_{q+1}+\ell_q),
\]
and
\[
\mathcal I_{2,3}:=[T^{\rm med}_{q+1},T_{q+1}).
\]

\medskip 
{\bf $\bullet$ Estimate of the $\mathcal I_{2,1}$ regime.} 
For $t \in \mathcal I_{2,1}$, 
using \eqref{def-fq} and \eqref{chiq-def} 
we first see that 
\[
\chi_{q + 1}(t) = 1,\quad f_q(t) - (f_q *_{t} \varphi_{\ell_q})(t) = 0.
\] 
For the first term $I_{11}$, standard mollification estimates,
\eqref{def-a-0}, \eqref{bdd-e-c1}, and \eqref{uc} give
\begin{align}\label{est-moll-uqzq}
\left| \bigl(\|u_q\|_{L^2_x}^2*_{t}\varphi_{\ell_q}\bigr)(t)
-\|u_q(t)\|_{L^2_x}^2
\right|
&\lesssim \ell_q \|u_q\|_{C^1_{[t-\ell_q,t]}L^2_x}
\|u_q\|_{C_tL^2_x}
\notag\\
&\lesssim
\ell_q\lambda_q^{d+1} \leq  \frac1{48}\delta_{q+3},
\end{align}
and
\begin{align}\label{est-e}
\left|e(t)-\bigl(e*_{t}\varphi_{\ell_q}\bigr)(t)\right|
\lesssim \ell_q\|e\|_{C^1_{[T_{q+2},+\infty)}} \leq \frac1{48}\delta_{q+3}.
\end{align}
Therefore,
\begin{align}\label{est-j1}
\sup_{t\in\mathcal I_{2,1}}|I_{11}(t)| \leq \frac1{24}\delta_{q+3}.
\end{align}

Regarding the second term $I_{12}$, 
we recall the following lemma.
\begin{lemma}[\cite{bms21}, Proposition 2]\label{lem-mean}
	Let $a \in C^{\infty}\left(\mathbb{T}^d ; \mathbb{R}\right), v \in C_0^{\infty}\left(\mathbb{T}^d ; \mathbb{R}\right)$. Then for any $r \in[1, \infty]$ and $\lambda\in \mathbb{N}$,
	$$
	\left|\int_{\mathbb{T}^d} a v(\lambda x)\d x\right| \leq \lambda^{-1} C_r\|\nabla a\|_{L^r_x}\|v\|_{L^{r'}_x}.
	$$
\end{lemma}

Applying Lemmas~\ref{mae} and \ref{lem-mean} 
we obtain 
\begin{align}\label{est-en-high}
\sum_{k\in\mathcal K} \int_{\mathbb{T}^d} a_{(k)}^2 P_{\neq 0}(|W_{(k)}|^2) \, \mathrm{d} x &\lesssim \sum_{k\in\mathcal K} \sigma^{-1} \|\nabla (a_{(k)}^2)\|_{C_{t,x}} \|P_{\neq 0}(|W_{(k)}|^2)\|_{L^1_x} \notag \\
&\lesssim \lambda_{q+1}^{-\frac{1}{7}} \ell_q^{-5(d+5)} \leq \frac{1}{48} \delta_{q + 3}.
\end{align}
Using \eqref{def-c8} and \eqref{rl1s} 
we also get
\begin{align}\label{prin-end}
2d \varepsilon_u^{-1} (\ell_q + \|\mathring{R}_{\ell_q}\|_{C_{[T_{q+1},+\infty)}L^1_x}) &\leq 2d \varepsilon_u^{-1} (\ell_q + c_* \delta_{q + 2}) \leq \frac{1}{48} \delta_{q + 2}.
\end{align} 
Thus, 
we obtain 
\begin{align}\label{est-j2}
\sup_{t\in [T_{q+1},+\infty)}|I_{12}(t)| \leq \frac{1}{24} \delta_{q + 2}.
\end{align}
Moreover, $I_{13}\equiv0$ for $d\ge3$. For $d=2$, Lemmas~\ref{product estimate lemma}
and~\ref{mae} give
\begin{align}\label{est-i13}
\sup_{t\in[T^{\rm med}_{q+1},+\infty)}|I_{13}(t)|
&\lesssim
\sum_{\substack{k,k'\in\mathcal K\\ k\ne k'}}
\|a_{(k)}\|_{C_{t,x}}\|a_{(k')}\|_{C_{t,x}}
\|\phi_{(k)}\phi_{(k')}\|_{L^1_x} \notag\\
&\lesssim \ell_q^{-3}\lambda_{q+1}^{-\frac67}
\le \frac1{48}\delta_{q+3}.
\end{align}
Combining \eqref{ver-en-diff}, \eqref{ver-en-i2}, \eqref{ver-en-diff-3}, \eqref{est-j1},
\eqref{est-j2},  and the estimate for $I_{13}$, we obtain, for
$t\in\mathcal I_{2,1}$,
\[
\left|e(t)-\|u_{q+1}(t)\|_{L_x^2}^2-\frac12\delta_{q+2}\right|
\le \frac1{24}\delta_{q+2}+\frac18\delta_{q+3}.
\]
Consequently,
\begin{align}\label{e-t-2}
\frac14\delta_{q+2} \le e(t)-\|u_{q+1}(t)\|_{L_x^2}^2
\le \delta_{q+2}.
\end{align}

\medskip
{\bf $\bullet$ Estimate of the $\mathcal I_{2,2}$ regime.}
Let
\[
h_q(t):=e(t)-\|u_q(t)\|_{L^2_x}^2, \qquad
A_q:=\frac12\delta_{q+2}, \qquad
B_q:=\frac34\delta_{q+3}.
\]
For $t\in\mathcal I_{2,2}$, we have $\chi_{q+1}(t)=1$, and thanks to
\eqref{ver-en-diff}--\eqref{ver-en-diff-3},
\[
e(t)-\|u_{q+1}(t)\|_{L^2_x}^2 =
\bigl(f_q*_{t}\varphi_{\ell_q}\bigr)(t) +h_q(t)-\bigl(h_q*_{t}\varphi_{\ell_q}\bigr)(t)
+I_{12}+I_{13}+I_2.
\]
By the definition of the time-scale,
\[
T_{q+1}-\ell_q>T_{q+2}+\ell_q.
\]
By \eqref{def-fq},
\[
B_q \leq \bigl(f_q*_{t}\varphi_{\ell_q}\bigr)(t) \leq  A_q.
\]
Thanks to \eqref{est-moll-uqzq} and \eqref{est-e},
\[
\left|
h_q(t)-\bigl(h_q*_{t}\varphi_{\ell_q}\bigr)(t) \right| \leq \frac1{24}\delta_{q+3}.
\]
Moreover, \eqref{est-r0-0} and \eqref{rq=r0}, with $q+1$ in place of $q$, imply
\[
\|\mathring R_{\ell_q}(t)\|_{L^1_x}
\le\frac18c_*\delta_{q+3},
\qquad t\in\mathcal I_{2,2}.
\]
Consequently, by \eqref{def-c8} and \eqref{est-en-high}, 
\[
|I_{12}(t)| \leq \frac1{48}\delta_{q+3}
+2d\ve_u^{-1} \left(\ell_q+\frac18c_*\delta_{q+3}\right)
\leq \frac1{24}\delta_{q+3}.
\]
Together with \eqref{ver-en-i2} and \eqref{est-i13}, this yields
\[
\left|
h_q(t)-\bigl(h_q*_{t}\varphi_{\ell_q}\bigr)(t)
+I_{12}(t)+I_{13}(t)+I_2(t)
\right|
\le\frac7{48}\delta_{q+3}.
\]
Therefore,
\begin{align}
e(t)-\|u_{q+1}(t)\|_{L^2_x}^2
&\le \frac12\delta_{q+2} +\frac7{48}\delta_{q+3}
\le\frac{55}{96}\delta_{q+2} <\delta_{q+2}, \label{e-t-3.1}\\
e(t)-\|u_{q+1}(t)\|_{L^2_x}^2 &\ge \frac34\delta_{q+3} -\frac7{48}\delta_{q+3}
=\frac{29}{48}\delta_{q+3} >\frac12\delta_{q+3}.
\label{e-t-3.2}
\end{align}

\medskip 
{\bf $\bullet$ Estimate of the $\mathcal I_{2,3}$ regime.} 
For $t\in \mathcal I_{2,3}$,  
since there are no 
perturbations added to the initial step $u_{0}$ before $q+1$ step 
on $[0,T_{q+1}]$, 
we have $\rr_q(t)= \rr_0(t)$ and $u_q(t) = u_0(t)$ for all $t\in \mathcal I_{2,3}$.

Hence, 
by \eqref{asp-0-energy} and \eqref{asp-e-q},
\begin{align}\label{est-e-133}
    \frac{3}{4}\delta_{q+3}\leq e(t) - \|u_{q}(t) \|_{L_x^2}^2\leq \delta_{q+2} .
\end{align}
We note that by the definition of $f_q$ in \eqref{def-fq}, 
\begin{align*}
    f_q(t)\equiv \frac34\delta_{q+ 3}, \quad (f_q*_{t}  \varphi_{\ell_q})(t)\equiv \frac34\delta_{q+ 3},\quad t\in \mathcal I_{2,3}.
\end{align*}
Then, 
using \eqref{est-e-133}, \eqref{def-fq}, \eqref{est-moll-uqzq} and \eqref{est-e}  we obtain
\begin{align}\label{j1-2-low}
	I_{11}& = (1-\chi_{q+1}^2)(e(t) - \|u_{q}(t) \|_{L_x^2}^2)- \frac34\delta_{q+ 3} +\chi_{q+1}^2 (f_q*_{t}  \varphi_{\ell_q})(t)\notag\\
	&\quad+\chi_{q+1}^2( e(t) - \|u_{q}(t) \|_{L_x^2}^2 - ( e - \|u_{q} \|_{L_x^2}^2)*_{t}  \varphi_{\ell_q}(t))\notag\\
	&\geq(1-\chi_{q+1}^2) \frac{3}{4}\delta_{q+3} - \frac34\delta_{q+ 3} +\chi_{q+1}^2 \frac34\delta_{q+ 3}- \chi_{q+1}^2\frac{1}{24} \delta_{q+3}\geq -\frac{1}{24} \delta_{q+3}.
\end{align}
We also have
\begin{align}\label{j1-2-up}
	I_{11}
 \leq(1-\chi_{q+1}^2) \delta_{q+2}  - \frac34\delta_{q+ 3} +\chi_{q+1}^2 \frac34\delta_{q+ 3}+ \chi_{q+1}^2\frac{1}{24} \delta_{q+3}\leq \delta_{q+2}.
\end{align}
By \eqref{est-r0-0}, we infer that for $t\in [T_{q+2}, T_{q+1}+\ell_q]$,
\begin{align*}
	\|\mathring{R}_{\ell_q}\|_{ C_{[T_{q+2}, T_{q+1}+\ell_q]}L^{1}_x} \leq \frac18 c_*\delta_{q+ 3}. 
\end{align*}  
Hence, using Lemmas~\ref{mae} and \ref{lem-mean} we get
\begin{align}\label{est-i12-3}
   |I_{12}| &\leq \chi_{q+1}^2(  \sum_{k\in\mathcal K   } \int_{\T^d} a_{(k)}^2 P_{\neq 0}( \left|W_{(k)}\right|^2) \d x+ 2d\ve_u^{-1}( \ell_q+   \|\mathring{R}_{\ell_q}(t)\|_{L^1_x}))\notag\\
   &\leq \chi_{q+1}^2(\frac{1}{48}\delta_{q+3} +2d\ve_u^{-1}(  \ell_q+ c_*\delta_{q+3}))\notag\\
   &\leq \frac{1}{24}\delta_{q+3}.
\end{align}

Thus, plugging \eqref{ver-en-i2}, \eqref{est-i13}, \eqref{j1-2-low}, \eqref{j1-2-up} and \eqref{est-i12-3} into \eqref{ver-en-diff} we arrive at 
for $t\in \mathcal I_{2,3}$, 
\begin{align}\label{e-t-4.1}
   e(t) - \|u_{q+1}(t)\|_{L_x^2}^2
   & \leq f_q+ I_{11}+|I_{12}|+|I_{13}|+|I_2|  \notag \\ 
    &   \leq \frac34\delta_{q+ 3}+\delta_{q+2}+ \frac{1}{24}\delta_{q+3}+\frac1{48}\delta_{q+3}+\frac{1}{24}\delta_{q+3}
   \leq  \frac32\delta_{q+2}.
\end{align}
Moreover, we also have
\begin{align}\label{e-t-4.2} 
e(t)- \|u_{q+1}(t)\|_{L_x^2}^2 
&\geq  f_q+ I_{11}-|I_{12}|-|I_{13}|-|I_2| \notag \\ 
& \geq \frac34\delta_{q+ 3}-\frac{1}{24} \delta_{q+3}- \frac{1}{24}\delta_{q+3}-\frac1{48}\delta_{q+3}-\frac{1}{24}\delta_{q+3} \geq \frac12 \delta_{q+3}. 
\end{align}

Finally, 
combining \eqref{e-t-2}, \eqref{e-t-3.1},  
\eqref{e-t-3.2}, \eqref{e-t-4.1} 
and \eqref{e-t-4.2} we verify \eqref{energy-est} at level $q+1$. 
Consequently, 
the proof of Proposition~\ref{Prop-Iterat} is complete.

\section{Proof of Theorem \ref{Thm-Non-euler-Local} } \label{Sec-prf-thm}

We are now in state to prove the main results 
in Theorem \ref{Thm-Non-euler-Local}
concerning the local dissipative energy solutions
with prescribed initial data 
satisfying {\rm (H1)} and {\rm (H2)}.

\medskip 
\paragraph{\bf Verification of initial estimates}
Let us start with verifying the inductive estimates \eqref{ul2}-\eqref{energy-est} 
at the initial step $q=0$. 
Recall that the initial relaxed solution $u_0$ is the solution $\wt u$ 
given by \eqref{def-solu} 
to the approximate $\Lambda$-Euler equations 
\eqref{equa-nse-2}, 
and the initial Reynolds stress 
$\mathring{R}_{0}$ is given by 
\eqref{r0}.  

First, 
estimate 
\eqref{ul2} follows easily from 
the choice of $\delta_0$ in \eqref{delta0}, and \eqref{rq=r0} is automatically satisfied.

Moreover, since $\widetilde u$ is given by the explicit multiplier,
\[
\partial_t\widetilde u
=-\nu(-\Delta)^\alpha\P_{\geq\Lambda(t)}^2\widetilde u.
\]
Because $2\alpha<1$, using Sobolev embedding and \eqref{spa-regu}, we obtain
\begin{align}\label{est-u0-c1}
\|\widetilde u\|_{C^1_{[T_2,+\infty),x}}
&\lesssim \|\widetilde u\|_{C_{[T_2,+\infty)}H^{d+2}_x}
+\|\partial_t\widetilde u\|_{C_{[T_2,+\infty)}H^{d+1}_x}\notag\\
&\lesssim \left(1+T_2^{-\frac{d+2}{2\alpha}}\right)\|v_0\|_{L^2_x}
\lesssim \lambda_0^{d+1},
\end{align}
thereby verifying \eqref{uc} at step $q=0$. 

Regarding the initial Reynolds stress,  using estimate  \eqref{est-r0-l1} 
and \eqref{def-delta2} we have 
\begin{align}
    \|\mathring{R}_{0}\|_{C_tL^{1}_x} \leq C_0 \|v_0\|_{L^2_x}\leq \frac{1}{8}c_* \delta_2,
\end{align}
which justifies \eqref{rl1s}. 
Moreover, by the explicit expression \eqref{r0} 
of the initial Reynolds stress, 
\begin{align}\label{est-r0-c-1}
\| \mathring{R}_{0}\|_{C^{1}_{[T_{q+2},+\infty)}L^1_x}& \lesssim \| \p_t(\varphi_\Lambda^2\widehat{\wt u}) \|_{C_{[T_{q+2},+\infty)}L^{2}_{\xi}} +\| \wt u \|_{C_{[T_{q+2},+\infty)}L^{2}_x}\notag\\
&\lesssim \| \Lambda(t)\|_{C^{1}_{[T_{q+2},+\infty)}}\| \wt u \|_{C_{[T_{q+2},+\infty)}L^{2}_x}+\| \wt u \|_{C^1_{[T_{q+2},+\infty)}L^{2}_x}\notag\\
&\lesssim \eta t^{-\eta-1}\| \wt u \|_{C_{[T_{q+2},+\infty)}L^{2}_x}+\| \wt u \|_{C^1_{[T_{q+2},+\infty)}L^{2}_x},
\end{align} 
which along with \eqref{est-wtu-c} 
yields that 
\begin{align}\label{est-r0-c-end}
    \| \mathring{R}_{0} \|_{C^{1}_{[T_{q+2},+\infty)}L^1_x} &\lesssim \lambda_{q}^{2d+2}.
\end{align}
Similarly, one also has 
\begin{align}\label{est-r0-w-1}
\| \mathring{R}_{0}\|_{C_{[T_{q+2},+\infty)}W^{1,1}_x}\lesssim \|\wt u \|_{C_{[T_{q+2},+\infty)}H^{1}_x}\lesssim  \lambda_{q}^{2d+2}. 
\end{align} 
Thus,  \eqref{rl1b-s} is justified. 
At last, 
the energy estimate 
\eqref{energy-est} 
follows directly from \eqref{asp-0-energy}. 

Hence,  estimates \eqref{ul2}-\eqref{energy-est} are satisfied at level $q=0$. 

\medskip 
\paragraph{\bf Continuous energy solution} 
Now, applying Proposition~\ref{Prop-Iterat} we obtain
a sequence of relaxed solutions $(u_{q},\rr_{q})$ to \eqref{equa-nsr} satisfying the iterative estimates \eqref{ul2}-\eqref{energy-est}, and, by \eqref{u-B-L2tx-conv},
\begin{align}\label{interpo}
 \sum_{q \geq 0} \| u_{q+1} - u_q \|_{C_t L^2_{ x}}
	\lesssim  &\,  \sum_{q \geq 0}   \delta_{q+1}^{\frac{1}{2}} <\infty.
\end{align}
It follows that $\{u_{q}\}$ is a Cauchy sequence in $C_tL^2_x$ converging to some $u$ in $C_tL^2_x$. 
Taking into account \eqref{rl1s} 
we infer that 
$\lim_{q \to \infty}  \mathring{R}_{q}  = 0$, 
and thus, 
$u$ solves \eqref{equa-euler} in the distributional sense. Moreover, 
by \eqref{energy-est}, for any $t\in (0,+\infty)$  
letting $q\to \infty$ we obtain 
\begin{align}\label{ener-tb0}
	e(t)-  \|u(t)\|_{L^2_x}^2  =\lim_{q \to +\infty}( e(t)- \|u_q(t)\|_{L^2_x}^2)= 0.
\end{align} 
In particular, 
$u$ is a solution to \eqref{equa-euler} 
with continuous energy profile.

\medskip 
\paragraph{\bf Decreasing energy channel}
We construct the energy channel 
$(e^{\mathrm{h}},  e^{\mathrm{l}})$ 
as in \eqref{upe-undere-def}. 
As pointed out in the introduction, 
it is very delicate to 
control the competition 
between the dissipation energy of $u_0$ 
and the perturbation energy of $\overline f, 
\underline f$.   

For this purpose, 
let us first derive the quantitative estimate of the dissipation energy,  
which is closely 
related to the intermittency 
and the regularity roughness of the initial data. 

More precisely, 
using the expression \eqref{def-solu} 
and the energy balance \eqref{eq-e-2}, 
we get that for $t\in (0,t_0]$, where $t_0$ is given by Proposition~\ref{Prop-R0}, 
\begin{align}\label{eq-e-u0-1}
	\frac12 \frac{\d}{\d t} \|  u_0(t) \|_{L^2_x}^2& = -\nu  \| |\nabla|^\alpha \P_{\geq \Lambda(t)} u_0(t) \|_{L^2_x}^2 \notag\\
        &\lesssim -\Lambda^{2\alpha}(t) \sum_{\xi\in \mathbb{Z}^d}| \varphi_{\Lambda(t)} (\xi) e^{-\nu\int_{\tau(\xi)}^t\varphi^2_{\Lambda(s)}(\xi)\d s (2\pi|\xi|)^{2\alpha}}\widehat{v_0}(\xi)|^2.
\end{align}
Let $n_* = n_*(t)$ and
\begin{align*}
    m_*(t):=n_*(t)+3.
\end{align*}
Since 
\[
\mathrm{supp}\,\chi_{m_*(t)} \subset [2^{m_*(t)-1},2^{m_*(t)+1}]
=[2^{n_*(t)+2},2^{n_*(t)+4}]
\]
and $2^{n_*(t)}\leq \Lambda(t)<2^{n_*(t)+1}$, we have
\begin{align*}
    2\Lambda(t)<2^{n_*(t)+2}\leq |\xi|
\end{align*}
for every $\xi$ in the support of $\chi_{m_*(t)}(|\cdot|)$. Therefore,
\begin{align*}
  \chi_{m_*(t)}(|\xi|)\varphi_{\Lambda(t)}(\xi) =\chi_{m_*(t)}(|\xi|).
\end{align*}
Moreover, on the same support,
\begin{align*}
    |\xi|\leq 2^{m_*(t)+1}\leq 16\Lambda(t).
\end{align*}
Hence,
\begin{align}\label{est-en-v0-1}
 \sum_{\xi\in \mathbb{Z}^d}| \varphi_{\Lambda(t)}  e^{-\nu\int_{\tau(\xi)}^t\varphi^2_{\Lambda(s)}(\xi)\d s (2\pi|\xi|)^{2\alpha}}\widehat{v_0 }(\xi)|^2 
 &\geq \sum_{\xi\in \mathbb{Z}^d}| \chi_{m_*(t)}(|\xi|)\varphi_{\Lambda(t)}(\xi)  e^{-\nu\int_{\tau(\xi)}^t\varphi^2_{\Lambda(s)}(\xi)\d s (2\pi|\xi|)^{2\alpha}}\widehat{v_0 }(\xi)|^2 \notag\\
 &\geq \sum_{\xi\in \mathbb{Z}^d}| \chi_{m_*(t)}(|\xi|)  e^{-\nu(t-\tau(\xi)) (2\pi|\xi|)^{2\alpha}}\widehat{v_0 }(\xi)|^2 \notag\\
 &\geq e^{-16^{2\alpha}\nu t_0^{1-2 \alpha\eta}}
 \sum_{\xi\in \mathbb{Z}^d}| \chi_{m_*(t)}(|\xi|)\widehat{v_0 }(\xi)|^2.
\end{align}
 Plugging \eqref{est-en-v0-1} into \eqref{eq-e-u0-1}, we arrive at
\begin{align}\label{eq-e-u0}
	\frac12 \frac{\d}{\d t} \|  u_0(t) \|_{L^2_x}^2
          &\lesssim -\Lambda^{2\alpha}(t)\sum_{\xi\in \mathbb{Z}^d}| \chi_{m_*(t)} (|\xi|)\widehat{v_0 }(\xi)|^2\lesssim - \Lambda^{2\alpha} (t) \|\Delta_{m_*(t)} v_0  \|_{L_x^2}^2.
\end{align}
Then, let 
\begin{align}\label{def-sn}
    s_{n}:=\frac{d-D_{n,2,\infty}}{2}+1-\kappa_2.
\end{align} 
By the roughness condition \eqref{asp-rou} of the initial data,  
we obtain the dissipation estimate 
\begin{align}\label{eq-e-u0-2}
	\frac12 \frac{\d}{\d t} \|  u_0(t) \|_{L^2_x}^2 &\lesssim -  \Lambda^{2\alpha} (t)2^{ -(d-D_{m_*(t),2,\infty})m_*(t)} \|\Delta_{m_*(t)} v_0  \|_{L_x^\infty}^2\notag\\
		 &\lesssim   - \Lambda^{2\alpha}  (t)2^{-2s_{m_*(t)}m_*(t)} \notag\\
 &\lesssim - \Lambda^{2(\alpha-s_{m_*(t)})}(t), \quad 
 t\in (0,t_0]. 
\end{align}

Regarding the perturbation energy, 
we estimate 
for $t\in [\ot_{q+1}, \ot_q]$, 
\begin{align}\label{of-ineq-1}
	\overline{f}'(t)  &\leq
    \frac{\delta_{q+2}-\delta_{q+1}}{\ot_{q+1}-\ot_q}.
\end{align}  
Using our choice of 
the backward times in \eqref{Tq-def} 
and the amplitude of Reynolds stress in \eqref{def-delta} 
we then derive 
\begin{align} \label{energy-up-per}
   \overline{f}'(t) \leq 
   \frac43\frac{\delta_{q+1}}{ \ot_q} 
   = \frac{32C_R}{3c_*}4^{k_{q-1}}\ot_q^{k_{q-1}-1}
   \leq \frac{32C_R}{3c_*}4^{K_*}\ot_q^{k_{q-1}-1}
   \lesssim \ot_q^{k_{q-1}-1}, \quad 
   t\in [\ot_{q+1}, \ot_q], \quad q\geq 2,
\end{align} 
where in the third step we used the uniform upper bound $k_{q-1}\leq K_*$ in \eqref{eq:k_q_upper_bound}. 

Thus, combining both estimates \eqref{eq-e-u0-2} and \eqref{energy-up-per} together 
we obtain that   
for the upper energy profile  
$e^{\mathrm{h}}$ 
defined in \eqref{upe-undere-def}, 
\begin{align}\label{oe-ineq-1}
	\frac{\d e^{\mathrm{h}}}{\d t} (t) &\leq  -C_1 t^{ 2\eta(s_{m_*(t)}-\a) }
    + C_2 \ot_q^{k_{q-1}-1},  \quad 
    t\in [\ot_{q+1}, \ot_q], \quad q\geq 2,
\end{align}
 where $C_1$ and $C_2$ are constants depending on $c_*$, $\nu$, $\alpha$, $\eta$, $t_0$ and $v_0$, but independent of $q$ and $\theta$. 

Estimating in an analogous manner 
we also have for the lower energy profile $e^{\mathrm{l}}$ in \eqref{upe-undere-def}, 
\begin{align}\label{ue-ineq}
	\frac{\d e^{\mathrm{l}}}{\d t} (t) &\leq -C_3 t^{2\eta(s_{m_*(t)}-\a)} +C_4\ot_q^{k_{q-1}-1}, 
    \quad t\in [\ut_{q+1}, \ut_q],\quad q\geq 2,
\end{align} 
where $C_3$ and $C_4$ are constants depending on $c_*$, $\nu$, $\alpha$, $\eta$, $t_0$ and $v_0$, but independent of $q$ and $\theta$. Here we also used that $\underline{T}_{q+1}\simeq \overline{T}_{q+1}=4^{-(\theta+q)}$, since $\ell_q=\lambda_q^{-30d}$ is much smaller than $4^{-(\theta+q)}$ by \eqref{la} and \eqref{def-a-0}.

Now, 
in order to show the dissipative property of solutions, we need 
\begin{align}\label{oe-de}
	\frac{\d e^{\mathrm{h}}}{\d t} (t)<0\quad \text{and }\quad \frac{\d e^{\mathrm{l}}}{\d t} (t)<0,\quad\forall\, t\in (0,T^*],
\end{align}
for some $T^*\in (0,+\infty)$. In view of estimates \eqref{oe-ineq-1} 
and \eqref{ue-ineq},  
this in turn requires a uniform gap between the dissipation exponent and the perturbation exponent. 
Since $4^{-\theta}\leq t_0$ and
\[
T_{q-1}+\ell_{q-2}=4^{-(\theta+q-2)}\leq 4^{-\theta},
\qquad q\geq 2,
\]
we have $n_{q-1}\geq N$ for all $q\geq 2$ by the definition \eqref{def-nq} and the choice of $t_0$ in Proposition~\ref{Prop-R0}. Moreover, since $t\in (0,t_0]$, the definition of
$n_*(t)$ and the choice of $t_0$ also imply that $m_*(t)\geq N$. 
Moreover, by \eqref{def-eta} and \eqref{def-alpha},
\begin{align*}
    \frac1\eta<4\alpha<2\kappa_2-\kappa_0-\kappa_1.
\end{align*} 
Hence, by the intermittency condition \eqref{est-1.8-2},
\begin{align}\label{en-dis-inq-5}
D_{m_*(t),2,\infty} -D^\sharp_{n_{q-1},1,\infty}
+2\kappa_2-\kappa_0 > \frac1\eta, \qquad
t\in(0,\ot_2],\quad q\geq2.
\end{align} 
Equivalently,
\begin{align}\label{en-dis-inq}
  2\eta(s_{m_*(t)}-\a)< k_{q-1}-1,
  \qquad t\in (0,\ot_2],\quad q\geq 2.
\end{align}
More precisely, choosing
\begin{align*}
    \sigma_*:=\eta(2\kappa_2-\kappa_0-\kappa_1)-1>0,
\end{align*}
we obtain from \eqref{asp-int} that
\begin{align}\label{e7.20}
    2\eta(s_{m_*(t)}-\a)\leq k_{q-1}-1-\sigma_*,
    \qquad t\in (0,\ot_2],\quad q\geq 2.
\end{align}
Let $\gamma_q:=k_{q-1}-1$. By \eqref{kq-1} and the bound $k_{q-1}\leq K_*$,
we have
\[
0< \gamma_q\le K_*-1, \qquad q\geq 2.
\]
Moreover, for
$t\in [\ot_{q+1},\ot_q]\cup[\ut_{q+1},\ut_q]$, we have
$\ot_q\lesssim t$, which implies
\[
\ot_q^{\gamma_q}\lesssim t^{\gamma_q},
\]
where the implicit constant is uniform in $q$ because $0<\gamma_q\le K_*-1$.
Hence, using \eqref{oe-ineq-1}, \eqref{ue-ineq} and \eqref{e7.20}, we infer that  
\begin{align*}
    \frac{\d e^{\mathrm{h}}}{\d t} (t)
    &\leq -C_1 t^{\gamma_q-\sigma_*}+C_2'\, t^{\gamma_q}
    =\big(-C_1 t^{-\sigma_*}+C_2'\big)t^{\gamma_q},
    \qquad t\in [\ot_{q+1},\ot_q],
\\
    \frac{\d e^{\mathrm{l}}}{\d t} (t)
    &\leq -C_3 t^{\gamma_q-\sigma_*}+C_4'\, t^{\gamma_q}
    =\big(-C_3 t^{-\sigma_*}+C_4'\big)t^{\gamma_q},
    \qquad t\in [\ut_{q+1},\ut_q],
\end{align*}
where $C_2'$ and $C_4'$ are independent of $q$ and $\theta$. Since $\sigma_*>0$,
the factors $t^{-\sigma_*}$ blow up as $t\to 0^+$. Thus, letting $T^*=\ut_2$ and
choosing $\theta$ sufficiently large so that
\begin{align}\label{theta-local-end}
    C_1 (T^*)^{-\sigma_*}>C_2'
\qquad\text{and}\qquad
C_3 (T^*)^{-\sigma_*}>C_4', 
\end{align}
we conclude that
\begin{align}\label{oe-de-end}
	\frac{\d e^{\mathrm{h}}}{\d t} (t)<0\quad \text{and }\quad \frac{\d e^{\mathrm{l}}}{\d t} (t)<0,\quad\forall\, t\in (0,T^*].
\end{align}
Therefore, we obtain a pair of continuous functions satisfying \eqref{el-eh}-\eqref{el0-eh0}
and decreasing on $[0,T^*]$. The exact-energy realization for every
$e\in\mathcal E_{\mathrm{ch}}(v_0)$ was established above in the paragraph
`Continuous energy solution', and the non-emptiness of
$\mathcal E_{\mathrm{ch}}(v_0)$ was noted after \eqref{bdd-e-c1}. 
The proof of Theorem \ref{Thm-Non-euler-Local} is complete. \hfill $\square$ 

\begin{remark}
    In the case where $v_0\in W^{1,p}_x$,  
    the proof proceeds similarly as above. 
    With the same choice $m_*(t):=n_*(t)+3$, 
    we can choose 
    \begin{align*}
      s_n= (d-D_{n,2,p})(\frac{1}{2}-\frac{1-C_*}{p})+1
    \end{align*}
    and replace 
    \eqref{en-dis-inq} by 
    \begin{align*}
\eta(2-2 \alpha+(d-D_{m_*(t), 2, p})(1-\frac{2-2 C_*}{p}))<\eta(2-2 \alpha+(d-D_{n_{q-1}, 1, p}^{\sharp})(1-\frac{1}{p})-\frac{\kappa_* C_*}{p})-1, \\
t \in(0, \bar{T}_2], \quad q \geq 2 .
\end{align*}
Equivalently,
\begin{align}  \label{DDsharp-eta}
(D_{m_*(t), 2, p}-D_{n_{q-1}, 1, p}^{\sharp})
(1-\frac1p)
+(d-D_{m_*(t), 2, p})
\frac{1-2 C_*}{p}-\frac{\kappa_* C_*}{p}>\frac{1}{\eta},
\qquad
t\in(0,\ot_2],\quad q\geq2.
\end{align} 
Indeed, by \eqref{1.11-2},
\[
D_{m_*(t),2,p}
-D^\sharp_{n_{q-1},1,p}
\geq
-\frac{\kappa_*C_*}{p}.
\]
and by \eqref{asp-reg-2}, 
\[
d-D_{m_*(t),2,p}\geq\kappa_*.
\]
Thus, the left-hand side of \eqref{DDsharp-eta} is bounded from below by
\[
\frac{\kappa_*}{p}
(1-4C_*+\frac{C_*}{p})>
\frac{\kappa_*(1-4C_*)}{p}, 
\] 
which yields \eqref{DDsharp-eta}  
by the choice of $\eta$ and $\alpha$ in property {\rm (W3)'}. 

\end{remark}

\section{A criterion for global dissipative profiles}
\label{Sec-global-diss}

Sections~\ref{Sec-Lam-Euler}--\ref{Sec-prf-thm} already provide the local ingredients for dissipativity. The background approximate solution $u_0=\widetilde u$ from
Section~\ref{Sec-Lam-Euler} satisfies the energy identity \eqref{eq-e-2}, and
Section~\ref{Sec-prf-thm} proves that the local comparison profiles decrease on
$(0,\ut_2]$. The purpose of Sections~\ref{Sec-global-diss}--\ref{Sec-sf-closure}
is to convert these local results into a global statement. Section~\ref{Sec-global-diss} shows that if we modify only the first outer
transition interval $[\ot_2,\ut_1]$ while preserving the $q=0$ estimates needed
for Proposition~\ref{Prop-Iterat}, then we obtain global dissipative weak
Euler solutions. Section~\ref{Sec-frozen-outer} applies this abstract result
to a smooth stabilized cutoff, and
Section~\ref{Sec-sf-closure} verifies the resulting hypotheses 
for data satisfying the assumption of Theorem~
\ref{Thm-Non-euler}, 
in particular for parallel shear flows.

\medskip 
Choose a function $\vartheta_{\mathrm{out}}\in C^{\infty}([0,1])$
such that
\begin{align}
\vartheta_{\mathrm{out}}(s)=0\ \text{for } s\le \tfrac14, \qquad
\vartheta_{\mathrm{out}}(s)=1\ \text{for } s\ge \tfrac34, \qquad
 0\le \vartheta_{\mathrm{out}}'(s)\le 3. 
\label{def-vartheta-br}
\end{align}
For $r\in [0,\tfrac12]$, we define
\begin{align}
\psi_r(s):=\vartheta_{\mathrm{out}}(s)+r\,\vartheta_{\mathrm{out}}(s)
\bigl(1-\vartheta_{\mathrm{out}}(s)\bigr),
\qquad s\in [0,1].
\label{def-psi-r}
\end{align}
Then each $\psi_r$ belongs to $C^{\infty}([0,1])$, satisfies
 \begin{align}
\psi_r(0)=0, \qquad \psi_r(1)=1, \qquad 0\le \psi_r'\le 5,
\label{psi-r-basic}
\end{align}
and $\psi_r'$ vanishes near $0$ and $1$. Moreover, if $r_1\neq r_2$, then
$\psi_{r_1}\neq \psi_{r_2}$.

We recall from Subsection \ref{Subsec-Ext-Global} 
the set of 
energy profiles 
$\mathcal E_{\mathrm{ch}}[U_0,\delta_2^\sharp](v_0)$ 
defined in \eqref{def-Ech-sharp}, 
where 
$U_0$ is a divergence-free background approximate solution
$U_0\in C([0,+\infty);L^2_x)\cap C^1((0,+\infty);C_x^{\infty})$, 
$\delta_2^\sharp>0$ 
and $v_0$ is the given initial datum. 
Moreover, 
recall that 
\begin{align*}
    T_*:=T_1+\ell_0=4^{-\theta}.
\end{align*}
Now we construct monotone channel profiles with a modified outer transition.

\begin{proposition}
\label{prop-global-diss}
Assume the hypotheses of Theorem~\ref{Thm-Non-euler}. Let
$u_0^\sharp\in C([0,+\infty);L^2_x)\cap C^1((0,+\infty);C_x^{\infty})$ be a 
background approximate solution and let
\[
E_0^\sharp(t):=\|u_0^\sharp(t)\|_{L^2_x}^2,
\qquad t\ge 0.
\]
Assume that
\begin{align}
u_0^\sharp(t)=u_0(t), \qquad 0\le t\le T_*,
\label{gd-agree}
\end{align}
and that $E_0^\sharp$ is non-increasing on $[0,+\infty)$. Assume moreover that
\begin{align}
\|(E_0^\sharp)'\|_{L^{\infty}([T_2,+\infty))} \lesssim \lambda_0^{d+1}, 
\label{gd-E0-prime-ass}
\end{align}
and that there exists $\delta_2^\sharp\ge \delta_3$ such that
\begin{align}
\delta_1\ge \frac34\delta_2^\sharp,
\label{gd-delta1}
\end{align}
and
\begin{align}
E_0^\sharp(\ot_2)-E_0^\sharp(\ut_1)
\ge 4(\delta_2^\sharp-\delta_3).
\label{gd-bridge}
\end{align}
Then, for every $r\in [0,\tfrac12]$, there exists a non-increasing profile
$e_{\mathrm{gd},r}^\sharp\in \mathcal E_{\mathrm{ch}}[u_0^\sharp,\delta_2^\sharp](v_0)$.
If in addition $\delta_2^\sharp>\delta_3$, then
$e_{\mathrm{gd},r_1}^\sharp\neq e_{\mathrm{gd},r_2}^\sharp$ whenever
$r_1\neq r_2$.
\end{proposition}

\begin{proof}
Fix $r\in [0,\tfrac12]$. We keep the original channel architecture from
Section~\ref{Subsec-Energy-Chan}: constant plateaux on the overlap intervals
$[\ut_{q+1},\ot_{q+1}]$, transition intervals
$[\ot_{q+1},\ut_q]$, and only the first outer transition is replaced.

\medskip
\noindent
{\bf Step 1: the inner region $[0,\ot_2]$.}
The choice of $a$ in \eqref{def-a-0} ensures that
$\ot_{q+1}<\ut_q$ for every $q\ge 1$, so the intervals below are non-empty.
Define $f_{\mathrm{in}}:[0,\ot_2]\to [0,+\infty)$ by $f_{\mathrm{in}}(0):=0$ and
\begin{align}
&f_{\mathrm{in}}(t):=\frac34\delta_{q+2},
\qquad t\in [\ut_{q+1},\ot_{q+1}],\quad q\ge 1,
\label{gd-fin-plateau}
\\
&f_{\mathrm{in}}(t) :=\frac34\delta_{q+2}
+\frac34(\delta_{q+1}-\delta_{q+2})
\vartheta_{\mathrm{out}}\!\left(\frac{t-\ot_{q+1}}{\ut_q-\ot_{q+1}}\right),
\qquad t\in [\ot_{q+1},\ut_q],\quad q\ge 2.
\label{gd-fin-transition}
\end{align}
Then $f_{\mathrm{in}}\in C([0,\ot_2])\cap C^1((0,\ot_2])$, it is
non-decreasing, and for every $q\ge 2$,
\begin{align}
\frac34\delta_{q+2} \le f_{\mathrm{in}}(t)
\le \frac34\delta_{q+1} <\delta_{q+1},
\qquad t\in [\ut_{q+1},\ot_q].
\label{gd-fin-bounds}
\end{align}
On each transition interval $[\ot_{q+1},\ut_q]$, using the same computation as in the remark after \eqref{bdd-e-c1}, we obtain
\begin{align}
0\le f_{\mathrm{in}}'(t)
\lesssim \frac{\delta_{q+1}}{\ut_q-\ot_{q+1}}
\lesssim \frac{\delta_{q+1}}{\ot_q}
\lesssim \ot_q^{k_{q-1}-1}
\lesssim t^{k_{q-1}-1},
\qquad t\in [\ot_{q+1},\ut_q],\quad q\ge 2,
\label{gd-fin-der}
\end{align}
where the implicit constants are independent of $q$.

Since $\ot_2<T_*$, assumption \eqref{gd-agree} implies
$E_0^\sharp(t)=E_0(t)$ on $[0,\ot_2]$. On each plateau interval
$[\ut_{q+1},\ot_{q+1}]$, we have $f_{\mathrm{in}}'\equiv 0$, hence
\begin{align}
\frac{\mathrm d}{\mathrm dt}\big(E_0^\sharp(t)+f_{\mathrm{in}}(t)\big)
=(E_0)'(t)\le 0,
\qquad t\in [\ut_{q+1},\ot_{q+1}],\quad q\ge 1,
\label{gd-inner-mon-plateau}
\end{align}
by the energy identity \eqref{eq-e-2}. On each transition interval
$t\in [\ot_{q+1},\ut_q]\subset (0,\ot_2]$, estimate \eqref{eq-e-u0-2} and the
exponent gap \eqref{e7.20} yield
\begin{align*}
(E_0^\sharp)'(t) =(E_0)'(t)
\le -c_1 t^{k_{q-1}-1-\sigma_*}
\end{align*}
for some constant $c_1>0$ independent of $q$. By \eqref{gd-fin-der},
\begin{align*}
f_{\mathrm{in}}'(t)\le C_{\mathrm{in}} t^{k_{q-1}-1}
\end{align*}
for a constant $C_{\mathrm{in}}>0$ independent of $q$. If necessary, we
incorporate the requirement
\begin{align}
c_1\ot_2^{-\sigma_*}\ge 2C_{\mathrm{in}}
\label{gd-theta-choice}
\end{align}
into the final lower bound on $\theta$ announced after \eqref{def-a-0}; after
this final choice of $\theta$, the parameter $a$ is chosen large enough so that
\eqref{def-a-0} still holds. Then, for $t\in [\ot_{q+1},\ut_q]$,
\begin{align}
\frac{\mathrm d}{\mathrm dt}\big(E_0^\sharp(t)+f_{\mathrm{in}}(t)\big)
\le \big(-c_1 t^{-\sigma_*}+C_{\mathrm{in}}\big)t^{k_{q-1}-1}
\le 0.
\label{gd-inner-mon-transition}
\end{align}
Combining \eqref{gd-inner-mon-plateau} and \eqref{gd-inner-mon-transition}, we
conclude that $E_0^\sharp+f_{\mathrm{in}}$ is non-increasing on $[0,\ot_2]$.

\medskip
\noindent
{\bf Step 2: modification of the first outer transition $[\ot_2,\ut_1]$.}
If $\delta_2^\sharp=\delta_3$, let
\begin{align}
 f_{\mathrm{out},r}^\sharp(t):=\frac34\delta_3,
 \qquad t\in [\ot_2,\ut_1].
 \label{gd-fbr-degenerate}
\end{align}
Otherwise \eqref{gd-bridge} implies
$E_0^\sharp(\ot_2)>E_0^\sharp(\ut_1)$, and we define
\begin{align}
\Theta^\sharp(t):=
\frac{E_0^\sharp(\ot_2)-E_0^\sharp(t)}{E_0^\sharp(\ot_2)-E_0^\sharp(\ut_1)},
\qquad t\in [\ot_2,\ut_1],
\label{gd-theta}
\end{align}
and
\begin{align}
 f_{\mathrm{out},r}^\sharp(t) :=\frac34\delta_3
 +\frac34(\delta_2^\sharp-\delta_3)\psi_r(\Theta^\sharp(t)),
 \qquad t\in [\ot_2,\ut_1].
 \label{gd-fbr}
\end{align}
Because $E_0^\sharp$ is non-increasing, $\Theta^\sharp$ is well-defined,
continuous, non-decreasing, and takes values in $[0,1]$. Using
\eqref{psi-r-basic} and \eqref{gd-bridge}, we obtain
\begin{align}
\frac34\delta_3\le f_{\mathrm{out},r}^\sharp(t)
\le \frac34\delta_2^\sharp<\delta_2^\sharp,
\qquad t\in [\ot_2,\ut_1],
\label{gd-fbr-bounds}
\end{align}
and
\begin{align}
\frac{\mathrm d}{\mathrm dt}\big(E_0^\sharp(t)+f_{\mathrm{out},r}^\sharp(t)\big)
\le 0,
\qquad t\in [\ot_2,\ut_1].
\label{gd-bridge-mon}
\end{align}
Indeed, if $\delta_2^\sharp>\delta_3$, then
\begin{align*}
(f_{\mathrm{out},r}^\sharp)'(t)
&=\frac34(\delta_2^\sharp-\delta_3)\psi_r'(\Theta^\sharp(t))
\frac{-(E_0^\sharp)'(t)}{E_0^\sharp(\ot_2)-E_0^\sharp(\ut_1)}
\\
&\le -(E_0^\sharp)'(t),
\end{align*}
and the degenerate case is immediate. Since $\psi_r'$ vanishes near $0$ and $1$,
$f_{\mathrm{out},r}^\sharp$ is constant near both endpoints of $[\ot_2,\ut_1]$.

\medskip
\noindent
{\bf Step 3: the outer tail.}
Let
\begin{align}
 f_{\mathrm{tail}}^\sharp(t):=\frac34\delta_2^\sharp,
 \qquad t\in [\ut_1,+\infty).
 \label{gd-ftail}
\end{align}
Since $E_0^\sharp$ is non-increasing,
\begin{align}
\frac{\mathrm d}{\mathrm dt}\big(E_0^\sharp(t)+f_{\mathrm{tail}}^\sharp(t)\big)
=(E_0^\sharp)'(t)\le 0,
\qquad t\in [\ut_1,+\infty),
\label{gd-tail-mon}
\end{align}
and \eqref{gd-delta1} implies
\begin{align}
\frac34\delta_2^\sharp\le f_{\mathrm{tail}}^\sharp(t)\le \delta_1,
\qquad t\in [\ut_1,+\infty).
\label{gd-tail-bounds}
\end{align}

\medskip
\noindent
{\bf Step 4: definition of the profile and verification of the channel bounds.}
Define $f_r^\sharp:[0,+\infty)\to [0,+\infty)$ by
\begin{align}
 f_r^\sharp(t):=
 \begin{cases}
  f_{\mathrm{in}}(t), & 0\le t\le \ot_2, \\
  f_{\mathrm{out},r}^\sharp(t), & \ot_2\le t\le \ut_1, \\
  f_{\mathrm{tail}}^\sharp(t), & t\ge \ut_1,
 \end{cases}
 \label{gd-fdef}
\end{align}
and then define
\begin{align}
 e_{\mathrm{gd},r}^\sharp(t):=E_0^\sharp(t)+f_r^\sharp(t),  \qquad t\ge 0.
 \label{gd-edef}
\end{align}
By construction, $f_r^\sharp$ belongs to $C([0,+\infty))\cap C^1((0,+\infty))$:
indeed, $f_{\mathrm{in}}\equiv \frac34\delta_3$ on $[\ut_2,\ot_2]$, the  first outer-transition interpolant
$f_{\mathrm{out},r}^\sharp$ is constant near $\ot_2$ with the same value, and it is
constant near $\ut_1$ with value $\frac34\delta_2^\sharp$, matching the outer
tail. Hence $e_{\mathrm{gd},r}^\sharp\in C([0,+\infty))\cap C^1((0,+\infty))$ and,
by \eqref{gd-inner-mon-plateau}, \eqref{gd-inner-mon-transition},
\eqref{gd-bridge-mon}, and \eqref{gd-tail-mon}, it is non-increasing on
$[0,+\infty)$.

We next verify the channel bounds. By \eqref{gd-fin-bounds}, for every $q\ge 2$
and every $t\in [\ut_{q+1},\ot_q]$,
\begin{align}
\frac34\delta_{q+2} \le e_{\mathrm{gd},r}^\sharp(t)-E_0^\sharp(t)
\le \frac34\delta_{q+1}
<\delta_{q+1}.
\label{gd-channel-inner}
\end{align}
On the first outer overlap interval $[\ut_2,\ot_1]$, the plateau on
$[\ut_2,\ot_2]$, the outer transition bound \eqref{gd-fbr-bounds}, and the tail bound
\eqref{gd-tail-bounds} give
\begin{align}
\frac34\delta_3 \le e_{\mathrm{gd},r}^\sharp(t)-E_0^\sharp(t)
\le \delta_2^\sharp,
\qquad t\in [\ut_2,\ot_1].
\label{gd-channel-first}
\end{align}
Finally, \eqref{gd-tail-bounds} yields
\begin{align}
\frac34\delta_2^\sharp
\le e_{\mathrm{gd},r}^\sharp(t)-E_0^\sharp(t)
\le \delta_1,
\qquad t\in [\ut_1,+\infty).
\label{gd-channel-tail}
\end{align}
Together with $e_{\mathrm{gd},r}^\sharp(0)=\|v_0\|_{L^2_x}^2$, this implies
\eqref{asp-0-energy} and \eqref{asp-e-q} with $U_0=u_0^\sharp$ and
$\delta_2^\sharp$ in place of $\delta_2$.

It remains to verify \eqref{bdd-e-c1}. First, since $E_0^\sharp$ is
non-increasing and $0\le f_r^\sharp\le \delta_1$ by \eqref{gd-tail-bounds}, we
have the uniform $C^0$-bound
\begin{align}
\|e_{\mathrm{gd},r}^\sharp\|_{C^0([0,+\infty))}
\le \|v_0\|_{L^2_x}^2+\delta_1
\lesssim \lambda_0^{d+1},
\label{gd-c0-energy}
\end{align}
where the last inequality follows from \eqref{def-a-0} after enlarging the universal
constant there if necessary.

On each transition interval $[\ot_{j+1},\ut_j]$, $j\ge 2$, we apply estimate
\eqref{gd-fin-der} to obtain
\begin{align}
\|f_r^\sharp\|_{C^1([\ot_{j+1},\ut_j])}\lesssim \lambda_{j-1}^{d+1}.
\label{gd-c1-transition}
\end{align}
On the inner plateaux $[\ut_{j},\ot_{j}]$, the derivative of
$f_r^\sharp$ vanishes. On the first outer transition interval $[\ot_2,\ut_1]$, formula
\eqref{gd-fbr} together with \eqref{psi-r-basic} and
\eqref{gd-E0-prime-ass} yields
\begin{align}
\|f_r^\sharp\|_{C^1([\ot_2,\ut_1])} 
\lesssim \lambda_0^{d+1}.
\label{gd-c1-bridge}
\end{align}
Since $f_r^\sharp$ is constant  
on $[\ut_2,\ot_2]$ and $[\ut_1,+\infty)$, we infer from
\eqref{gd-c1-bridge} that
\begin{align}
\|f_r^\sharp\|_{C^1([T_2,+\infty))} 
\lesssim \lambda_0^{d+1}.
\label{gd-c1-outer}
\end{align}

We next bound the derivative of $E_0^\sharp$ on the inner part
$[T_{q+2},\ot_2]$. By \eqref{eq-e-2}, \eqref{gd-agree} and the fact that
$\mathbb P_{\ge \Lambda(t)}$ is an $L^2_x$-bounded Fourier multiplier,
for every $t\in (0,\ot_2]$ we have
\begin{align*}
|(E_0^\sharp)'(t)| = |(E_0)'(t)|
= 2\nu\| |\nabla|^\alpha \mathbb P_{\ge \Lambda(t)}u_0(t)\|_{L^2_x}^2
\lesssim \|u_0(t)\|_{\dot H^\alpha_x}^2.
\end{align*}
Applying \eqref{spa-regu} with $s=\alpha$, we obtain
\begin{align}
\|(E_0^\sharp)'\|_{L^\infty([T_{q+2},\ot_2])}
\lesssim 1+T_{q+2}^{-1}.
\label{gd-E0-inner-1}
\end{align}
Moreover, by \eqref{Tq-def}, \eqref{la}, and \eqref{def-a-0}, we have
$T_{q+2}\simeq 4^{-(\theta+q+1)}$ and therefore
\begin{align}
1+T_{q+2}^{-1}
\lesssim 4^{\theta+q+1}
\lesssim \lambda_q^{d+1}, \qquad q\ge 0.
\label{gd-E0-inner-2}
\end{align}
Hence
\begin{align}
\|(E_0^\sharp)'\|_{L^\infty([T_{q+2},\ot_2])}
\lesssim \lambda_q^{d+1}.
\label{gd-E0-inner-c1}
\end{align}

Now fix $q\ge 0$. If a transition interval $[\ot_{j+1},\ut_j]$ with $j\ge 2$
meets $[T_{q+2},\ot_2]$, then necessarily $2\le j\le q+1$. Therefore, by
\eqref{gd-c1-transition} we have
\begin{align*}
\|f_r^\sharp\|_{C^1([T_{q+2},\ot_2])}\lesssim \lambda_q^{d+1}.
\end{align*}
Combining this with \eqref{gd-E0-prime-ass}, \eqref{gd-c0-energy}, \eqref{gd-c1-outer}, and
\eqref{gd-E0-inner-c1}, and using
$\lambda_q\ge \lambda_0$, we obtain
\begin{align*}
\|e_{\mathrm{gd},r}^\sharp\|_{C^1([T_{q+2},+\infty))}\lesssim \lambda_q^{d+1}.
\end{align*}
Hence \eqref{bdd-e-c1} holds, and so
$e_{\mathrm{gd},r}^\sharp\in \mathcal E_{\mathrm{ch}}[u_0^\sharp,\delta_2^\sharp](v_0)$.

If $\delta_2^\sharp>\delta_3$ and $r_1\neq r_2$, then
$\psi_{r_1}\neq \psi_{r_2}$. Since $\Theta^\sharp$ is continuous, non-decreasing,
and satisfies $\Theta^\sharp(\ot_2)=0$ and $\Theta^\sharp(\ut_1)=1$, its image is
$[0,1]$. Thus there exists $t\in [\ot_2,\ut_1]$ such that
$\psi_{r_1}(\Theta^\sharp(t))\neq \psi_{r_2}(\Theta^\sharp(t))$, and then
\eqref{gd-fbr} implies
$e_{\mathrm{gd},r_1}^\sharp(t)\neq e_{\mathrm{gd},r_2}^\sharp(t)$.
\end{proof}

\begin{proposition}[Iteration with a modified first outer amplitude]
\label{prop-iterat-sharp}
Assume the hypotheses of Theorem~\ref{Thm-Non-euler} and let
$\delta_2^\sharp$ satisfy
\begin{align}
\delta_3\le \delta_2^\sharp\le \frac32\delta_3.
\label{gd-delta2-sharp-range}
\end{align}
Define the modified amplitude sequence $(\widehat\delta_j)_{j\ge 0}$ by
\begin{align}
\widehat\delta_j:=
\begin{cases}
\delta_2^\sharp, & j=2,\\
\delta_j, & j\neq 2.
\end{cases}
\label{def-delta-hat}
\end{align}
Let $(u_0,p_0,\mathring R_0,e)$ be a level-$0$ Euler-Reynolds quadruple which
satisfies the $q=0$ versions of \eqref{ul2}--\eqref{energy-est} with
$\widehat\delta_j$ in place of $\delta_j$. Then one can construct a sequence
$(u_q,\mathring R_q)_{q\ge 0}$ solving \eqref{equa-nsr} such that, for every
$q\ge 0$, the estimates \eqref{ul2}--\eqref{energy-est} hold with
$\widehat\delta_j$ in place of $\delta_j$, and
\begin{align}
\|u_{q+1}-u_q\|_{C_tL^{2}_{x}}\lesssim \widehat\delta_{q+1}^{\frac12}.
\label{u-B-L2tx-conv-hat}
\end{align}
Consequently, the convergence argument in the proof of
Theorem~\ref{Thm-Non-euler-Local} remains valid with $\widehat\delta_j$ in place of
$\delta_j$.
\end{proposition}

\begin{proof}
The local construction in Sections~\ref{Sec-Interm-Flow},
\ref{Sec-Rey-stress}, and \ref{Sec-Energy} depends on the amplitude sequence at
level $q$ through the three consecutive values
$\delta_{q+1},\delta_{q+2},\delta_{q+3}$, together with the deduction of
\eqref{ul2} at level $q+1$ from the increment estimate
\eqref{u-B-L2tx-conv}; see \eqref{verubl2}. The extra appearance of
$\delta_{q+3}$ comes from the Reynolds-stress estimate in
Section~\ref{Sec-Energy}. Therefore, once the bounds
\eqref{ul2}--\eqref{energy-est} are available at level $q$, the same local
argument yields a next iterate $(u_{q+1},\mathring R_{q+1})$ satisfying
\eqref{ul2}--\eqref{energy-est} with $q+1$ in place of $q$ and the increment
estimate \eqref{u-B-L2tx-conv-hat}, provided the triple
$(\widehat\delta_{q+1},\widehat\delta_{q+2},\widehat\delta_{q+3})$ is used in
place of $(\delta_{q+1},\delta_{q+2},\delta_{q+3})$.

For the first two steps this gives the triples
$(\widehat\delta_1,\widehat\delta_2,\widehat\delta_3) =(\delta_1,\delta_2^\sharp,\delta_3)$ and
$(\widehat\delta_2,\widehat\delta_3,\widehat\delta_4)
=(\delta_2^\sharp,\delta_3,\delta_4)$. Throughout the $q=0$ construction, use
\[
f_0^\sharp(t):= \frac34\delta_3\,
1_{[T_2+\ell_0,+\infty)}(t)
\]
in place of $f_0$. By the channel bounds,
\[
e(t)-\|u_0(t)\|_{L^2_x}^2\ge f_0^\sharp(t),
\]
so $\gamma_0$ given by \eqref{def-gamma} is still non-negative. This removes the jump at $T_1$ and does not worsen
the amplitude or non-energy estimates.

We use the same partition 
$\mathcal I_{2,1}$, $\mathcal I_{2,2}$ and $\mathcal I_{2,3}$ 
as in the proof of the energy inductive estimates in Section \ref{Sec-Energy}. 
By the time-scale choice,
\[
T_1-\ell_0>T_2+\ell_0.
\]
Hence, on $\mathcal I_{2,1}\cup\mathcal I_{2,2}$,
\[
\bigl(f_0^\sharp*_{t}\varphi_{\ell_0}\bigr)(t) =\frac34\delta_3.
\]
The estimate on $e(t)$ yields
\[|e(t)-\|u_1(t)\|_{L^2_x}^2-\frac34\delta_3| \le
\begin{cases}
\frac1{24}\delta_2^\sharp+\frac5{48}\delta_3,
&t\in\mathcal I_{2,1},\\
 \frac7{48}\delta_3,
&t\in\mathcal I_{2,2},
\end{cases}
\]
In both cases the bound is less than or equal to $\frac{1}{6}\delta_3$
Therefore,
\[
\frac7{12}\delta_3 \le e(t)-\|u_1(t)\|_{L^2_x}^2 \le\frac{11}{12}\delta_3
<2\delta_2^\sharp.
\]
On $\mathcal I_{2,3}$, the calculation in
\eqref{e-t-4.1}--\eqref{e-t-4.2} is unchanged and implies
\[
\frac12\delta_3 \le e(t)-\|u_1(t)\|_{L^2_x}^2 \le\delta_2^\sharp+\frac{41}{48}\delta_3 <2\delta_2^\sharp.
\]
Thus the level-$1$ energy estimate closes. All remaining level-$1$
estimates, including \eqref{u-B-L2tx-conv-hat}, are unchanged. The
$q=1$ step is the standard construction with the triple $(\delta_2^\sharp,\delta_3,\delta_4)$.

Now let $q\ge 2$ and assume inductively that the iterates up to level $q$
are already constructed and satisfy \eqref{uc}--\eqref{energy-est} with
$\widehat\delta_j$ in place of $\delta_j$. Since
\[
\widehat\delta_{q+1}=\delta_{q+1}, \qquad
\widehat\delta_{q+2}=\delta_{q+2}, \qquad \widehat\delta_{q+3}=\delta_{q+3},
\]
the same local argument as in Proposition~\ref{Prop-Iterat} yields a level
$q+1$ iterate satisfying \eqref{uc}--\eqref{energy-est} and
\eqref{u-B-L2tx-conv-hat}. Thus the only point left is to verify the global
$L^2_x$ estimate \eqref{ul2} for the modified sequence.

The level-$0$ bound \eqref{ul2} is part of the hypotheses. For every $q\ge 0$,
we sum the increment estimates to obtain
\begin{align*}
\|u_q\|_{C_tL^2_x}
&\le \|u_0\|_{C_tL^2_x}+\sum_{n=0}^{q-1}\|u_{n+1}-u_n\|_{C_tL^2_x}\\
&\lesssim \widehat\delta_0^{\frac12}+\sum_{n=0}^{q-1}\widehat\delta_{n+1}^{\frac12}
\lesssim \sum_{n=0}^{q}\widehat\delta_n^{\frac12},
\end{align*}
where we used \eqref{u-B-L2tx-conv-hat} and the fact that
$\widehat\delta_0=\delta_0$. Hence \eqref{ul2} also holds for every level
$q\ge 1$.

Therefore, the whole iterative package \eqref{ul2}--\eqref{energy-est}, together
with \eqref{u-B-L2tx-conv-hat}, is valid for the modified amplitude sequence.
The convergence part of the proof of Theorem~\ref{Thm-Non-euler-Local} uses only the
summability of the increment bounds, the decay \eqref{rl1s}, and the energy
estimate \eqref{energy-est}, and therefore it remains valid as well.
\end{proof}

Now, we are ready to prove the abstract criterion Theorem \ref{thm-global-diss-abstract} 
for global dissipative solutions. 

\medskip 

\noindent 
{\bf Proof of Theorem \ref{thm-global-diss-abstract}.} 
Fix $r\in [0,\tfrac12]$. We first verify the additional hypothesis
\eqref{gd-E0-prime-ass} needed in Proposition~\ref{prop-global-diss}. By
\eqref{est-r0-0} and \eqref{gd-agree-full} with $q=1$,
\begin{align*}
\|\mathring R_0^\sharp\|_{C_{[T_2,T_1]}L^1_x}
\leq \|\mathring R_0\|_{C_{[0,T_1+\ell_0]}L^1_x} \leq \frac18 c_*\delta_3,
\end{align*}
while \eqref{gd-floor} implies
\begin{align*}
\|\mathring R_0^\sharp\|_{C_{[T_1,+\infty)}L^1_x}
\le \frac18 c_*\delta_2^\sharp.
\end{align*}
Hence
\begin{align}
\|\mathring R_0^\sharp\|_{C_{[T_2,+\infty)}L^1_x}<+\infty.
\label{gd-r0-l1-T2}
\end{align}
Since $(u_0^\sharp,p_0^\sharp,\mathring R_0^\sharp)$ solves the Euler-Reynolds
system \eqref{equa-nsr}, we have for every $t>0$
\begin{align}
\frac12 (E_0^\sharp)'(t)
= -\int_{\T^d} \nabla u_0^\sharp(t,x):\mathring R_0^\sharp(t,x)\,\mathrm dx.
\label{gd-energy-ER}
\end{align}
Therefore, by \eqref{gd-u0-c1} and \eqref{gd-r0-l1-T2},
\begin{align*}
\|(E_0^\sharp)'\|_{L^{\infty}([T_2,+\infty))}
\lesssim \|u_0^\sharp\|_{C^1_{[T_2,+\infty),x}}
\|\mathring R_0^\sharp\|_{C_{[T_2,+\infty)}L^1_x} 
\lesssim \lambda_0^{d+1}.
\end{align*}
Thus, Proposition~\ref{prop-global-diss} applies and yields a non-increasing
profile
$e_{\mathrm{gd},r}^\sharp\in
\mathcal E_{\mathrm{ch}}[u_0^\sharp,\delta_2^\sharp](v_0)$. We now verify the
level-$0$ hypotheses of Proposition~\ref{prop-iterat-sharp} for the quadruple
$(u_0^\sharp,p_0^\sharp,\mathring R_0^\sharp,e_{\mathrm{gd},r}^\sharp)$ and the
modified amplitude sequence $(\widehat\delta_j)_{j\ge 0}$ from
\eqref{def-delta-hat}.

Since $E_0^\sharp$ is non-increasing and $u_0^\sharp(0)=v_0$, we have
\begin{align*}
\|u_0^\sharp(t)\|_{L^2_x}\le \|v_0\|_{L^2_x},
\qquad t\ge 0.
\end{align*}
Together with the choice of $\delta_0$ in \eqref{delta0}, this verifies
\eqref{ul2} with $q=0$. Estimate \eqref{uc} at level $q=0$ is exactly
\eqref{gd-u0-c1}, and \eqref{rl1b-s} at level $q=0$ is exactly
\eqref{gd-r0-c1w11}. Moreover, since $T_*= \overline T_{1}$, assumption
\eqref{gd-agree-full} yields
\[
\mathring R_0^\sharp(t)=\mathring R_0(t), \qquad 0\le t\leq \overline T_{1},
\]
which is precisely \eqref{rq=r0} at level $q=0$.

For the $L^1_x$ smallness \eqref{rl1s}, we split $[0,+\infty)$ into
$[0,T_1]$ and $[T_1,+\infty)$. On $[0,T_1]$, \eqref{gd-agree-full} and
\eqref{est-r0-0} with $q=1$ give
\begin{align*}
\|\mathring R_0^\sharp\|_{C_{[0,T_1]}L^1_x}
\le \|\mathring R_0\|_{C_{[0,T_1+\ell_0]}L^1_x}
\le \frac18 c_*\delta_3
\le \frac18 c_*\delta_2^\sharp.
\end{align*}
On $[T_1,+\infty)$, assumption \eqref{gd-floor} yields
\begin{align*}
\|\mathring R_0^\sharp\|_{C_{[T_1,+\infty)}L^1_x} \leq \frac18 c_*\delta_2^\sharp.
\end{align*}
Hence
\begin{align*}
\|\mathring R_0^\sharp\|_{C_tL^1_x}
\le \frac18 c_*\delta_2^\sharp,
\end{align*}
which is stronger than \eqref{rl1s} at level $q=0$.

We next verify \eqref{energy-est} at level $q=0$. Since $T_1>\ut_1$,
Step~3 in the proof of Proposition~\ref{prop-global-diss} gives
\begin{align*}
e_{\mathrm{gd},r}^\sharp(t)-E_0^\sharp(t) = \frac34\delta_2^\sharp,
\qquad t\ge T_1.
\end{align*}
In particular,
\begin{align*}
e_{\mathrm{gd},r}^\sharp(t)-E_0^\sharp(t) \ge \frac12\delta_2^\sharp,
\qquad t\ge T_1,
\end{align*}
while the upper bound from Proposition~\ref{prop-global-diss} implies
\begin{align*}
e_{\mathrm{gd},r}^\sharp(t)-E_0^\sharp(t)
\le \delta_1 \le 2\delta_1, \qquad t\ge T_1.
\end{align*}
Thus, \eqref{energy-est} holds with $q=0$.

Proposition~\ref{prop-iterat-sharp} can therefore be used with the
initial quadruple
$(u_0^\sharp,p_0^\sharp,\mathring R_0^\sharp,e_{\mathrm{gd},r}^\sharp)$. This
produces a sequence $(u_{q,r},\mathring R_{q,r})_{q\ge 0}$ solving the
Euler-Reynolds system and satisfying \eqref{u-B-L2tx-conv-hat}. Since
$\sum_{q\ge 0}\widehat\delta_{q+1}^{1/2}<\infty$, the sequence
$(u_{q,r})_{q\ge 0}$ is Cauchy in $C_tL^2_x$, hence converges to some
$u_r\in C([0,+\infty);L^2_x)$. Moreover, \eqref{rl1s} yields
$\|\mathring R_{q,r}\|_{C_tL^1_x}\le c_*\widehat\delta_{q+2}\to 0$, so passing
to the limit in the Euler-Reynolds equations yields that $u_r$ is a weak Euler
solution with initial data $v_0$.

Finally, for every fixed $t>0$ and all sufficiently large $q$, we have
$t\geq T_{q+1}$. Then \eqref{energy-est} for the modified sequence implies
\begin{align*}
\frac12\widehat\delta_{q+2} \leq e_{\mathrm{gd},r}^\sharp(t)-\|u_{q,r}(t)\|_{L^2_x}^2 \leq 2\widehat\delta_{q+1}.
\end{align*}
Taking $q\to +\infty$ and using convergence in $C_tL^2_x$, we obtain
\eqref{gd-energy-exact} for every $t>0$. The identity at $t=0$ follows from the
continuity of $u_r$ and $e_{\mathrm{gd},r}^\sharp(0)=\|v_0\|_{L^2_x}^2$.
Since $e_{\mathrm{gd},r}^\sharp$ is non-increasing,
$u_r$ is admissible and therefore dissipative.

This proves the existence of at least one global dissipative solution. If
$\delta_2^\sharp>\delta_3$, then Proposition~\ref{prop-global-diss} shows that
$r\mapsto e_{\mathrm{gd},r}^\sharp$ is non-constant. Then choosing any infinite set of parameters $r\in [0,\tfrac12]$ yields infinitely many distinct exact
energy profiles, and hence infinitely many distinct global dissipative weak
Euler solutions.
\hfill $\square$

\section{Smooth stabilization of the frequency cutoff}
\label{Sec-frozen-outer}
We now apply Theorem~\ref{thm-global-diss-abstract} after smoothly
stabilizing the frequency cutoff. More precisely, we modify $\Lambda$
only after $T_*$ so that the resulting cutoff agrees with $\Lambda$ on
$(0,T_*]$, transitions smoothly to a constant value on a short interval,
and remains constant thereafter. Only the background approximate
solution from Section~\ref{Sec-Lam-Euler} is modified. The abstract argument from
Section~\ref{Sec-global-diss} will then convert the resulting $q=0$ estimates
into global dissipative weak Euler solutions.

\medskip 
Let $n_1\in \mathbb N$ be determined by \eqref{def-nq} with $q=1$, so that
\[
2^{n_1}\le \Lambda(T_1+\ell_0)=\Lambda(T_*)<2^{n_1+1}.
\]
Let 
\begin{align}
\beta_1:=\frac{k_1}{\eta}>0, 
\label{def-beta1-sf}
\end{align}
or, equivalently,
\begin{align*}
\beta_1=
\begin{cases}
2-2\alpha+d-D^\sharp_{n_1,1,\infty}-\kappa_0,
& \text{in the } C^{1-\kappa}\text{ case},\\[0.6ex]
2-2\alpha +(d-D^\sharp_{n_1,1,p}) \bigl(1-\frac1p\bigr)
-\dfrac{\kappa_*C_*}{p},
& \text{in the } W^{1,p}\text{ case}.
\end{cases}
\end{align*}

By our choice of $\theta$ in \eqref{cho-theta}, we have $2T_*\le1/2$ and
\begin{align}
\Lambda(T_*)\notin2^{\mathbb N}.
\label{sf-no-dyadic-hit}
\end{align}
Let
\[
\widetilde\ell_0:=a_0^{-30d},
\]
where $a_0$ is defined by \eqref{def-a-0}. Increasing the universal constant in
\eqref{def-a-0}, if necessary, we may assume
\[
\ot_2<T_*-2\widetilde\ell_0.
\]
By continuity and \eqref{sf-no-dyadic-hit}, there exists
$\varepsilon_0\in(0,T_*]$, independent of the final choice $a\ge a_0$, such that
\begin{align}
2^{n_1}<\Lambda(t)<2^{n_1+1}, \qquad
t\in[T_*-\widetilde\ell_0,T_*+\varepsilon_0].
\label{sf-same-shell}
\end{align}
Since $E_0$ is strictly decreasing on $(0,t_0]$, the number
\[
B_0:=E_0(\ot_2)-E_0(T_*-2\widetilde\ell_0)
\]
is positive. Choose $T_{\mathrm{fr}}\in(T_*,T_*+\varepsilon_0]$ sufficiently close to $T_*$ so that
\begin{align}
1<\mu_{\mathrm{fr}}
:=\left(\frac{\Lambda(T_{\mathrm{fr}})}{\Lambda(T_*)}\right)^{-\beta_1}
< \min\left\{\frac32,\,
1+\frac{B_0}{4\delta_3}\right\}.
\label{sf-close-delta2}
\end{align}
This choice is independent of the final value of $a$.

Fix a smooth cut-off
$\chi_{\mathrm{fr}}\in C^{\infty}([0,+\infty))$ such that
\begin{align}
0\le \chi_{\mathrm{fr}}\le 1, \qquad
\chi_{\mathrm{fr}}\equiv 1 \ \text{on } [0,T_*], \qquad
\chi_{\mathrm{fr}}\equiv 0 \ \text{on } [T_{\mathrm{fr}},+\infty), \qquad
\chi_{\mathrm{fr}}'\le 0,
\label{def-chi-fr}
\end{align}
and
\begin{align}\label{chi-fr-der-bdd}
\|\chi_{\mathrm{fr}}'\|_{L^\infty}
\le C(T_{\mathrm{fr}}-T_*)^{-1},
\end{align}
with a universal constant $C$.

Let
\begin{align}
\Lambda_{\mathrm{fr}}:=\Lambda(T_{\mathrm{fr}})
\label{def-Lambda-fr-val}
\end{align}
and define the \emph{smooth stabilized cutoff}
$\Lambda^{\mathrm{sf}}:(0,+\infty)\to [2,+\infty)$ by
\begin{align}
\Lambda^{\mathrm{sf}}(t) :=\chi_{\mathrm{fr}}(t)\Lambda(t)
+\bigl(1-\chi_{\mathrm{fr}}(t)\bigr)\Lambda_{\mathrm{fr}}.
\label{def-Lambda-sf}
\end{align} 
As in the case of $T_{\mathrm{fr}}$, 
$\Lambda^{\mathrm{sf}}$  
depends only on $\theta$ and $\ve_0$, but is independent of $a$. 
 
Since $\Lambda$ is decreasing and $\Lambda(t)\ge \Lambda_{\mathrm{fr}}$ for
$t\le T_{\mathrm{fr}}$, we have
\begin{align*}
\partial_t\Lambda^{\mathrm{sf}}(t)
=\chi_{\mathrm{fr}}'(t)\bigl(\Lambda(t)-\Lambda_{\mathrm{fr}}\bigr)
+\chi_{\mathrm{fr}}(t)\Lambda'(t)\le 0,
\end{align*}
so $\Lambda^{\mathrm{sf}}$ is smooth and non-increasing. Moreover,
$\Lambda^{\mathrm{sf}}=\Lambda$ on $(0,T_*]$ and
$\Lambda^{\mathrm{sf}}\equiv \Lambda_{\mathrm{fr}}$ on
$[T_{\mathrm{fr}},+\infty)$.

Moreover, 
for $\xi\in \mathbb Z^d$, define the activation time
\begin{align}
\tau_{\mathrm{sf}}(\xi)
:=\inf\bigl\{t\ge 0:\ \Lambda^{\mathrm{sf}}(t)\le |\xi|\bigr\}\in [0,+\infty],
\label{def-tau-sf}
\end{align}
with the convention that $\inf\varnothing=+\infty$. 
Let
$m_{\mathrm{sf}}(t,\xi)$ be the multiplier obtained from \eqref{multi-m} by
replacing $(\Lambda,\tau)$ with $(\Lambda^{\mathrm{sf}},\tau_{\mathrm{sf}})$.  
Note that  
\begin{align} \label{def-msf-full}
m_{\mathrm{sf}}(t,\xi)
=\exp\left(
-\nu (2\pi|\xi|)^{2\alpha}
\int_0^t \varphi^2_{\Lambda^{\mathrm{sf}}(s)}(\xi)\,\d s
\right), \quad \xi\in\mathbb Z^d,  
\end{align}
agrees with \eqref{multi-m} on the strictly decreasing part of the cutoff, because the integrand
vanishes before activation. 
Let 
\begin{align}
\widehat u_0^{\mathrm{sf}}(t,\xi):=m_{\mathrm{sf}}(t,\xi)\widehat v_0(\xi),
\qquad
\mathring R_0^{\mathrm{sf}}(t)
:= -\nu\mathcal R(-\Delta)^\alpha \P_{\ge \Lambda^{\mathrm{sf}}(t)}^2u_0^{\mathrm{sf}}(t),
\qquad
E_0^{\mathrm{sf}}(t):=\|u_0^{\mathrm{sf}}(t)\|_{L^2_x}^2.
\label{def-u0-sf}
\end{align}
We also write $\mathcal S^{\mathrm{sf}}(t)$ for the Fourier multiplier with symbol
$m_{\mathrm{sf}}(t,\xi)$, so that
\[u_0^{\mathrm{sf}}(t)=\mathcal S^{\mathrm{sf}}(t)v_0.\]

We show that stabilizing the cutoff does not affect the inner regime.

\begin{lemma}
\label{lem-sf-agree}
For the smooth stabilized cutoff \eqref{def-Lambda-sf}, we have
\begin{align}
u_0^{\mathrm{sf}}(t)=u_0(t), \qquad \mathring R_0^{\mathrm{sf}}(t)=\mathring R_0(t),
\qquad E_0^{\mathrm{sf}}(t)=E_0(t), \qquad 0\le t\le T_*.
\label{sf-agree-initial}
\end{align}
In particular,
\begin{align}
\|\mathring R_0^{\mathrm{sf}}\|_{C_{[0,T_*]}L^1_x}
\le \frac18 c_*\delta_3.
\label{sf-r0-initial}
\end{align}
\end{lemma}

\begin{proof} 
First note that 
since $\chi_{\mathrm{fr}}\equiv 1$ on $[0,T_*]$,  
\eqref{def-Lambda-sf} 
yields that 
\begin{equation}\label{eq:Lamsf-eq-Lam}
\Lambda^{\mathrm{sf}}(t)=\Lambda(t)\qquad\text{for all }t\in[0,T_*],
\end{equation}
which along \eqref{def-msf-full} yields that $m_{\mathrm{sf}}(t,\xi)=m(t,\xi)$ for all $t\in [0,T_*]$ and $\xi$. 
Thus, 
the associated propagation operators  coincide for all $t\in[0,T_*]$, 
i.e., $\mathcal S^{\mathrm{sf}}(t)=\mathcal S(t)$ on $[0,T_*]$, 
and so, 
\[
u_0^{\mathrm{sf}}(t)=\mathcal S^{\mathrm{sf}}(t)v_0=\mathcal S(t)v_0=u_0(t), \quad t\in [0,T_*]. 
\]  
Consequently, 
the remaining two identities in~\eqref{sf-agree-initial} follow directly from  the definitions of $\mathring R_0^{\mathrm{sf}}$ and $E_0^{\mathrm{sf}}$, 
and estimate~\eqref{sf-r0-initial} follows from~\eqref{sf-agree-initial} and~\eqref{est-r0-0} with $q=1$.  
The proof is complete. 
\end{proof}

As a consequence 
of Lemma \ref{lem-sf-agree} 
and the abstract criterion 
Theorem \ref{thm-global-diss-abstract}, 
we have the following conditional reduction to the case 
of smooth stabilized cutoff. 

\begin{theorem}
\label{thm-sf-conditional}
Assume the hypotheses of Theorem~\ref{Thm-Non-euler}. Let
$\Lambda^{\mathrm{sf}}$ be defined by \eqref{def-Lambda-sf}, and let
$u_0^{\mathrm{sf}}$, $\mathring R_0^{\mathrm{sf}}$, and $E_0^{\mathrm{sf}}$ be as in
\eqref{def-u0-sf}. Assume moreover that the following hold: 

\smallskip
\noindent
{\rm (i)} There exists a pressure $p_0^{\mathrm{sf}}$ such that
$(u_0^{\mathrm{sf}},p_0^{\mathrm{sf}},\mathring R_0^{\mathrm{sf}})$ solves the
Euler-Reynolds system \eqref{equa-nsr} at level $q=0$. Moreover,
$u_0^{\mathrm{sf}}\in C([0,+\infty);L^2_x)\cap C^1((0,+\infty);C_x^{\infty})$ and
\begin{align}
\frac12 E_0^{\mathrm{sf}}(t) +\nu\int_0^t
\|(-\Delta)^{\alpha/2}\P_{\ge \Lambda^{\mathrm{sf}}(s)}u_0^{\mathrm{sf}}(s)\|_{L^2_x}^2\,\mathrm ds
=\frac12\|v_0\|_{L^2_x}^2, \qquad t\ge 0.
\label{sf-energy-balance}
\end{align}

\smallskip
\noindent
{\rm (ii)} The $q=0$ regularity package needed in the base step holds for the
modified background approximate solution:
\begin{align}
\|u_0^{\mathrm{sf}}\|_{C_{[T_2,+\infty),x}^{1}}\lesssim \lambda_0^{d+1},
\label{sf-u0-c1}
\end{align}
and
\begin{align}
\|\mathring R_0^{\mathrm{sf}}\|_{C^{1}_{[T_2,+\infty)}L^1_x}
+\|\mathring R_0^{\mathrm{sf}}\|_{C_{[T_2,+\infty)}W^{1,1}_x}
\lesssim \lambda_0^{2d+2}.
\label{sf-r0-c1w11}
\end{align}

\smallskip
\noindent
{\rm (iii)} There exists $\delta_2^{\mathrm{sf}}$ such that
\begin{align}
&\delta_3\le \delta_2^{\mathrm{sf}}\le \frac32\delta_3,
\label{sf-delta-scale}
\\ 
&\delta_1\ge \frac34\delta_2^{\mathrm{sf}},
\label{sf-delta1}
\\ 
&\frac{8}{c_*}\|\mathring R_0^{\mathrm{sf}}\|_{C_{[T_1,+\infty)}L^1_x}
\le \delta_2^{\mathrm{sf}},
\label{sf-floor}
\\
& E_0^{\mathrm{sf}}(\ot_2)-E_0^{\mathrm{sf}}(\ut_1)
\ge 4\bigl(\delta_2^{\mathrm{sf}}-\delta_3\bigr).
\label{sf-bridge}
\end{align} 

Then, for every $r\in [0,\tfrac12]$, there exists a non-increasing profile
$e_{\mathrm{gd},r}^{\mathrm{sf}}\in
\mathcal E_{\mathrm{ch}}[u_0^{\mathrm{sf}},\delta_2^{\mathrm{sf}}](v_0)$ and a
global dissipative weak Euler solution $u_r\in C([0,+\infty);L^2_x)$ with
initial data $v_0$ and
\[
\|u_r(t)\|_{L^2_x}^2=e_{\mathrm{gd},r}^{\mathrm{sf}}(t),
\qquad t\ge 0.
\]
If $\delta_2^{\mathrm{sf}}>\delta_3$, then there are infinitely many distinct
global dissipative weak Euler solutions with initial data $v_0$.
\end{theorem}

\begin{proof}
By Lemma~\ref{lem-sf-agree},
\begin{align*}
u_0^{\mathrm{sf}}(t)=u_0(t),
\qquad
\mathring R_0^{\mathrm{sf}}(t)=\mathring R_0(t),
\qquad 0\le t\le T_*.
\end{align*}
Hence \eqref{gd-agree-full} holds with
$(u_0^\sharp,\mathring R_0^\sharp)=(u_0^{\mathrm{sf}},\mathring R_0^{\mathrm{sf}})$.
Assumption {\rm (i)} yields the Euler-Reynolds structure at level $q=0$ and the
monotonicity of $E_0^{\mathrm{sf}}$ through \eqref{sf-energy-balance};
assumptions {\rm (ii)} and {\rm (iii)} are exactly \eqref{gd-delta-scale}--\eqref{gd-r0-c1w11} with
$(u_0^\sharp,\mathring R_0^\sharp,\delta_2^\sharp)$ replaced by
$(u_0^{\mathrm{sf}},\mathring R_0^{\mathrm{sf}},\delta_2^{\mathrm{sf}})$. Therefore we can apply Theorem~\ref{thm-global-diss-abstract}
to obtain the desired result.
\end{proof}

\begin{remark}
\label{rem-sf-honest}
Relative to the
abstract criterion in 
Theorem \ref{thm-global-diss-abstract}, the only additional inputs
are the Euler-Reynolds structure of the smooth stabilized background
approximate solution, the $q=0$ regularity package, and the  quantitative floor and energy-drop conditions in assumption {\rm (iii)} of
Theorem~\ref{thm-sf-conditional}.
\end{remark}

\section{Closure of the stabilized cutoff construction}
\label{Sec-sf-closure}

Theorem~\ref{thm-sf-conditional} reduces the global problem to three
properties of the smooth stabilized background approximate solution: the
Euler-Reynolds structure at level $q=0$, the positive-time $q=0$ regularity
package from Sections~\ref{Sec-Lam-Euler} and \ref{Sec-prf-thm}, and a first
outer floor at the same scale as $\delta_3$. In this section, we verify all three for initial data satisfying the assumption of Theorem~\ref{Thm-Non-euler}.

\medskip

\begin{lemma}
\label{lem-sf-closure-bg}
Assume that $v_0$ satisfies the assumption of Theorem~
\ref{Thm-Non-euler}, and let
$\Lambda^{\mathrm{sf}}$ be any smooth stabilized cutoff defined by
\eqref{def-Lambda-sf}. Then there exists a pressure $p_0^{\mathrm{sf}}$ such that
$(u_0^{\mathrm{sf}},p_0^{\mathrm{sf}},\mathring R_0^{\mathrm{sf}})$ solves the
Euler-Reynolds system \eqref{equa-nsr} at level $q=0$. Moreover,
$u_0^{\mathrm{sf}}=\mathcal S^{\mathrm{sf}}(t)v_0$ belongs to
$C([0,+\infty);L^2_x)\cap C^1((0,+\infty);C_x^{\infty})$ and
\begin{align}
\frac12 E_0^{\mathrm{sf}}(t) +\nu\int_0^t
\|(-\Delta)^{\alpha/2}\P_{\ge \Lambda^{\mathrm{sf}}(s)}u_0^{\mathrm{sf}}(s)\|_{L^2_x}^2\,\mathrm ds =\frac12\|v_0\|_{L^2_x}^2,
\qquad t\ge 0.
\label{sf-energy-balance-closure}
\end{align}
Moreover,
\begin{align}
\|u_0^{\mathrm{sf}}\|_{C_{[T_2,+\infty),x}^{1}}\lesssim \lambda_0^{d+1},
\label{sf-u0-c1-closure}
\end{align}
and
\begin{align}
\|\mathring R_0^{\mathrm{sf}}\|_{C^{1}_{[T_2,+\infty)}L^1_x}
+\|\mathring R_0^{\mathrm{sf}}\|_{C_{[T_2,+\infty)}W^{1,1}_x}
\lesssim \lambda_0^{2d+2}, 
\label{sf-r0-c1w11-closure}
\end{align} 
where the implicit constants 
are independent of $a$. 
\end{lemma}

\begin{proof} 
 For every fixed $t\ge 0$, the assumption of Theorem~
\ref{Thm-Non-euler} yields that
$u_0^{\mathrm{sf}}(t)=\mathcal S^{\mathrm{sf}}(t)v_0$ is a steady Euler solution. 
On the other hand, the symbol $m_{\mathrm{sf}}(t,\xi)$ satisfies the same scalar
ODE as in Section~\ref{Sec-Lam-Euler}, with $\Lambda^{\mathrm{sf}}$ in place of
$\Lambda$. 
Therefore, $u_0^{\mathrm{sf}}$ solves the approximate
$\Lambda^{\mathrm{sf}}$-Euler equation obtained from \eqref{equa-nse-2} by
replacing $\Lambda$ with $\Lambda^{\mathrm{sf}}$. Choosing at each time the Euler
pressure associated with the steady field $u_0^{\mathrm{sf}}(t)$, we obtain a
pressure $p_0^{\mathrm{sf}}$ such that
$(u_0^{\mathrm{sf}},p_0^{\mathrm{sf}},\mathring R_0^{\mathrm{sf}})$ solves
\eqref{equa-nsr} at level $q=0$.

The proof of Theorem~\ref{thm-solu-le} adapts directly to this modified cutoff.  
The arguments from Section~\ref{Sec-Lam-Euler} use only that the cutoff is
smooth, non-increasing, equal to $t^{-\eta}$ near $t=0$, and bounded below by
$2$. These properties are preserved by $\Lambda^{\mathrm{sf}}$. Hence the
positive-time smoothing, the low/high-frequency decomposition, and the $L^2$
continuity at $t=0$ are unchanged, which yields
\eqref{sf-energy-balance-closure} and
$u_0^{\mathrm{sf}}\in C([0,+\infty);L^2_x)\cap C^1((0,+\infty);C_x^{\infty})$.

To verify the $q=0$ regularity package, we repeat the proof of the initial
estimates in Section~\ref{Sec-prf-thm}. The analogue of \eqref{est-u0-c1}
therefore yields \eqref{sf-u0-c1-closure}. For the Reynolds stress, the proof of
\eqref{est-r0-w-1} is identical, and the only new point is the term
involving $\partial_t \P_{\ge \Lambda^{\mathrm{sf}}(t)}^2$. 
Since
\[
(\Lambda^{\mathrm{sf}})'(t) =\chi_{\mathrm{fr}}'(t)\bigl(\Lambda(t)-\Lambda_{\mathrm{fr}}\bigr)
+\chi_{\mathrm{fr}}(t)\Lambda'(t),
\]
and
\[
|\Lambda(t)-\Lambda_{\mathrm{fr}}| \le (T_{\mathrm{fr}}-T_*)\sup_{s\in [T_*,T_{\mathrm{fr}}]}|\Lambda'(s)|, \qquad
|\chi_{\mathrm{fr}}'|\lesssim (T_{\mathrm{fr}}-T_*)^{-1},
\]
we obtain the uniform bound
\begin{align}
\| (\Lambda^{\mathrm{sf}})'\|_{L^{\infty}([T_2,+\infty))}
\lesssim \sup_{s\in [T_2,+\infty)}|\Lambda'(s)|<\infty.
\label{sf-Lambda-prime-bdd}
\end{align}
Hence the proof of \eqref{est-r0-c-1}--\eqref{est-r0-w-1}, with
$\Lambda^{\mathrm{sf}}$ in place of $\Lambda$, implies \eqref{sf-r0-c1w11-closure}. 

Note that since $\Lambda^{\mathrm{sf}}$ is independent of $a$,  the implicit constants in \eqref{sf-u0-c1-closure} and \eqref{sf-r0-c1w11-closure} are also independent of $a$.
\end{proof}
Now we prove a dyadic multiplier bound for a stabilized cutoff.

\begin{lemma}\label{lem-sf-multiplier}
Assume
\[
2T_*\le\frac12,\qquad T_*<T_{\mathrm{fr}}\le2T_*.
\]
Let $\Lambda^{\mathrm{sf}}$ and $m_{\mathrm{sf}}$ be defined by
\eqref{def-Lambda-sf} and \eqref{def-msf-full}. Then for every $1\le p\le\infty$ there exists
$C_{p,d}>0$, independent of $\theta$, $a$, $T_{\mathrm{fr}}$, $n$, and $t$, such that
\begin{align}\label{sf-Lp-mult}
\|\mathcal S^{\mathrm{sf}}(t)\P_{\ge \Lambda^{\mathrm{sf}}(t)}^2\Delta_n f\|_{L^p_x}
\le C_{p,d}\|\Delta_n f\|_{L^p_x}
\end{align}
for all $n\ge1$, $t\ge0$, and mean-free $f\in L^p_x$.
\end{lemma}

\begin{proof}
As in Lemma~\ref{bd-m}, it suffices to prove a uniform $W^{d+1,1}$ bound for
\begin{align*}
b^{\mathrm{sf}}_{n,t}(\zeta) :=m_{\mathrm{sf}}(t,2^n\zeta)
\varphi^2_{\Lambda^{\mathrm{sf}}(t)}(2^n\zeta)\Psi(\zeta),
\end{align*}
where $\Psi$ is given by \eqref{def-Psi-lem31}. For $t\le T_*$ this follows from Lemma~\ref{bd-m}, since
$\Lambda^{\mathrm{sf}}=\Lambda$ on $(0,T_*]$ and the corresponding multipliers agree up to time $t$.

Assume $t>T_*$, we let $\xi=2^n\zeta$. We only consider the support of the symbol and its $\zeta$-derivatives; elsewhere the
bounds are trivial. On this support, 
\begin{align}\label{sf-active-frequency}
|\xi|\gtrsim \Lambda^{\mathrm{sf}}(t)\ge \Lambda(T_{\mathrm{fr}}).
\end{align}
Write the pre-stabilized accumulated dissipation as
\begin{align*}
A_{\mathrm{sf}}^{\mathrm{pre}}(t,\xi) :=\int_0^{t\wedge T_{\mathrm{fr}}}
\varphi^2_{\Lambda^{\mathrm{sf}}(s)}(\xi)\,\d s .
\end{align*}

We estimate the $\xi$-derivatives of $A_{\mathrm{sf}}^{\mathrm{pre}}$. On $(0,T_*]$, the cutoff is the original
power-law cutoff. For every $1\le N\le d+1$, the derivative
$\nabla_\xi^N\varphi^2_{\Lambda(s)}(\xi)$ can be nonzero only when
$\Lambda(s)\simeq |\xi|$. Since $\Lambda(s)=s^{-\eta}$ on $(0,T_*]$, the set of such
$s$ has length $\lesssim |\xi|^{-1/\eta}$. On this set the integrand is bounded by
$\lesssim |\xi|^{-N}$. Hence the contribution from $(0,T_*]$ is bounded by
$\lesssim |\xi|^{-N-\frac1\eta}$.

On $[T_*,T_{\mathrm{fr}}]$, split into two cases. If $|\xi|>2\Lambda(T_*)$, then
$\varphi_{\Lambda^{\mathrm{sf}}(s)}(\xi)\equiv1$ throughout the transition, and this interval contributes no
$\xi$-derivatives. If $|\xi|\le2\Lambda(T_*)$, by \eqref{sf-active-frequency}
\begin{align*}
\Lambda(T_{\mathrm{fr}})\lesssim |\xi|\le2\Lambda(T_*).
\end{align*}
Since $T_{\mathrm{fr}}\le2T_*\le1/2$ and $\Lambda(t)=t^{-\eta}$ on $(0,1/2]$, this implies
$|\xi|\simeq T_*^{-\eta}$ and hence
\begin{align*}
T_{\mathrm{fr}}-T_*\lesssim |\xi|^{-1/\eta}.
\end{align*}
The transition contribution is therefore also bounded by
$\lesssim |\xi|^{-N-\frac1\eta}$. Consequently, for $1\le N\le d+1$,
\begin{align}\label{sf-A-der}
|\nabla_\xi^N A_{\mathrm{sf}}^{\mathrm{pre}}(t,\xi)|
\lesssim |\xi|^{-N-\frac1\eta}.
\end{align}

Let
\begin{align*}
G_{\mathrm{sf}}^{\mathrm{pre}}(t,\xi)
:=A_{\mathrm{sf}}^{\mathrm{pre}}(t,\xi)(2\pi|\xi|)^{2\alpha}.
\end{align*}
Exactly as in the corrected proof of Lemma~\ref{bd-m}, \eqref{sf-A-der} yields
\begin{align*}
|\partial_\xi^\beta G_{\mathrm{sf}}^{\mathrm{pre}}(t,\xi)|
\lesssim G_{\mathrm{sf}}^{\mathrm{pre}}(t,\xi)|\xi|^{-|\beta|}
+|\xi|^{2\alpha-\frac1\eta-|\beta|}, \qquad 1\le |\beta|\le d+1.
\end{align*}
Since $2\alpha-\frac1\eta<0$, it follows that
\begin{align}\label{sf-msf-transition-mihlin}
\left|\nabla_\xi^N e^{-\nu G_{\mathrm{sf}}^{\mathrm{pre}}(t,\xi)}\right|
\lesssim |\xi|^{-N},
\qquad 1\le N\le d+1.
\end{align}

If $T_*<t<T_{\mathrm{fr}}$, combining \eqref{sf-msf-transition-mihlin} with the standard derivative bounds
for $\varphi_{\Lambda^{\mathrm{sf}}(t)}^2$ and arguing as in the proof of \eqref{high-deri-b}-\eqref{high-shell-bound} give 
\begin{align*}
\|b^{\mathrm{sf}}_{n,t}\|_{W^{d+1,1}(\mathbb R^d)}\lesssim1.
\end{align*}

It remains to consider $t\ge T_{\mathrm{fr}}$. Then
\begin{align*}
b^{\mathrm{sf}}_{n,t}(\zeta)
=e^{-\nu G_{\mathrm{sf}}^{\mathrm{pre}}(T_{\mathrm{fr}},2^n\zeta)}
H_{s,\lambda}(2^n\zeta)\Psi(\zeta),
\end{align*}
where
\begin{align*}
H_{s,\lambda}(\xi)
:=e^{-\nu s(2\pi|\xi|)^{2\alpha}\varphi_\lambda^2(\xi)}\varphi_\lambda^2(\xi), \qquad
s=t-T_{\mathrm{fr}},\quad \lambda=\Lambda_{\mathrm{fr}}.
\end{align*}
We use the following fixed-scale dyadic bound: for every multi-index $\beta$ with
$|\beta|\le d+1$,
\begin{align}\label{fixed-scale-heat-mihlin}
\|\partial_\zeta^\beta\bigl(H_{s,\lambda}(2^n\zeta)\Psi(\zeta)\bigr)\|_{L^\infty_\zeta}
\lesssim_{\beta,d}1,
\end{align}
uniformly in $s\ge0$, $\lambda\ge2$, and $n\ge1$. Indeed, let $R:=2^n/\lambda$ and
$M:=d+1$. If $R$ lies outside a fixed compact sub-interval of $(0,\infty)$ (say, $[\frac18,16]$), then
$\varphi(R\zeta)$ is identically $0$ or $1$ on $\supp\Psi$, and the estimate follows directly from
$r^j e^{-cr}\lesssim_j1$. Otherwise $R$ ranges in a compact subset of $(0,\infty)$. On $\supp\Psi$,
$|\zeta|$ is bounded above and below. With
\[
h_R(\zeta):=\varphi^2(R\zeta), \qquad
a_R(\zeta):=|\zeta|^{2\alpha}h_R(\zeta),
\]
we use the flat-cutoff estimate \eqref{flat-est} for $\varphi$ to obtain, for $1\le|\gamma|\le M$,
\[
|\partial_\zeta^\gamma h_R(\zeta)|+|\partial_\zeta^\gamma a_R(\zeta)|
\le C_M a_R^{1-\varepsilon_M}(\zeta),
\]
uniformly in $R$. Expanding derivatives of $h_Re^{-r a_R}$ by the product rule and Faa di Bruno formula,
each term is bounded by either
\[
C r^j a_R^{1+j(1-\varepsilon_M)}e^{-r a_R}
\quad\text{or}\quad
C r^j a_R^{(j+1)(1-\varepsilon_M)}e^{-r a_R},
\]
with $0\le j\le M$. Since $\varepsilon_M<(M+1)^{-1}$ and $r^j e^{-cr}\lesssim_j1$, these expressions are uniformly bounded for
$r\ge0$. Taking $r=\nu s\,2^{2\alpha n}$ implies \eqref{fixed-scale-heat-mihlin}.

Combining \eqref{sf-msf-transition-mihlin}, \eqref{fixed-scale-heat-mihlin}, and using the product rule, we obtain
\begin{align*}
\|b^{\mathrm{sf}}_{n,t}\|_{W^{d+1,1}(\mathbb R^d)}\lesssim1
\end{align*}
uniformly in $n$ and $t$. The Poisson summation and Young inequality argument from
Lemma~\ref{bd-m} prove \eqref{sf-Lp-mult}.
\end{proof}

\begin{lemma}
\label{lem-sf-short-tail}
For every $T_{\mathrm{fr}}\in (T_*,T_*+\varepsilon_0]$, let
$\Lambda^{\mathrm{sf}}$ be defined by \eqref{def-Lambda-sf} and $\mu_{\mathrm{fr}}$  given by \eqref{sf-close-delta2}.  Then
\begin{align}
\|\mathring R_0^{\mathrm{sf}}\|_{C_{[T_1,+\infty)}L^1_x}
\le \frac18 c_*\mu_{\mathrm{fr}}\delta_3.
\label{sf-tail-d3}
\end{align}
Moreover,
\begin{align}
\mu_{\mathrm{fr}}\to 1 \qquad\text{as} \qquad T_{\mathrm{fr}} \to  T_*.
\label{sf-mu-to-one}
\end{align}
\end{lemma}

\begin{proof}

On $[T_1,T_*]$, Lemma~\ref{lem-sf-agree} and \eqref{est-r0-0} with $q=1$ yield
\[
\|\mathring R_0^{\mathrm{sf}}\|_{C_{[T_1,T_*]}L^1_x}
=\|\mathring R_0\|_{C_{[T_1,T_*]}L^1_x}
\le \|\mathring R_0\|_{C_{[0,T_*]}L^1_x}
\le \frac18 c_*\delta_3.
\]

Now let $t\ge T_*$. Since
$\Lambda^{\mathrm{sf}}(t)\in[\Lambda(T_{\mathrm{fr}}),\Lambda(T_*)]$ and
\eqref{sf-same-shell} holds on the whole stabilized region, the dyadic index in the
proof of Proposition~\ref{Prop-R0} is constant and equals $n_1$. We use
Lemma~\ref{lem-sf-multiplier} in place of Lemma~\ref{bd-m}, the same Bernstein summation as in
\eqref{est-r0-linfty}--\eqref{asp-reg}, and in the Sobolev case the analogue
\eqref{est-r0-lp}--\eqref{asp-reg-3}, to obtain
\begin{align*}
\|\mathring R_0^{\mathrm{sf}}(t)\|_{L^1_x}
\le C_R\bigl(\Lambda^{\mathrm{sf}}(t)\bigr)^{-\beta_1}
\le C_R\Lambda(T_{\mathrm{fr}})^{-\beta_1},
\end{align*}
where $\beta_1$ is given by \eqref{def-beta1-sf} and $C_R$ was chosen large enough in advance.

Since $\Lambda(T_*)=T_*^{-\eta} 
= \overline{T}_1^{-\eta}$ and $\beta_1=k_1/\eta$,  
by \eqref{def-delta}, this becomes
\[
\|\mathring R_0^{\mathrm{sf}}(t)\|_{L^1_x}
\le C_R\Big(\frac{\Lambda(T_{\mathrm{fr}})}{\Lambda(T_*)}\Big)^{-\beta_1}
\Lambda(T_*)^{-\beta_1}
=\frac18 c_*\mu_{\mathrm{fr}}\delta_3.
\]
Taking the supremum over $t\ge T_*$ and combining with the previous estimate on
$[T_1,T_*]$ prove \eqref{sf-tail-d3}. Since $\Lambda$ is continuous and strictly
decreasing near $T_*$, \eqref{sf-mu-to-one} is immediate from
\eqref{sf-close-delta2}.
\end{proof}

Now we prove an estimate for smooth stabilized energy profile.

\begin{lemma} \label{prop-sf-closure}

Assume the hypotheses of Theorem~\ref{Thm-Non-euler}.  
Let 
\begin{align}
\delta_2^{\mathrm{sf}} :=\mu_{\mathrm{fr}}\delta_3 
\label{sf-close-delta2-2}
\end{align}
with $\mu_{\mathrm{fr}}$ as in \eqref{sf-close-delta2}. 
Then, 
\begin{align}
E_0^{\mathrm{sf}}(\ot_2)-E_0^{\mathrm{sf}}(\ut_1)
\ge 4\bigl(\delta_2^{\mathrm{sf}}-\delta_3\bigr). 
\label{sf-close-bridge}
\end{align}

\end{lemma}

\begin{proof}
Using \eqref{sf-close-delta2},
\[
\ut_1=T_*-2\ell_0\ge T_*-2\widetilde\ell_0,
\]
and the monotonicity of $E_0$, we infer that
\[
B_*=E_0(\ot_2)-E_0(\ut_1) \ge E_0(\ot_2)-E_0(T_*-2\widetilde\ell_0) =B_0
\ge 4(\mu_{\mathrm{fr}}-1)\delta_3 =4\bigl(\delta_2^{\mathrm{sf}}-\delta_3\bigr),
\]
which implies \eqref{sf-close-bridge}.
\end{proof}

Now we can finish the proof of the main result Theorem \ref{Thm-Non-euler}. 

\medskip 

\noindent
{\bf Proof of Theorem~\ref{Thm-Non-euler}.} 
We only need to verify that assumptions (i)--(iii) of Theorem~\ref{thm-sf-conditional} hold under the present choice of $\Lambda^{\mathrm{sf}}$ and $\delta_2^{\mathrm{sf}}$. Assumptions (i) and (ii) follow directly from Lemma~\ref{lem-sf-closure-bg}, and \eqref{sf-bridge} is a consequence of Lemma~\ref{prop-sf-closure}. 

It remains to verify \eqref{sf-delta-scale}--\eqref{sf-floor} in assumption (iii).
For this purpose, 
we note that 
by our choice of $\theta$, 
\begin{align}
\frac98\delta_3<\delta_1.
\label{sf-delta1-prep}
\end{align}
Moreover, by \eqref{sf-close-delta2},
\begin{align}
\frac34\mu_{\mathrm{fr}}\delta_3\le \delta_1.
\label{sf-mu-choice-2}
\end{align}
Lemma~\ref{lem-sf-short-tail} then yields
\[
\frac{8}{c_*}\|\mathring R_0^{\mathrm{sf}}\|_{C_{[T_1,+\infty)}L^1_x}
\le \mu_{\mathrm{fr}}\delta_3 =\delta_2^{\mathrm{sf}},
\]
which is precisely \eqref{sf-floor}. Moreover, \eqref{sf-close-delta2} yields \eqref{sf-delta-scale}, and \eqref{sf-mu-choice-2} implies \eqref{sf-delta1}.

Therefore, 
all assumptions (i)--(iii) of Theorem~\ref{thm-sf-conditional} are verified.  By virtue of Theorem~\ref{thm-sf-conditional}, 
we infer that 
there exist infinitely many distinct global dissipative weak Euler solutions with initial data $v_0$.  
\hfill $\square$

\medskip 
At the end of this section, we consider a natural structural subclass of initial data, namely, parallel shear flows. 

\begin{lemma}[Parallel shear flows are radially stable]
\label{lem-shear-radial-stable}
Let
\[
v_0(x)=g(P_{e^\perp}x)e
\]
be a parallel shear flow on $\mathbb T^d$ in the sense of \eqref{parasf}, where $g\in L^2_x$. Then
$v_0$ is stable under radial multipliers.
\end{lemma}

\begin{proof}
Let $\rho$ be any bounded radial symbol and
\[
\mathcal T_{\rho}v_0(x)
=\sum_{\xi\in \mathbb Z^d}\rho(|\xi|)\widehat v_0(\xi)e^{i2\pi x\cdot \xi}.
\]
Since $v_0$ has Fourier support in the hyperplane $e^\perp$, the same is true for
$\mathcal T_{\rho}v_0$. Hence there exists a scalar function $g_{\rho}$ on
$e^\perp$ such that
\[
\mathcal T_{\rho}v_0(x)=g_{\rho}(P_{e^\perp}x)e.
\]
Exactly as in \eqref{para-flow},
\[
(\mathcal T_{\rho}v_0(x)\cdot \nabla)\mathcal T_{\rho}v_0(x)
=(g_{\rho}(P_{e^\perp}x)e\cdot \nabla g_{\rho}(P_{e^\perp}x))e=0,
\]
so $\mathcal T_{\rho}v_0$ is a steady Euler solution.
\end{proof}

As a consequence of Theorem~\ref{Thm-Non-euler} and Lemma~\ref{lem-shear-radial-stable}, 
we have the following corollary 
for global dissipative solutions with 
parallel shear data.

\begin{corollary}[Unconditional global dissipative solutions for parallel shear data]
\label{cor-sf-parallel}
Let $d\ge2$. Let either
\begin{align*}
 v_0\in C^{1-\kappa}(\T^d),\qquad 0<\kappa<1,
\end{align*}
with $v_0$ satisfying {\rm (H1)} and {\rm (H2)}, or let
\begin{align*}
v_0\in W^{1,p}(\T^d),\qquad 2<p<\infty,
\end{align*}
with $v_0$ satisfying {\rm (H1)'} and {\rm (H2)'}. Suppose in addition
that $v_0$ is a parallel shear flow on $\mathbb T^d$. Then there exist
infinitely many distinct global dissipative weak Euler solutions with initial
data $v_0$.
\end{corollary}

\begin{proof}
By Lemma~\ref{lem-shear-radial-stable}, parallel shear data are
stable under radial multipliers. Hence all hypotheses of
Theorem~\ref{Thm-Non-euler} are satisfied, and the conclusion follows.
\end{proof}

\begin{remark}
\label{rem-sf-closure-scope}
It is obvious that the 
radial-multiplier stability assumption of Theorem~\ref{Thm-Non-euler}
implies 
the abstract hypothesis
\[
\mathcal S(t)v_0\ \text{is a steady Euler solution for every }t\ge 0. 
\] 
For the parallel shear examples
from SubSection~\ref{Subsec-Com-Main}, this extra property is automatic, and the same radiality
argument as in Appendix~\ref{App-TC-Flow} gives the corresponding Euclidean
statement for Taylor--Couette flows and the truncated power-law vortex.
\end{remark}

\section{Appendix: Taylor--Couette flows and the truncated power-law vortex} 
\label{App-TC-Flow}

This appendix contains the proof of Lemma~\ref{lem-tcflow} for Taylor-Couette flows in the Euclidean space $\mathbb{R}^2$, and the proof of Lemma~\ref{lem-m-ini} for the truncated power-law vortex on $\mathbb{R}^2$. 

\begin{proof}[Proof of Lemma~\ref{lem-tcflow}]
For $t=0$, the claim is trivial since $\mathcal{S}(0)=\mathrm{Id}$. Hence, let
$t>0$. Let $m(t,\xi)$ be a radial multiplier. 
Since $\mathcal{S}(t)$ is a Fourier multiplier operator, it is a convolution operator. Thus, 
\begin{align}\label{sv0-o}
    \mathcal{S}(t)v_0= \int_{\R^2} K_t(x-y)v_0(y)\d y, 
\end{align}
where 
$$K_t(x):=\int_{\mathbb{R}^2} m(t,\xi) e^{i 2\pi x \cdot \xi}\d \xi.$$ 

We note that $K_t$ is a radial function. Hence, there exists a scalar function
$\kappa_t:[0,+\infty)\to \R$ such that $K_t(x)=\kappa_t(|x|)$.
Actually, taking any rotation matrix 
$R\in O(2)$ we have
\begin{align*}
    K_t(Rx)=\int_{\mathbb{R}^2} m(t,\xi) e^{i2\pi Rx \cdot \xi}\d \xi= \int_{\mathbb{R}^2} m(t,\xi) e^{i2\pi x \cdot R^{-1}\xi}\d \xi= \int_{\mathbb{R}^2} m(t,\eta) e^{i2\pi x \cdot \eta}\d \eta= K_t(x),
\end{align*}
where we used the identity 
$Rx \cdot \xi =x\cdot R^T\xi= x \cdot R^{-1}\xi$ and $\eta=R^{-1}\xi$. Hence, \eqref{sv0-o} is equal to 
\begin{align}
    \mathcal{S}(t)v_0= \int_{\R^2} \kappa_t(|x-y|)v_0(y)\d y. 
\end{align}

Let $x=(r\cos\theta,r\sin\theta)$ and $y=(r'\cos\theta',r'\sin\theta')$ in polar coordinates. 
Since $v_0$ is a Taylor-Couette flow, there exists a scalar function $g$ such that 
$v_0(y)=g(r')e_{\theta'}$, where 
$e_{\theta'}=(-\sin\theta',\cos\theta')$, 
$e_\theta=(-\sin\theta,\cos\theta)$, 
and 
$e_r=(\cos\theta,\sin \theta)$. 
Since 
\begin{align*}
    e_{\theta'}\cdot e_r=\sin(\theta-\theta')=\sin\phi, \qquad
    e_{\theta'}\cdot e_\theta=\cos(\theta-\theta')=\cos\phi,
\end{align*}
with $\phi=\theta-\theta'$, it follows that 
\begin{align}
    \mathcal{S}(t)v_0= (\mathcal{S}(t)v_0)_re_r+(\mathcal{S}(t)v_0)_\theta e_\theta,
\end{align}
where 
\begin{align*}
    (\mathcal{S}(t)v_0)_r= \int_0^{\infty}\int_0^{2\pi} \kappa_t(\sqrt{r^2+(r')^2-2rr'\cos\phi}) g(r')\sin\phi \, r'\d\theta'\d r'
\end{align*}
and
\begin{align*}
    (\mathcal{S}(t)v_0)_\theta= \int_0^{\infty}\int_0^{2\pi} \kappa_t(\sqrt{r^2+(r')^2-2rr'\cos\phi}) g(r')\cos\phi \, r'\d\theta'\d r',
\end{align*}
with $\phi=\theta-\theta'$.

Note that by the change of variables $\phi=\theta-\theta'$ and periodicity, 
\begin{align}\label{int-svrt}
    \int_0^{2\pi} \kappa_t(\sqrt{r^2+(r')^2-2rr'\cos\phi})\sin\phi \, \d\theta'&=\int_{\theta-2\pi}^{\theta} \kappa_t(\sqrt{r^2+(r')^2-2rr'\cos\phi})\sin\phi \, \d\phi\notag\\
    &= \int_{-\pi}^{\pi} \kappa_t(\sqrt{r^2+(r')^2-2rr'\cos\phi})\sin\phi \, \d\phi=0,
\end{align}
which yields that $(\mathcal{S}(t)v_0)_r\equiv0$. 

Similarly to \eqref{int-svrt}, we have
\begin{align*}
    (\mathcal{S}(t)v_0)_\theta&= \int_0^{\infty}g(r')r'\int_0^{2\pi} \kappa_t(\sqrt{r^2+(r')^2-2rr'\cos\phi}) \cos\phi \d\theta'\d r'\notag\\
    &= \int_0^{\infty}g(r')r' f_t(r,r')\d r'=:G_t(r),
\end{align*}
where $f_t(r,r'):=\int_0^{2\pi} \kappa_t(\sqrt{r^2+(r')^2-2rr'\cos\phi}) \cos\phi \d\phi $ 
is independent of $\theta$ and $\theta'$.

Thus, 
we conclude that 
$\mathcal{S}(t)v_0= G_t(r) e_\theta$, 
which yields 
that 
$\mathcal{S}(t)v_0$ is still a Taylor-Couette flow.
\end{proof}

\begin{proof}[Proof of Lemma~\ref{lem-m-ini}]  
By abuse of notation, we still write $\Delta_n$ for the Euclidean Littlewood-Paley projection on $\R^2$, namely,
\begin{align}
    \widehat{\Delta_n g}(\xi)=\chi_n(\xi)\widehat g(\xi) , \quad \xi\in\R^2,
\end{align}
where, for $n\geq 1$, $\chi_n(\xi)=\chi(2^{-n}\xi)$ with $\chi$ a smooth radial
nonnegative Littlewood-Paley cutoff supported in an annulus away from the origin
and strictly positive on some open annulus. Applying this to $f$, we get
\begin{align}
    \Delta_n f(x)=\check{\chi}_n*f(x), \quad x\in\R^2,
\end{align}
where $\check{\chi}_n$ denotes the inverse Fourier transform of $\chi_n$ on $\R^2$.  
Let 
$z:=2^{n} y$ and $\wt x:= 2^n x$. 
Since $\check{\chi}_n(y)=2^{2n}\check{\chi}(2^{n}y)$, 
we get
\begin{align}\label{est-r1-1}
     \Delta_n f(x)=2^{-n(1-\alpha)} \int_{\mathbb{R}^{2}} \check{\chi}(\tilde{x}-z) \psi(2^{-n}z)|z|^{-\alpha}z^\perp \d z 
&=:2^{-n(1-\alpha)} F_n(\tilde{x}).
\end{align}
Let
\begin{align*}
f_*(z):=|z|^{-\alpha}z^\perp, \qquad F_*(\tilde x):=(\check\chi*f_*)(\tilde x).
\end{align*}

We divide the space into two regimes  $\R^2=\mathcal{O}_1\cup \mathcal{O}_2$ 
with 
$\mathcal{O}_1=\{x: |x|\leq  2^{1-n}\}$ and $\mathcal{O}_2=\R^2\setminus \mathcal{O}_1$.

\medskip 
{\bf $(i)$ The case $x\in \mathcal{O}_1$:} 
In the case where $x\in \mathcal{O}_1$, 
we have $f(x)=|x|^{-\alpha}x^{\perp}$.
In order to estimate the $L^\infty$-norm of $\Delta_n f$, we first
show that $\sup_{|\wt x|\leq 2}|F_n(\wt x)|$ is uniformly bounded from below and above.  

Regarding the upper bound of $F_n$ on $|\wt x|\leq 2$, we have
\begin{align}\label{est-F}
    F_n(\tilde{x})=\int_{|\tilde{x}-z|\geq4} \check{\chi}(\tilde{x}-z) \psi(2^{-n}z)|z|^{-\alpha}z^\perp \d z+\int_{|\tilde{x}-z|<4} \check{\chi}(\tilde{x}-z) \psi(2^{-n}z)|z|^{-\alpha}z^\perp \d z.
\end{align}
For the first term on the right-hand side of \eqref{est-F}, since $\check{\chi}$ is a Schwartz function, we deduce that 
\begin{align*}
   |\int_{|\tilde{x}-z|\geq4} \check{\chi}(\tilde{x}-z) \psi(2^{-n}z)|z|^{-\alpha}z^\perp \d z| &\lesssim \int_{|\tilde{x}-z|\geq4} | \tilde{x}-z|^{-3} |z|^{1-\alpha} \d z\lesssim \int_{|z|\geq2}  |z|^{-2-\alpha} \d z
    \leq C
\end{align*}
for some universal constant $C$. 
For the second term on the right-hand side of \eqref{est-F}, we get
\begin{align}
    |\int_{|\tilde{x}-z|<4} \check{\chi}(\tilde{x}-z) \psi(2^{-n}z)|z|^{-\alpha}z^\perp \d z| &\lesssim \int_{|\tilde{x}-z|< 4} |\check{\chi}(\tilde{x}-z)| |z|^{1-\alpha} \d z
    \lesssim \int_{|z|< 6} |z|^{1-\alpha}  \d z\leq C,
\end{align}
where $C$ is a universal constant.  
Therefore,
\begin{align*}
    \sup_{|\tilde x|\leq 2}|F_n(\tilde x)|\leq C.
\end{align*}

Next, we prove a uniform lower bound. Since
\begin{align*}
    \sup_{|\tilde x|\leq 2}|\check\chi(\tilde x-z)|\,|z|^{1-\alpha}\in L^1_z(\R^2),
\end{align*}
the dominated convergence theorem yields
\begin{align}\label{Fn-conv}
    \sup_{|\tilde x|\leq 2}|F_n(\tilde x)-F_*(\tilde x)|\to 0
    \qquad\text{as }n\to \infty.
\end{align}
Moreover,
\begin{align*}
    f_*=\nabla^\perp\Big(\frac{1}{2-\alpha}|x|^{2-\alpha}\Big)
\end{align*}
in the sense of tempered distributions. Hence,
\begin{align*}
    \widehat{f_*}(\xi)=c_\alpha i\xi^\perp |\xi|^{\alpha-4},\qquad \xi\neq 0,
\end{align*}
for some nonzero constant $c_\alpha$. Therefore,
\begin{align*}
    \widehat{F_*}(\xi)=\chi(|\xi|)\widehat{f_*}(\xi)\not\equiv 0.
\end{align*}
Since $\widehat{F_*}\in C_c^\infty(\R^2)$, the function $F_*$ is real analytic. If $F_*$ vanished on $\{|\tilde x|\leq 2\}$, then by the identity theorem for real analytic functions we would have $F_*\equiv 0$, a contradiction. Thus,
\begin{align*}
    \sup_{|\tilde x|\leq 2}|F_*(\tilde x)|>0.
\end{align*}
Combining this with \eqref{Fn-conv}, we find $n_0\in \mathbb{N}$ and $c_0>0$ such that
\begin{align*}
    \sup_{|\tilde x|\leq 2}|F_n(\tilde x)|\geq c_0,\qquad n\geq n_0.
\end{align*}
For the finitely many indices $1\leq n<n_0$, assume by contradiction that $\sup_{|\tilde x|\leq 2}|F_n(\tilde x)|=0$. Then, by \eqref{est-r1-1},
\begin{align}\label{dnf0}
    \Delta_nf(x)=0\quad \text{for all}\quad |x|\leq 2^{1-n}.
\end{align}
Since $f\in L^1_c(\R^2)$, the Fourier transform $\widehat f$ extends to an entire function. Thus $\widehat{\Delta_nf}=\chi_n \widehat f\in C_c^\infty(\R^2)$ and $\Delta_n f$ is real analytic. In view of \eqref{dnf0} and \cite[Corollary 1.2.6]{KP02}, we infer that $\Delta_n f\equiv 0$ on $\R^2$, namely, $\chi_n\widehat f\equiv 0$. Because $\chi_n$ is strictly positive on a nonempty open annulus, this forces $\widehat f$ to vanish on an open set and hence everywhere, contradicting \eqref{mi}. Therefore, after possibly decreasing $c_0$, we obtain
\begin{align*}
    c_0\leq \sup_{|\tilde x|\leq 2}|F_n(\tilde x)|\leq C,\qquad n\geq 1.
\end{align*}
Putting them together we obtain 
\begin{align}\label{est-inf--1}
& \|\Delta_n f\|_{L^{\infty}(\mathcal{O}_1)} = 2^{-n(1-\alpha)} \sup _{|\tilde{x}|\leq 2}|F_n(\tilde{x})| \simeq 2^{-n(1-\alpha)}.
\end{align}
	
Moreover, for $p\in [1,\infty)$, using \eqref{est-r1-1}, we get
\begin{align}
\int_{\mathcal{O}_1} 
|\Delta_n f |^{p} \d x= 2^{-n p(1-\alpha)} \int_{|\tilde{x}| \leq 2}|F_n(\tilde{x})|^p \cdot 2^{-2n} \d \tilde{x}  =2^{-n(p(1-\alpha)+2)} \int_{|\tilde{x}| \leq 2}|F_n(\tilde{x})|^p \d \tilde{x}. 
\end{align}
By \eqref{Fn-conv}, the integral above converges to $\int_{|\tilde x|\leq 2}|F_*(\tilde x)|^p\d \tilde x>0$ as $n\to\infty$. For the finitely many indices $1\leq n<n_0$, the same contradiction argument as above shows that $\int_{|\tilde x|\leq 2}|F_n(\tilde x)|^p\d \tilde x>0$. Hence, after adjusting the constants if necessary, 
\begin{align}\label{est-inf-lp-1}
 \|\Delta_n f\|_{L^p(\mathcal{O}_1)}\simeq 2^{-n\left(1-\alpha+\frac{2}{p}\right)}. 
\end{align}

\medskip 
{\bf $(ii)$ The case $x\in \mathcal{O}_2$:} 
Next we consider the case where $x\in \mathcal{O}_2$. Since $|x|>2^{1-n}$, we have
$x\neq 0$, and $f$ is smooth in a neighborhood of $x$. For $|y|\leq \frac12|x|$,
Taylor's expansion around $x$ yields
\[
f(x-y)=\sum_{k=0}^M \frac{(-1)^k}{k!} \nabla^k f(x) y^{\otimes k} +R_M(x, y), 
\]
where $M\geq 1$,
\[
R_M(x, y):=\frac{(-1)^{M+1}}{M!} \int_0^1(1-t)^M \nabla^{M+1} f(x-t y) y^{\otimes(M+1)} d t,
\]
and $\nabla^k f(x) y^{\otimes k}=  \sum_{i_1,i_2,\dots,i_k=1}^2 \p_{i_1}\p_{i_2}\dots\p_{i_k}f(x)y_{i_1}y_{i_2}\dots y_{i_k}$.
For general $y\in \R^2$, we define
\begin{align*}
    R_M(x,y):=f(x-y)-\sum_{k=0}^M \frac{(-1)^k}{k!} \nabla^k f(x) y^{\otimes k}.
\end{align*}

Note that $\nabla^{k}{\chi}_n(0)\equiv 0$ for any $k=0,1,\dots,M$. 
Then, by Fourier inversion, 
\begin{align*}
    \int_{\R^2} \check{\chi}_n(y) \d y   =\chi_n(0)=0,
\end{align*}
and, similarly, 
for any multi-index $\beta$ with $|\beta|\le M$, there exists a universal constant $c_\beta$ such that 
\begin{align*}
   \int_{\R^2} y^\beta\check{\chi}_n(y) \d y 
   = c_\beta \partial^\beta\chi_n(0)=0.
\end{align*} 
It follows that 
\begin{align}  \label{Deltanf-RM}
\Delta_n f(x)=\int_{\R^2} \check{\chi}_n(y) f(x-y)\d y=\int_{\R^2} \check{\chi}_n(y) R_M(x, y) \d y.
\end{align}

First note that 
for $|x|\in [2^{-n},1]$, 
$f(x)= |x|^{-\alpha}x^{\perp}$, 
and thus, 
by Euler's Homogeneous Function Theorem (see, e.g., \cite[p287]{Apostol}) we get
\begin{align}\label{est-nkf-1}
    |\nabla^k f(x) |\lesssim |x|^{1-\alpha-k}, \quad |x|\in [2^{-n},1].
\end{align}
For $|x|\in (1,2]$, we have
\begin{align*}
    |\nabla^k f(x) |\lesssim \sum_{i=0}^k |\nabla^i \psi| |\nabla^{k-i} (|x|^{-\alpha}x^{\perp})|
    \lesssim \sum_{i=0}^k|x|^{1-\alpha-k+i} \lesssim |x|^{1-\alpha-k}, \quad |x|\in (1,2],
\end{align*}
which along with \eqref{est-nkf-1} yields 
\begin{align}\label{est-nkf}
    |\nabla^k f(x) |\lesssim |x|^{1-\alpha-k}, \quad |x|\geq 2^{-n}.
\end{align}

Now, for $x\in \mathcal{O}_2$ 
we split \eqref{Deltanf-RM} into 
$$
\Delta_n f(x)=2^{2n}\int_{|y| \leq \frac12|x|}\check{\chi} (2^{n}y) R_M(x, y) \d y+2^{2n}\int_{|y|> \frac12|x|}\check{\chi}(2^{n} y) R_M(x, y) \d y
=:J_1+J_2.
$$
For the first term $J_1$,
since $|y| \leq \frac12|x|$, 
$|x-ty|\geq \frac12|x|> 2^{-n}$. 
Using the integral remainder formula above and \eqref{est-nkf} we get 
$$
 |R_M(x, y) |\lesssim \sup_{t\in [0,1]}|\nabla^{M+1} f(x-ty) | |y|^{M+1}\lesssim  \sup_{t\in [0,1]}|x-ty|^{-\alpha-M}|y|^{M+1}\lesssim |x|^{-\alpha-M}|y|^{M+1},
$$
which yields that 
\begin{align}\label{p-est-j1}
    |J_1| &\lesssim|x|^{-\alpha-M} 2^{2 n} \int_{|y| \leq \frac12|x|} |\check{\chi} (2^n y ) ||y|^{M+1} d y\notag\\
    &\lesssim 2^{-n(M+1)}|x|^{-\alpha-M}   \int_{\R^2} |\check{\chi} (z ) ||z|^{M+1} d z\notag\\
    &\lesssim 2^{-n(M+1)}|x|^{-\alpha-M}.
\end{align}

Regarding the second term $J_2$, since $|y|>\frac12|x|\geq 2^{-n}$, for any $N\geq 1$ we have
\begin{align}\label{est-j2-chin} 
    |\check{\chi}_n(y) |=2^{2 n} |\check{\chi}\left(2^n y\right) | \lesssim 2^{2 n} (1+2^n|y|)^{-N}.
\end{align}
where the implicit constant is independent of $n$. Moreover, using the algebraic definition of
$R_M$ and \eqref{est-nkf} we get
\begin{align}\label{est-j2-rm}
    |R_M(x, y)| &\lesssim|f(x-y)|+|\sum_{k=0}^M \frac{(-1)^k}{k!} \nabla^k f(x) y^{\otimes k}|\notag \\
    &\lesssim |x-y|^{1-\alpha}+ \sum_{k=0}^M |x|^{1-\alpha-k}|y|^k  \notag \\
&\lesssim (|x|+|y|)^{1-\alpha}+ |x|^{1-\alpha}(1+(\frac{|y|}{|x|})^M).
\end{align}
Hence, combining  \eqref{est-j2-chin} and \eqref{est-j2-rm} together we deduce that
\begin{align}\label{aest-j2-1}
|J_2| &\lesssim 2^{2n} \int_{|y|>\frac12|x|}\left(1+2^n|y|\right)^{-N}\left( (|x|+|y|)^{1-\alpha}+|x|^{1-\alpha} (1+ (\frac{|y|}{|x|} )^M )\right) \d y\notag\\
& \lesssim \int_{|z|> 2^{(n-1)}|x|}(1+|z|)^{-N}\left((|x|+2^{-n}|z| )^{1-\alpha}+|x|^{1-\alpha} (1+2^{-n M} \frac{|z|^M}{|x|^M})\right) \d z,
\end{align}
where $z=2^n y$. 

Concerning the right-hand side of \eqref{aest-j2-1}, 
we have 
$$
J_{21}:=\int_{|z|>2^{(n-1)}|x|}(1+|z|)^{-N}\left(|x|+2^{-n}|z|\right)^{1-\alpha} \d z\lesssim 2^{-n(1-\alpha)} \int_{|z|>2^{(n-1)}|x|}(1+|z|)^{-N}|z|^{1-\alpha} \d z.
$$
Choosing $N\geq 4$ we get 
\begin{align}\label{p-est-j21}
    J_{21}\lesssim 2^{-n(1-\alpha)}\int_{|z|>2^{(n-1)}|x|}|z|^{1-\alpha-N} \d z\lesssim 2^{-n(N-2)}|x|^{-(N-3+\alpha)}.
\end{align}
Moreover, we have
\begin{align}\label{p-est-j22}
    J_{22}:= \int_{|z|> 2^{(n-1)}|x|}(1+|z|)^{-N} |x|^{1-\alpha}  \d z\lesssim 2^{-n(N-2)}|x|^{-(N-3+\alpha)}.
\end{align}
Finally, we estimate
\begin{align*}
    J_{23}:= \int_{|z|> 2^{(n-1)}|x|}(1+|z|)^{-N} |x|^{1-\alpha}  2^{-n M} \frac{|z|^M}{|x|^M}\d z=2^{-n M}|x|^{1-\alpha-M}   \int_{|z|>2^{(n-1)}|x|}(1+|z|)^{-N}|z|^M \d z. 
\end{align*}
Choosing $N$ even larger such that $N>M+2$ we have
\begin{align}\label{p-est-j23}
    J_{23}\lesssim 2^{-n M}|x|^{1-\alpha-M}\left(2^n|x|\right)^{-(N-M-2)}=2^{-n(N-2)}|x|^{-(N-3+\alpha)}.
\end{align}
Hence, plugging \eqref{p-est-j21}, \eqref{p-est-j22} and \eqref{p-est-j23} into \eqref{aest-j2-1} we arrive at
\begin{align}\label{p-est-j2}
    J_{2}\lesssim 2^{-n(N-2)}|x|^{-(N-3+\alpha)}.
\end{align}

Thus, combining \eqref{p-est-j1} and \eqref{p-est-j2} together and choosing $N=M+3$ we get
\begin{align}
    |\Delta_n f(x)|\leq |J_1|+|J_2|\lesssim 2^{-n(M+1)} |x|^{-\alpha-M},  
    \ \ x\in \mathcal{O}_2. 
\end{align}
In particular, we have
\begin{align}\label{est-inf-2}
& \|\Delta_n f\|_{L^\infty(\mathcal{O}_2)}\lesssim 2^{-n(1-\alpha)}. 
\end{align}
Moreover, choosing $M$ large enough such that $p(\alpha+M)>2$, we have
\begin{align}
    \int_{\mathcal{O}_2} |\Delta_n f(x)|^p\d x\lesssim 2^{-np(M+1)}    \int_{\mathcal{O}_2} |x|^{-p(\alpha+M)}\d x \lesssim 2^{-n(2-p\alpha+p)},
\end{align}
which yields that 
\begin{align}\label{est-lp-2}
& \|\Delta_n f\|_{L^p(\mathcal{O}_2)}\lesssim 2^{-n(1-\alpha+\frac{2}{p})}. 
\end{align}

\medskip 
Therefore, 
combining \eqref{est-inf--1}, \eqref{est-inf-lp-1}, \eqref{est-inf-2} and \eqref{est-lp-2}, we arrive at
\begin{align}\label{est-inf-lp-end}
\|\Delta_n f\|_{L^\infty(\R^2)}\simeq 2^{-n(1-\alpha)}, \quad  \|\Delta_n f\|_{L^p(\R^2)}\simeq 2^{-n\left(1-\alpha+\frac{2}{p}\right)},\quad p\in [1,\infty), 
\end{align}
thereby finishing the proof of Lemma~\ref{lem-m-ini}.
\end{proof}

\medskip
\noindent{\bf  Acknowledgment.}  
This manuscript is based upon work supported by the National Science Foundation under Grant No. DMS-1928930, while D.\ Zhang was in residence at the Simons Laufer Mathematical Sciences Institute in Berkeley, California, during the fall semester of 2025. 
Z. Zeng is partially supported by the NSFC grants (No. 12501281) and the Natural Science Foundation of Jiangsu Province (No. SBK2024043113). D.\ Zhang is also grateful for the NSFC grants (No. 12625110, 12322108, 12271352) 
and Shanghai Frontiers Science Center of Modern Analysis.

\end{document}